\documentclass[11pt]{amsart}
\usepackage{fullpage,amsmath,amsthm,amsfonts,amssymb,graphicx,amscd,float,picins,hyperref,
setspace}
\usepackage{color, xcolor}
\usepackage{tikz-cd}
\usetikzlibrary{cd}
\usepackage{wrapfig}
\usepackage[margin=1in]{geometry}
\usepackage[font=small,labelfont=bf]{caption}
\usepackage{tikz}
\usepackage{wrapfig}
\usepackage{lmodern}
\usepackage{hyperref}

\definecolor{ivory}{rgb}{1.0, 1.0, 1.0}
\definecolor{carnelian}{rgb}{0.7, 0.11, 0.11}
	\definecolor{cinnamon}{rgb}{0.82, 0.41, 0.12}

\newtheorem{thm}{Theorem}[section]

\newtheorem{lem}[thm]{Lemma}

\newtheorem{cor}[thm]{Corollary}

\theoremstyle{definition}

\newtheorem{df}[thm]{Definition}

\newcommand{\G}{\Gamma}

\theoremstyle{definition}
\newtheorem{rk}[thm]{Remark}
\newtheorem{ex}[thm]{Example}
\newtheorem{nt}[thm]{Notation}
\newtheorem{obs}[thm]{Observation}

\newtheorem*{mainthmA}{Theorem A}
\newtheorem*{mainthmB}{Theorem B}
\newtheorem*{mainthmC}{Theorem C}

\usepackage{multicol}

\begin{document}

%\pagecolor{ivory}

\title{ \vspace{-0.0cm}  Latent symmetry of graphs and stretch factors in Out$(F_r)$}% of 
%Fully Irreducible  \\ Train Track Maps}
\author{\vspace{-0.0cm} Paige Hillen}
\address{\tt University of California - Santa Barbara, Department of Mathematics
  \newline \indent  {\url{https://sites.google.com/view/paigehillen}}, } \email{\tt paigehillen@ucsb.edu}

%\date{}

\maketitle

\begin{abstract} \vspace{-0.9cm} Every irreducible outer automorphism of the free group of rank $r$ is topologically represented by an irreducible train track map $f: \Gamma \rightarrow \Gamma$ for some graph $\G$ of rank $r$. Moreover, $f$ can always be written as a composition of \textquotedblleft folds" and a graph isomorphism. We give a lower bound on the stretch factor of an irreducible outer automorphism in terms of the number of folds of $f$ and the number of edges in $\Gamma$. In the case that $f$ is periodic on the vertex set of $\G$, we show a precise notion of the latent symmetry of  $\Gamma$ gives a lower bound on the number of folds required. We use this notion of latent symmetry to classify all possible irreducible single fold train track maps. \end{abstract}

\section{Introduction}\label{intro}

Let $F_r$ denote the free group of rank $r$ for $r \geq 2$, and Out$(F_r)$ the group of outer automorphisms of $F_r$. Given $\varphi \in $ Out$(F_r)$, the stretch factor of $\varphi$, is given by
$$\lambda(\varphi) : = \sup_{w \in F_r} \lim \sup ||\varphi^n(w)||^{1/n},$$
where $|| \cdot ||$ is the cyclically reduced word length. The stretch factor measures the asymptotic growth rate of words under repeated application of $\varphi$. Irreducible elements of Out$(F_r)$ have an \textit{irreducible train track representative}, that is a self homotopy equivalence of a graph of rank $r$, which induces $\varphi$ on the fundamental group and has certain desirable properties under iteration \cite{bh92}. The stretch factor of $\varphi$ appears as the leading eigenvalue of the transition matrix of such a train track representative, and hence is a \textit{weak Perron number}, that is, a real positive algebraic integer which is larger than or equal to its algebraic conjugates in modulus. 

%From their characterization as the leading eigenvalue of an irreducible non-negative integral matrix (remark 1.8 in \cite{bh92}), stretch factors of irreducible outer automorphisms are \textit{weak Perron numbers}. That is, real positive algebraic integers which are at least as large in modulus as all of their algebraic conjugates. %\footnote{These are actually called almost Perron numbers, but \cite{almostperron} showed these numbers coincide with weak Perron numbers, and I like this definition more.}. 
  Conversely, Thurston showed every weak Perron number is the stretch factor of some outer automorphism \cite{thurstonstretch}, \cite{thurstonexplained}. In Thurston's proof, he explicitly constructs an irreducible train track map with stretch factor equal to a given weak Perron number. The maps he constructs are all on a $(1,N)-$bipartite graph with $7$ edges between the single vertex set and each vertex in the $N$ vertex set. There is no control on $N$, and hence no control on the rank of the corresponding free group. It remains an interesting question which weak Perron numbers can occur as stretch factors in a fixed rank. 
%Thurston showed every weak Perron number is the stretch factor of some outer automorphism (\cite{thurstonstretch},\cite{thurstonexplained}). In his proof, he explicitly constructs an irreducible train track map with stretch factor equal to a given weak Perron number. Notably, the train track maps he constructs are all on the rose graph (the wedge of $r$ copies of $S^1$), and there is no control on the number of edges, and hence on the rank of the corresponding free group. 
%Here, we are interested in the set of stretch factors that can occur for infinite order irreducible elements of Out$(F_r)$ for a fixed rank $r$. 
In particular, we are concerned with finding the minimal such stretch factor. 

Progress has been made towards this question: \cite{AKRstretch} gives an upper and lower bound for this minimum in terms of the rank $r$, and \cite{wiggd1} finds the minimal stretch factor among fully irreducible elements of Out($F_3)$. %We restrict our attention to train track representatives which can be written as a composition of proper full folds $ f_1, \dots f_m$ and a graph automorphism (/homeomorphism) $h$. 
Intuitively, fewer folds in the fold decomposition of $f$ (\cite{s83}) should yield shorter word lengths of images of edges under $f$, and thus a smaller stretch factor. This is captured in the following result.

\vspace{0.25cm}
\noindent \hyperref[lower bound]{\textbf{Theorem A.}} \textit{ Suppose $f:\Gamma \rightarrow \Gamma$ is an irreducible homotopy equivalence self graph map with fold decomposition consisting of $m$ total folds. Let $n = |\mathcal{E}\Gamma|$, where $\mathcal{E}\G$ is the edge set of $\G$. Then
$$(m+1)^{\frac{1}{n}} \leq \lambda_f$$
where $\lambda_f$ is the largest eigenvalue of the transition matrix of $f$.}
\vspace{0.15cm}

\begin{rk}\label{finite list} When $f$ is an irreducible train track representative of $\varphi \in $ Out$(F_r)$, we have $\lambda_f = \lambda(\varphi)$. Hence, given a specific stretch factor $\lambda$  in some rank $r$, the above theorem gives a finite list of pairs (number of edges, number of folds) which could possibly correspond to an irreducible  train track map with stretch factor less than $\lambda$. 
\end{rk}

Out$(F_r)$ plays a similar role for graphs that the mapping class group plays for surfaces, with fully irreducible elements of Out$(F_r)$ corresponding to pseudo-Anosov elements of $\mathcal{MCG}(S)$. In the mapping class group setting, every stretch factor of a pseudo-Anosov is a \textit{bi-Perron} algebraic unit, but it is still unknown exactly which such units can occur. In 1991, Penner showed bounds on the minimal stretch factor in terms of the genus $g$ for closed surfaces  \cite{penner}:
%\vspace{-0.8cm}
 $$ (A)^{\frac{1}{g}} \leq \min\{\lambda: \text{ pseudo-Anosov }f:S_g \rightarrow S_g \text{ has stretch factor }\lambda\} \leq (B)^{\frac{1}{g}}$$ 
 for explicit constants $A$ and $B$. Since then, many have studied minimal stretch factors, including the case of surfaces with punctures or for certain subsets of $\mathcal{MCG}(S)$  (\cite{hamsong}, \cite{choham}, \cite{eriko}, \cite{FLMsmall}, \cite{liechtipenner}, \cite{lovingleast}, \cite{yazdi}). %tsai, braids ERWAN LANNEAU, JEAN-LUC THIFFEAULT). 
 In 2021, Pankau and Liechti used Thurston's construction of pseudo-Anosov homeomorphisms to show every bi-Perron unit $\lambda$ has a power which is a stretch factor of a pseudo-Anosov homeomorphism on a closed orientable surface of genus coarsely determined by the algebraic degree of $\lambda$ \cite{pankau2}. However, there is no control on how large of a power one needs to take. For  genus $g$ surfaces with $n >0$ punctures, $\pi_1(S_{g,n})$ is a free group, and hence elements of the mapping class group correspond to outer automorphisms of $F_{2g+n-1}$. Such outer automorphisms are called geometric. In a certain sense, outer automorphisms are generically not geometric, meaning they cannot be realized as a homeomorphism on a surface \cite{gersten}.

% In 2007, Ham and Song found precisely the minimal stretch factor which can occur for pseudo-Anosov homeomorphisms on the 6 punctured sphere \cite{hamsong}. 
%IA unified answer for the minimal stretch factor on any finite type surface remains unknown.

Remark \ref{finite list} suggests a computational strategy for finding minimal stretch factors in Out$(F_r)$. Knowing which rank $r$ graphs can possibly support an irreducible train track map with at most $m$ folds would reduce the computation involved in this procedure. 
%We are also interested in the set of stretch factors that can arise from irreducible train track maps on $\G$ for a fixed graph $\G$. More specifically, can we determine the minimal stretch factor on $\G$? 
As we require $f: \Gamma \rightarrow \Gamma$ is \textit{irreducible} on the edges of $\Gamma$, and folds help ensure irreducibility, there is a delicate balance between reducing folds and maintaining mixing amongst the edges of $\G$ under applications of $f$.  With this in mind, and taking inspiration from the language of stacks and mixing edges introduced in \cite{AKRstretch}, we define a graph invariant called the \hyperref[stack score]{stack score}, denoted $\mathfrak{S}(\G) \in \mathbb{N}$, as a way to measure the latent symmetry of $\G$. Informally, a smaller stack score reflects a higher degree of latent symmetry. In turn, latent symmetry allows one to incorporate more mixing into the graph isomorphism which follows the folds, and hence require fewer folds.

\vspace{0.25cm}
\noindent \hyperref[symmetry and folds]{\textbf{Theorem B.}} 
\textit{Any irreducible expanding homotopy equivalence self graph map $f:\Gamma \rightarrow \Gamma$ which is periodic on the vertex set of $\G$ must have at least $\mathfrak{S}(\Gamma)$ folds. }
\vspace{0.25cm}

It appears the condition that $f$ is periodic on the vertex set (equivalently, $f$ is a bijection on the vertex set) is not too restricive. For example, $f$ having a Stallings fold decomposition consisting of only proper full folds (and a graph isomorphism) is enough to guarantee periodicity of the vertex set. However, if $f$ has complete and partial folds, it may or may not be periodic on the vertices. 

The stretch factor of $\varphi \in $ Out$(F_r)$ represented by an irreducible train track map $f: \G \rightarrow \G$ is the leading eigenvalue of the integral $|\mathcal{E} \G| \times |\mathcal{E}\G|$ transition matrix of $f$.  \cite{bh92} Hence the algebraic degree of the stretch factor is bounded from above by the number of edges of $\G$. The following corollary, directly implied by Theorems A and B, is another example of a property of $\G$ affecting the set of possible stretch factors of train track maps on $\G$.

\vspace{0.25cm}
\noindent \hyperref[main cor]{\textbf{Corollary \ref*{main cor}}} \textit{Let $f :\G \rightarrow \G$ be an  irreducible expanding homotopy equivalence self graph map which is periodic on the vertex set of $\G$. Let $n=|\mathcal{E}\G|$. Then 
$$(\mathfrak{S}(\G)+1)^{\frac{1}{n}} \leq \lambda_f$$
where $\mathfrak{S}(\G)$ is the stack score of $\G$ and $\lambda_f$ is the leading eigenvalue of the transition matrix of $f$. }
\vspace{0.25cm}

Leveraging the restriction that a single fold irreducible self graph map must be periodic on the vertices and take place on a graph with stack score equal to 1, we obtain the following result. 

\newpage 

\begin{multicols}{2}
\noindent \hyperref[single folds]{\textbf{Theorem C.}} 
 \textit{Suppose $\G$ is a connected rank $r$ graph and $f: \G \rightarrow \G$ is a single fold irreducible homotopy equivalence self graph map. %Suppose $f: \G \rightarrow \G$ is a single fold irreducible train track map representing an irreducible $\varphi \in$ Out$(F_r)$. 
 Then $\G$ is isomorphic to one of the graphs to the right for some $k \geq 2$. }
 \vspace{0.3cm}
 
 \noindent \textit{In particular:
\begin{enumerate}
\item[(i)] if  $r \equiv 0$ mod $3$, then $\G \cong G \in \{R_r, \Delta_k^{-}\}$,
%\vspace{0.1cm}
\item[(ii)] if $r \equiv 1$ mod $3$, then $\G \cong R_r$, and
%\vspace{0.1cm}
\item[(iii)] if $r \equiv 2$ mod $3$, then  $\G \cong G \in \{R_r, \Delta_k^+\}$,
\end{enumerate}
%\vspace{0.1cm}
for appropriate values of $k$. }
\columnbreak
%\textcolor{white}{Single fold}\\
\tikzset{every picture/.style={line width=0.75pt}} %set default line width to 0.75pt        

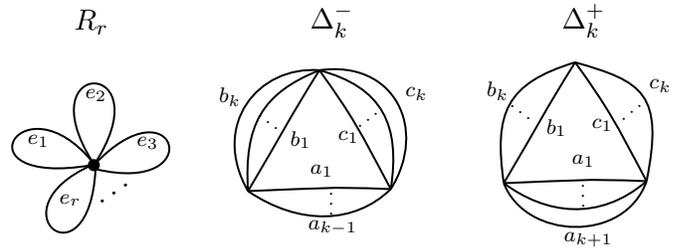
\hspace{-1.3cm} \vspace{0.3cm} \begin{tikzpicture}[x=0.75pt,y=0.75pt,yscale=-1,xscale=1]

%Curve Lines [id:da6924922789062755] 
\draw    (56.64,90.7) .. controls (23,37) and (93,38) .. (57.28,92.25) ;
%Curve Lines [id:da8706651017574671] 
\draw    (57.28,92.25) .. controls (99,48) and (116,108) .. (58.14,94.26) ;
%Shape: Ellipse [id:dp9558946616333523] 
\draw  [fill={rgb, 255:red, 0; green, 0; blue, 0 }  ,fill opacity=1 ][line width=0.75]  (58.98,90.32) .. controls (60.03,91.24) and (60.11,92.85) .. (59.17,93.92) .. controls (58.23,94.98) and (56.62,95.1) .. (55.58,94.17) .. controls (54.53,93.25) and (54.45,91.64) .. (55.39,90.58) .. controls (56.33,89.51) and (57.94,89.4) .. (58.98,90.32) -- cycle ;
%Curve Lines [id:da10834350526073577] 
\draw [line width=0.75]    (56.28,92.25) .. controls (-13,104) and (21,47) .. (56.64,90.7) ;
%Curve Lines [id:da084381339408659] 
\draw    (58.14,94.26) .. controls (55,155) and (5,115) .. (56.28,92.25) ;
%Curve Lines [id:da913529146454332] 
\draw    (135.04,105.36) .. controls (153,75) and (155,70) .. (171.04,44.36) ;
%Curve Lines [id:da016851898324729664] 
\draw    (135.04,105.36) .. controls (172,104) and (174,103) .. (207.04,104.36) ;
%Curve Lines [id:da6588945710687599] 
\draw    (171.04,44.36) .. controls (188,68) and (191,74) .. (207.04,104.36) ;
%Curve Lines [id:da23858804344759243] 
\draw    (135.04,105.36) .. controls (133,72) and (142,56) .. (171.04,44.36) ;
%Curve Lines [id:da936644372761162] 
\draw    (171.04,44.36) .. controls (203,53) and (212,72) .. (207.04,104.36) ;
%Curve Lines [id:da2378019011576622] 
\draw    (135.04,105.36) .. controls (161,123) and (189,122) .. (207.04,104.36) ;
%Curve Lines [id:da22821103469159265] 
\draw    (135.04,105.36) .. controls (116,72) and (140,42) .. (171.04,44.36) ;
%Curve Lines [id:da8896807307874222] 
\draw    (171.04,44.36) .. controls (209,39) and (226,74) .. (207.04,104.36) ;
%Curve Lines [id:da44969352087609327] 
\draw    (264.04,101.36) .. controls (282,71) and (284,66) .. (300.04,40.36) ;
%Curve Lines [id:da09284919830997662] 
\draw    (264.04,101.36) .. controls (301,100) and (303,99) .. (336.04,100.36) ;
%Curve Lines [id:da2677240522212432] 
\draw    (300.04,40.36) .. controls (317,64) and (320,70) .. (336.04,100.36) ;
%Curve Lines [id:da756640741133167] 
\draw    (264.04,101.36) .. controls (259,66) and (265,51) .. (300.04,40.36) ;
%Curve Lines [id:da9225502503516558] 
\draw    (300.04,40.36) .. controls (334,50) and (345,56) .. (336.04,100.36) ;
%Curve Lines [id:da6843141925934617] 
\draw    (264.04,101.36) .. controls (285,117) and (312,121) .. (336.04,100.36) ;
%Curve Lines [id:da1496828074210348] 
\draw    (264.04,101.36) .. controls (275,130) and (325,132) .. (336.04,100.36) ;

% Text Node
\draw (22.5,75.4) node [anchor=north west][inner sep=0.75pt]  [font=\scriptsize]  {$e_{1}$};
% Text Node
\draw (51.5,52.4) node [anchor=north west][inner sep=0.75pt]  [font=\scriptsize]  {$e_{2}$};
% Text Node
\draw (77.5,76.4) node [anchor=north west][inner sep=0.75pt]  [font=\scriptsize]  {$e_{3}$};
% Text Node
\draw (38.5,105.4) node [anchor=north west][inner sep=0.75pt]  [font=\scriptsize]  {$e_{r}$};
% Text Node
\draw (57.26,110.42) node [anchor=north west][inner sep=0.75pt]  [font=\large,rotate=-326.41]  {$\dotsc $};
% Text Node
\draw (46,12.4) node [anchor=north west][inner sep=0.75pt]  [font=\normalsize]  {$R_{r}$};
% Text Node
\draw (165,9.4) node [anchor=north west][inner sep=0.75pt]  [font=\normalsize]  {$\Delta _{k}^{-}$};
% Text Node
\draw (165,90.4) node [anchor=north west][inner sep=0.75pt]  [font=\scriptsize]  {$a_{1}$};
% Text Node
\draw (164,118.4) node [anchor=north west][inner sep=0.75pt]  [font=\scriptsize]  {$a_{k-1}$};
% Text Node
\draw (179.03,103.73) node [anchor=north west][inner sep=0.75pt]  [font=\tiny,rotate=-89.49]  {$\dotsc $};
% Text Node
\draw (179,73.4) node [anchor=north west][inner sep=0.75pt]  [font=\scriptsize]  {$c_{1}$};
% Text Node
\draw (203.24,66.23) node [anchor=north west][inner sep=0.75pt]  [font=\tiny,rotate=-142.02]  {$\dotsc $};
% Text Node
\draw (213,51.4) node [anchor=north west][inner sep=0.75pt]  [font=\scriptsize]  {$c_{k}$};
% Text Node
\draw (155,73.4) node [anchor=north west][inner sep=0.75pt]  [font=\scriptsize]  {$b_{1}$};
% Text Node
\draw (150.39,77.07) node [anchor=north west][inner sep=0.75pt]  [font=\tiny,rotate=-216.17]  {$\dotsc $};
% Text Node
\draw (119,51.4) node [anchor=north west][inner sep=0.75pt]  [font=\scriptsize]  {$b_{k}$};
% Text Node
\draw (292,9.4) node [anchor=north west][inner sep=0.75pt]  [font=\normalsize]  {$\Delta _{k}^{+}$};
% Text Node
\draw (297,85.4) node [anchor=north west][inner sep=0.75pt]  [font=\scriptsize]  {$a_{1}$};
% Text Node
\draw (306.03,99.73) node [anchor=north west][inner sep=0.75pt]  [font=\tiny,rotate=-89.49]  {$\dotsc $};
% Text Node
\draw (307,67.4) node [anchor=north west][inner sep=0.75pt]  [font=\scriptsize]  {$c_{1}$};
% Text Node
\draw (336,47.4) node [anchor=north west][inner sep=0.75pt]  [font=\scriptsize]  {$c_{k}$};
% Text Node
\draw (336.24,62.23) node [anchor=north west][inner sep=0.75pt]  [font=\tiny,rotate=-142.02]  {$\dotsc $};
% Text Node
\draw (284,67.4) node [anchor=north west][inner sep=0.75pt]  [font=\scriptsize]  {$b_{1}$};
% Text Node
\draw (254,48.4) node [anchor=north west][inner sep=0.75pt]  [font=\scriptsize]  {$b_{k}$};
% Text Node
\draw (277.69,71.99) node [anchor=north west][inner sep=0.75pt]  [font=\tiny,rotate=-216.17]  {$\dotsc $};
% Text Node
\draw (293,124.4) node [anchor=north west][inner sep=0.75pt]  [font=\scriptsize]  {$a_{k+1}$};

\end{tikzpicture}
\end{multicols}
\vspace{0.1cm}

%Among the fully irreducible single fold maps on the graphs above, we find those with the smallest stretch factor. 
Examples 6.2 and 6.3 in \cite{AKRstretch} are single fold irreducible train track maps on $R_r$ and $\{\Delta_k^-,\Delta_k^+\}$, respectively. Algom-Kfir and Rafi conjecture these maps on $\Delta_k^+$ and $\Delta_k^-$ attain the minimal stretch factor in their rank. For fully irreducible elements of Out$(F_3)$, \cite{wiggd1} shows this is indeed the case for $\Delta_2^-$, see Example \ref{ahlp single fold}.  As a consequence of Theorems \hyperref[lower bound]{A} and \hyperref[single folds]{C}, the Out$(F_r)$ conjugacy class determined by $\mathfrak{g}$ on $\Delta_2^-$ is in fact the unique minimizing conjugacy class among infinite order irreducible elements in Out$(F_3)$, see Corollary \ref{unique min}. \\

\noindent \textbf{Structure of the Paper.} \hyperref[background]{Section~\ref*{background}} gives necessary background about Out$(F_r)$ and graph maps. In \hyperref[folds and mixing]{Section~\ref*{folds and mixing}} we state and prove two lemmas relating folds and the length of images of edges. \hyperref[stack graphs]{Section~\ref*{stack graphs}} introduces stack graphs as a tool to understand the dynamics of components of irreducible graph maps. In \hyperref[lower bound proof]{Section~\ref*{lower bound proof}} we prove Theorem \hyperref[lower bound]{A} using stack graphs, and provide an alternate proof using Lemma \ref{hamsong} from \cite{hamsong} in the case that the transition matrix is primitive. \hyperref[latent symmetry]{Section \ref*{latent symmetry}} defines stack score and proves Theorem \hyperref[symmetry and folds]{B}. \hyperref[single fold maps]{Section~\ref*{single fold maps}} defines polygonal graphs and gives the proof of Theorem \hyperref[single folds]{C}. \hyperref[single fold maps]{Section~\ref*{further}} explores some applications and interesting examples. \\%Section 6 gives progress towards a method for finding the smallest stretch factors in rank 4 and 5. Section 7 consider further questions and research directions.  \\

\noindent \textbf{Acknowledgements.} Catherine Pfaff provided valuable guidance throughout the development of ideas in this paper. The author is also grateful to Darren Long for his steadfast encouragement and support, and to Mladen Bestvina, Naomi Andrew, and Robert Lyman for helpful conversations and comments on earlier versions of this paper. Chi Cheuk Tsang pointed out the use of Lemma \ref{hamsong} from \cite{hamsong} as a method of proving Theorem \hyperref[lower bound]{A} in the case that $f$ represents a fully irreducible outer automorphism. The author is also grateful to the referee for detailed comments on an earlier version of this paper. 
%\tableofcontents

%\newpage
\bigskip
\section{Background}\label{background}

Let $r \in \mathbb{Z}_{\geq 2}$ and $F_r$ be the free group of rank $r$. We are interested in the \textit{outer automorphisms of }$F_r$, 
$$\text{Out}(F_r) := \text{Aut}(F_r)/\text{Inn}(F_r).$$

In many ways, Out$(F_r)$ plays a similar role for graphs that the mapping class group plays for surfaces. Given a %finite type orientable 
surface $S$, the mapping class group  of $S$, $\mathcal{MCG}(S)$, is the group of isotopy classes of homeomorphisms on $S$. In 1974, Thurston classified elements of $\mathcal{MCG}(S)$ as either reducible, finite-order, or pseudo-Anosov \cite{thurston88}. Upon announcing his work, it was realized Nielsen made a similar discovery from a different perspective, and this classification is now known as the Nielsen--Thurston classification. 
Using the technology of train track maps on graphs, Bestvina and Handel developed an analogous classification of elements in Out$(F_n)$ \cite{bh92}. 
\begin{df}\label{irred, fullyirred} (Reducible, Irreducible, Fully Irreducible) An element $\varphi \in $ Out$(F_r)$ is called \textit{reducible} if there are free factors $A,B_1,\dots,B_k$ for $k >0$, such that $F_r = A * B_1 * \dots * B_k$ and $\varphi$ transitively permutes the conjugacy classes of the $B_i$.  Otherwise, $\varphi$ is \textit{irreducible}. We say $\varphi$ is \textit{fully irreducible} if every power of $\varphi$ is irreducible. \end{df}

%\begin{ex}\label{irred, not fully} Consider $r=3$ and $F_3=\langle x,y,z \rangle$. Let $\varphi:F_3 \rightarrow F_3$ be given by
%$$\varphi : x \mapsto y \mapsto z \mapsto xx$$
%Then $[\varphi] \in $ Out$(F_r)$ is irreducible. However, $\varphi^3(x) = xx$, so $[\varphi^3]$ preserves the conjugacy class of $\langle x \rangle$. Hence $[\varphi]$ is not fully irreducible. 
%\end{ex}

From some perspectives, fully irreducible outer automorphisms are analogous to pseudo-Anosov elements in the mapping class group. 
  
\begin{df}\label{graph} (Graph, Directed Graph) A \textit{graph} $\G$ is a 1-dimensional CW complex whose 0-simplices are vertices, denoted $\mathcal{V}G$, and whose 1-simplices are edges, denoted $\mathcal{E}\G$. Note that we allow for multiple edges between vertices, as well as self loops. We will always assume our graphs have finitely many edges and vertices.

 When there is a choice of orientation on each edge, $\G$ is a \textit{directed graph} and we let $\mathcal{E}\G$ denote the set of positively oriented edges, $\mathcal{E}^-\G$ the negatively oriented edges, and $\mathcal{E}^\pm \G$ the union of both. We let $\overline{e}$ denote the edge $e$ with reversed orientation. We have initial and terminal maps $$\iota, \tau : \mathcal{E}^\pm \G \rightarrow \mathcal{V}\G$$
 given by $\iota(e)=$ initial vertex of $e$ and $\tau(e)=$ terminal vertex of $e$.
 \end{df}
 
 \begin{df}\label{edge path} (Edge Path) An \textit{edge path} in $\G$ is a nonempty concatenation of oriented edges $e_1 \dots e_k$ such that $\tau(e_i) = \iota(e_{i+1})$ for all $1 \leq i \leq k-1$. If $u = e_1 \dots e_k$ is an edge path, then 
 \begin{enumerate}
 \item[(i)] $\iota(u) := \iota(e_1)$, 
 \item[(ii)]  $\tau(u):=\tau(e_k)$, and 
 \item[(iii)] $\overline{u} := \overline{e_k} \dots \overline{e_1}$. 
 \end{enumerate} Let $\mathcal{E}\mathcal{P}\G$ denote the set of edge paths in $\G$. Note that we can interpret $\mathcal{E}^{\pm} \G $ as a subset of $\mathcal{E}\mathcal{P}\G$ by identifying an oriented edge $e$ with the edge path equal to $e$. 
\end{df}

%Changing the identification is equivalent 

%to composing $f_*$ with an inner automorphism of $F_r$. Thus we can unambiguously consider $[f_*] \in$ Out$(F_r)$, without specifying the identification of $\pi_1(\G)$ with  $F_r$. 

\begin{df}\label{graph map} (Graph Map) Given graphs $\G_1$ and $\G_2$, a  \textit{graph map} $f: \G_1 \rightarrow \G_2$ consists of maps
\begin{enumerate} \item[(i)] $f_{V}:\mathcal{V}\G_1 \rightarrow \mathcal{V}\G_2$, and 
\item[(ii)] $f_E: \mathcal{E}^{\pm}\G_1 \rightarrow \mathcal{E}\mathcal{P}\G_2$ such that $f_V(\iota(e))=\iota(f_E(e))$ and $f_E(\overline{e})=\overline{f_E(e)}$ for every $e \in \mathcal{E}^{\pm}\G$.  
\end{enumerate}
\end{df}

\begin{nt}\label{word length} Given an edge path $u$ in a graph $\G$, we use $|u|$ to denote the number of edges in $u$. We say $u$ \textit{traverses} $e \in \mathcal{E}\G$ if $e$ or $\overline{e}$ appears as an edge in $u$. Note that if a sequence $e \overline{e}$ appears in an edge path $u$, both $e$ and $\overline{e}$ contribute to the number of edges in $u$. In other words, we do not tighten the path $u$ before counting the number of edges. Thus $|f(u)| \geq |u|$ for any graph map $f$ and edge path $u$. %By definition \ref{train track}, if $f: \G \rightarrow \G$ is a train track map and $e \in \mathcal{E}\G$, then $|f(e)|$ is equal to the cyclically reduced word length of $e$. Moreover, if $f = g \circ h$ for graph maps $g$ and $h$, then $|h(e)|$ is also equal to the cyclically reduced word length of $e$, since $h$ must also be injective on the interior of every edge. %Morevoer, given any edge path $u$ in $\G$, $|f(u)| \geq |u|$, where $|\cdot|$ is the cyclically reduced word length.
\end{nt}

\begin{df}\label{isomorphism} (Graph Isomorphism, Graph Automorphism)  A graph map $f: \G_1 \rightarrow \G_2$ is a \textit{graph isomorphism} if 
\begin{enumerate}
\item[(i)]$f_V$ is a bijection, and 
\item[(ii)] $f_E$ is injective with image equal to $\mathcal{E}^{\pm}\G_2$.
\end{enumerate} 
A graph isomorphism $f:\G \rightarrow \G$ is a \textit{graph automorphism}. \end{df}

\begin{nt} Given a graph map $f:\G_1\rightarrow \G_2$, we often drop the subscripts on the corresponding maps on the vertices and edges, and just write $f(e)$ for $f_E(e)$ and $f(v)$ for $f_V(v)$ when it is clear that $e$ is an edge and $v$ is a vertex. 
\end{nt}

\begin{nt} When $\G_1$ has no isolated vertices, a graph map $f:\G_1 \rightarrow \G_2$ is entirely determined by $f_E$ restricted to the set of positively oriented edges of $\G_1$. We will often define a graph map by just giving its image on every positively oriented edge. 
\end{nt}

In order to define graph maps on $\G$, we always assume our graphs have an orientation on each edge. However, since edge paths can traverse edges backwards, these orientations do not carry meaningful information about the nature of the graph itself (with the exception of stack graphs, see Definition \ref{stack graph}). %In other words, given any graph map $f: \G_1 \rightarrow \G_2$, we could change the orientations on $\G_1$ and $\G_2$ and define a new map which is essentially the same as $f$, but will just be written down differently. \\

\noindent If $f:\G \rightarrow \G$ is a homotopy equivalence on a connected graph $\G$, then the induced map 
$$f_*:\pi_1(\G) \rightarrow \pi_1(\G)$$
is an outer automorphism of $\pi_1(\G)$. As $\pi_1(\G)$ is isomorphic to a free group $F_r$, after a choice of identification of $\pi_1(\G)$ with  $F_r$, we can consider $f_*$ as an element of Out$(F_r)$. We say that $f: \G \rightarrow \G$ \textit{topologically represents }$f_*$. Different choices of identification of $\pi_1(\G)$ with $F_r$ give Out$(F_r)$-conjugate outer automorphisms. 

\begin{df}\label{transition matrix} (Transition Matrix) Given a self graph map $f : \G \rightarrow \G$, and an order on the set of edges $(e_1, \dots, e_n)$, the \textit{transition matrix} of $f$, denoted $T(f)$, is the $|\mathcal{E}\G| \times |\mathcal{E}\G|$ matrix $(a_{ij})$ where $a_{ij}$ is the number of times $f(e_i)$ traverses $e_j$ in either direction.
\end{df}

\begin{df}\label{irreducible, primitive} (Irreducible, Primitive) Let $M$ be an $n \times n$ matrix. 
\begin{itemize}
\item[(i)] $M$ is \textit{irreducible} if for each $1 \leq i,j \leq n$, there is a power $k$ such that the $ij$-th entry of $M^k$ is positive. When $M$ is non-negative, this is equivalent to requiring that $M$ has no non-trivial proper invariant coordinate subspaces. The coordinate subspaces are those which are spanned by a subset of the standard basis elements in $\mathbb{R}^n$. 

\item[(ii)] $M$ is \textit{primitive} if it is non-negative and there is a power $k$ such that all entries of $M^k$ are positive. 
\end{itemize}
\end{df}

\begin{df}\label{irreducible graph map} (Irreducible Graph Map) We call a self graph map $f:\G \rightarrow \G$  \textit{irreducible}  if $T(f)$ is an irreducible matrix and the valence of every vertex in $\G$ is at least $3$. %In this case, we may assume that every vertex in $\G$  has valence at least $3$. \cite{bh92}
\end{df}

\begin{df}\label{expanding} (Expanding Graph Map) A self graph map $f:\G \rightarrow \G$ is \textit{expanding} if $|f^n(e)| \rightarrow \infty$ as $n \rightarrow \infty$ for every edge $e \in \mathcal{E}\G$. 
When $f$ is an irreducible homotopy equivalence, this is equivalent to requiring the largest eigenvalue of $T(f)$ is strictly greater than 1 in modulus (see Lemma \ref{helpful facts}). \end{df}

\begin{df}\label{train track} (Train Track Map) A self graph map $f:\G \rightarrow \G$ is a \textit{train track map} if it is a homotopy equivalence and for all powers $n \in \mathbb{N}$, $f^n$ is locally injective on the interior of every edge $e$.  
\end{df}

We will sometimes refer to an irreducible train track map as an i.t.t. map and an irreducible homotopy equivalence graph map as an i.h.e. map. Our proofs do not use the locally injective property of train track maps, and hence our results are stated for i.h.e. maps.\\

%\newpage 

The following theorem reduces the question of stretch factors of irreducible outer automorphisms to a question about leading eigenvalues of their i.t.t. representatives. 

\begin{thm}[\cite{bh92}]\label{bestvina}Every irreducible outer automorphism $\varphi \in $ Out$(F_r)$ is represented by an irreducible train track map $f: \G \rightarrow \G$ on a connected rank $r$ graph $\G$. The leading eigenvalue of $T(f)$, denoted $\lambda_f$, is real, positive, and equal to the stretch factor of $\varphi$. Moreover, there is a length function $\ell$ on the edges of $\G$ such that $f$ is uniformly $\lambda_f-$expanding on $(\G,\ell)$. That is, $\ell(f(e))=\lambda_f \ell (e)$ for every $e \in \mathcal{E}\G$. Further, $\varphi$ is a finite-order homeomorphism if and only if $\lambda_f=1$.
\end{thm}

However, it should be noted that while every irreducible outer automorphism has an i.t.t. representative, a given i.t.t. map could induce an outer automorphism which is reducible.\\
%thurston's theorem 

In \cite{AKRstretch}, Algom-Kfir and Rafi define \textit{mixing edges} and \textit{stacks} of graph maps. We recall their definitions here.

\begin{df}\label{mixing edge} \cite{AKRstretch} (Mixing Edge) Given a graph map $f: \Gamma_1 \rightarrow \Gamma_2$, an edge $e$ is called a \textit{mixing edge} if $f(e)$ is an edge path consisting of more than one edge. 
\end{df}

\begin{df}\label{surplus edge} (Surplus Edge)  Given a graph map $f: \Gamma_1 \rightarrow \Gamma_2$, an edge $e$ is called a \textit{surplus edge} if $e$ is non-mixing and $f(e) \in \{f(u), \overline{f(u)}\}$ for some edge $u \in \mathcal{E}\G_1$ with $ u \notin \{e ,\overline{e}\}$.  
\end{df}

%\newpage 

\begin{df}\label{stack}  \cite{AKRstretch} (Stack) Given a self graph map $f :\G \rightarrow \G$, let $\sim$ be an equivalence relation on the edges of $\G$ generated by $e \sim f(e)$ if $e$ is non-mixing and non-surplus. An equivalence class of edges is called a \textit{stack} \footnote{This definition of stack differs slightly from that in \cite{AKRstretch}, as we allow $e \sim f(e)$ even if $f(e)$ appears in the image of a mixing edge.}. The stacks of $f$ partition $\mathcal{E}\G$.
\end{df}

 \begin{ex}\label{ahlp single fold} \vspace{-0.3cm}
 \begin{multicols}{2}
Let $\mathfrak{g}: \Delta_2^- \rightarrow \Delta_2^-$ be as pictured. This is an expanding i.t.t. map representing the fully irreducible outer automorphism $\varphi: x \mapsto y \mapsto z \mapsto zx^{-1}$, which has minimal stretch factor among fully irreducible elements of Out$(F_3)$ \cite{wiggd1}. $\mathfrak{g}$ has a single stack equal to $\mathcal{E} \Delta_2^-$ and a single mixing edge, $c_2$.  \\

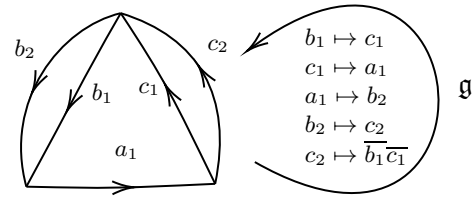
\columnbreak 
\textcolor{white}{\ }\\
\vspace{-2.5cm}
 %\begin{wrapfigure}{r}{7cm}
    % \centering
    \tikzset{every picture/.style={line width=0.75pt}} %set default line width to 0.75pt        

\hspace{0.8cm}\begin{tikzpicture}[x=0.75pt,y=0.75pt,yscale=-1,xscale=1]

%Curve Lines [id:da9007426014808222] 
\draw    (11.76,103.93) .. controls (35.51,60.35) and (38.15,53.17) .. (59.37,16.36) ;
\draw [shift={(32.07,66.01)}, rotate = 297.85] [color={rgb, 255:red, 0; green, 0; blue, 0 }  ][line width=0.75]    (10.93,-3.29) .. controls (6.95,-1.4) and (3.31,-0.3) .. (0,0) .. controls (3.31,0.3) and (6.95,1.4) .. (10.93,3.29)   ;
%Curve Lines [id:da7825403705909557] 
\draw    (106.99,102.49) .. controls (55.35,104.85) and (55.35,104.85) .. (11.76,103.93) ;
\draw [shift={(66.05,104.21)}, rotate = 178.48] [color={rgb, 255:red, 0; green, 0; blue, 0 }  ][line width=0.75]    (10.93,-3.29) .. controls (6.95,-1.4) and (3.31,-0.3) .. (0,0) .. controls (3.31,0.3) and (6.95,1.4) .. (10.93,3.29)   ;
%Curve Lines [id:da37431638737442285] 
\draw    (59.37,16.36) .. controls (81.8,50.3) and (85.77,58.91) .. (106.99,102.49) ;
\draw [shift={(81.69,52.29)}, rotate = 61.28] [color={rgb, 255:red, 0; green, 0; blue, 0 }  ][line width=0.75]    (10.93,-3.29) .. controls (6.95,-1.4) and (3.31,-0.3) .. (0,0) .. controls (3.31,0.3) and (6.95,1.4) .. (10.93,3.29)   ;
%Curve Lines [id:da9257580247309034] 
\draw    (59.37,16.36) .. controls (99,28.5) and (116,58.5) .. (106.99,102.49) ;
\draw [shift={(98.54,42.84)}, rotate = 58.4] [color={rgb, 255:red, 0; green, 0; blue, 0 }  ][line width=0.75]    (10.93,-3.29) .. controls (6.95,-1.4) and (3.31,-0.3) .. (0,0) .. controls (3.31,0.3) and (6.95,1.4) .. (10.93,3.29)   ;
%Curve Lines [id:da4336217292526092] 
\draw    (11.76,103.93) .. controls (0,60.5) and (29,23.5) .. (59.37,16.36) ;
\draw [shift={(14.69,56.4)}, rotate = 295.27] [color={rgb, 255:red, 0; green, 0; blue, 0 }  ][line width=0.75]    (10.93,-3.29) .. controls (6.95,-1.4) and (3.31,-0.3) .. (0,0) .. controls (3.31,0.3) and (6.95,1.4) .. (10.93,3.29)   ;
%Curve Lines [id:da24358758206264364] 
\draw [line width=0.75]    (127,91.5) .. controls (258,166.5) and (244,-57.5) .. (123,35.5) ;
\draw [shift={(123,35.5)}, rotate = 322.45] [color={rgb, 255:red, 0; green, 0; blue, 0 }  ][line width=0.75]    (10.93,-3.29) .. controls (6.95,-1.4) and (3.31,-0.3) .. (0,0) .. controls (3.31,0.3) and (6.95,1.4) .. (10.93,3.29)   ;

% Text Node
\draw (54.8,82.02) node [anchor=north west][inner sep=0.75pt]  [font=\footnotesize]  {$a_{1}$};
% Text Node
\draw (66.87,49.92) node [anchor=north west][inner sep=0.75pt]  [font=\footnotesize]  {$c_{1}$};
% Text Node
\draw (43.22,50.36) node [anchor=north west][inner sep=0.75pt]  [font=\footnotesize]  {$b_{1}$};
% Text Node
\draw (4.22,27.96) node [anchor=north west][inner sep=0.75pt]  [font=\footnotesize]  {$b_{2}$};
% Text Node
\draw (101.58,28.08) node [anchor=north west][inner sep=0.75pt]  [font=\footnotesize]  {$c_{2}$};
% Text Node
\draw (144,20.4) node [anchor=north west][inner sep=0.75pt]  [font=\footnotesize]  {$ \begin{array}{l}
b_{1} \mapsto c_{1}\\
c_{1} \mapsto a_{1}\\
a_{1} \mapsto b_{2}\\
b_{2} \mapsto c_{2}\\
c_{2} \mapsto \overline{b_{1}}\overline{c_{1}}
\end{array}$};
% Text Node
\draw (228,48.4) node [anchor=north west][inner sep=0.75pt]    {$\mathfrak{g}$};

\end{tikzpicture}
%  \end{wrapfigure}
  \end{multicols}
  \end{ex}

\vspace{-1cm}

%\noindent  Every irreducible $\gamma \in $ Out$(F_r)$ is topologically represented by a train track map $f: \G \rightarrow \G$ for some graph $\G$. Moreover, $\G$ can be chosen so that every vertex has valence at  least $3$ and $f$ is irreducible on the $\mathcal{E}\G$. \cite{bh92}

%\noindent Suppose a train track map $f:\G \rightarrow \G$ induces $\varphi \in \text{Out}(F_r)$. Conveniently,
%\begin{itemize}
%\item If $T(f)$ is irreducible, we call $f$ an \textit{irreducible train track map} and this implies $\varphi$ is an irreducible element of Out$(F_r)$. \cite{bh92}. 
%\item If $T(f)$ is primitive, then $\varphi$ is a fully irreducible element of Out$(F_r)$. (cite?)
%\end{itemize}
%(CHECK/ cite). $T(f)$ also reveals the stretch factor of $\varphi$: 

%Perron-frobenius plus other stuff: (cite)!)

\begin{df}\label{folds} \vspace{-0.6cm}(Folds) Given a directed graph $\Gamma$ and two edges $e_0,e_1 \in \mathcal{E}^{\pm}\G$ such that $\iota(e_0)=\iota(e_1)$, there are three procedures, called \textit{folds}, to form a new graph $\Gamma'$ and a surjective graph map $f:\Gamma \rightarrow \Gamma'$.  We describe these three types of folds first in terms of a procedure. Then, we give the equivalent definition of these folds in terms of a quotient graph and a quotient map. The latter definition is more standard, but the former definition determines our convention for labels on $\G'$. \\
\end{df}

 \begin{wrapfigure}{r}{4.5cm}
     \centering
    \tikzset{every picture/.style={line width=0.75pt}} %set default line width to 0.75pt        

\begin{tikzpicture}[x=0.75pt,y=0.75pt,yscale=-1,xscale=1]
%uncomment if require: \path (0,150); %set diagram left start at 0, and has height of 150

%Straight Lines [id:da1195136137570314] 
\draw [color={rgb, 255:red, 125; green, 172; blue, 71 }  ,draw opacity=1 ][line width=1.5]    (27,64) -- (48,21) ;
\draw [shift={(33.64,50.41)}, rotate = 296.03] [color={rgb, 255:red, 125; green, 172; blue, 71 }  ,draw opacity=1 ][line width=1.5]    (14.21,-4.28) .. controls (9.04,-1.82) and (4.3,-0.39) .. (0,0) .. controls (4.3,0.39) and (9.04,1.82) .. (14.21,4.28)   ;
%Straight Lines [id:da813577164635408] 
\draw [color={rgb, 255:red, 74; green, 144; blue, 226 }  ,draw opacity=1 ][line width=1.5]    (71,64) -- (48,21) ;
\draw [shift={(63.65,50.26)}, rotate = 241.86] [color={rgb, 255:red, 74; green, 144; blue, 226 }  ,draw opacity=1 ][line width=1.5]    (14.21,-4.28) .. controls (9.04,-1.82) and (4.3,-0.39) .. (0,0) .. controls (4.3,0.39) and (9.04,1.82) .. (14.21,4.28)   ;
%Shape: Ellipse [id:dp016181784453500825] 
\draw  [fill={rgb, 255:red, 0; green, 0; blue, 0 }  ,fill opacity=1 ] (44.8,22.65) .. controls (43.84,21.1) and (44.43,18.94) .. (46.11,17.82) .. controls (47.8,16.71) and (49.95,17.07) .. (50.91,18.62) .. controls (51.87,20.18) and (51.28,22.34) .. (49.6,23.46) .. controls (47.91,24.57) and (45.77,24.21) .. (44.8,22.65) -- cycle ;
%Straight Lines [id:da9803858531192833] 
\draw    (93.37,47.97) -- (120.49,47.97) ;
\draw [shift={(122.49,47.97)}, rotate = 180] [color={rgb, 255:red, 0; green, 0; blue, 0 }  ][line width=0.75]    (10.93,-3.29) .. controls (6.95,-1.4) and (3.31,-0.3) .. (0,0) .. controls (3.31,0.3) and (6.95,1.4) .. (10.93,3.29)   ;
%Straight Lines [id:da502949247073458] 
\draw [color={rgb, 255:red, 0; green, 0; blue, 0 }  ,draw opacity=1 ][line width=1.5]    (141,64) -- (163,19) ;
\draw [shift={(148.13,49.41)}, rotate = 296.05] [color={rgb, 255:red, 0; green, 0; blue, 0 }  ,draw opacity=1 ][line width=1.5]    (14.21,-4.28) .. controls (9.04,-1.82) and (4.3,-0.39) .. (0,0) .. controls (4.3,0.39) and (9.04,1.82) .. (14.21,4.28)   ;
%Straight Lines [id:da06427244901241469] 
\draw [color={rgb, 255:red, 0; green, 0; blue, 0 }  ,draw opacity=1 ][line width=1.5]    (187,64) -- (141,64) ;
\draw [shift={(172.8,64)}, rotate = 180] [color={rgb, 255:red, 0; green, 0; blue, 0 }  ,draw opacity=1 ][line width=1.5]    (14.21,-4.28) .. controls (9.04,-1.82) and (4.3,-0.39) .. (0,0) .. controls (4.3,0.39) and (9.04,1.82) .. (14.21,4.28)   ;
%Curve Lines [id:da9320752543024935] 
\draw [color={rgb, 255:red, 74; green, 144; blue, 226 }  ,draw opacity=0.8 ][line width=1.5]    (163.27,22.68) .. controls (141.13,71.44) and (137.64,61.81) .. (186.56,62.41) ;
%Straight Lines [id:da028182443253153888] 
\draw [color={rgb, 255:red, 125; green, 172; blue, 71 }  ,draw opacity=0.77 ][line width=1.5]    (163.27,17.86) -- (141.13,62.41) ;
%Shape: Ellipse [id:dp9204393603735865] 
\draw  [fill={rgb, 255:red, 0; green, 0; blue, 0 }  ,fill opacity=1 ] (68.1,66) .. controls (67.14,64.44) and (67.73,62.28) .. (69.41,61.17) .. controls (71.1,60.06) and (73.24,60.41) .. (74.21,61.97) .. controls (75.17,63.53) and (74.58,65.69) .. (72.9,66.8) .. controls (71.21,67.91) and (69.06,67.56) .. (68.1,66) -- cycle ;
%Shape: Ellipse [id:dp9793341079046625] 
\draw  [fill={rgb, 255:red, 0; green, 0; blue, 0 }  ,fill opacity=1 ] (23.83,66) .. controls (22.87,64.44) and (23.46,62.28) .. (25.15,61.17) .. controls (26.83,60.06) and (28.98,60.41) .. (29.94,61.97) .. controls (30.9,63.53) and (30.31,65.69) .. (28.63,66.8) .. controls (26.94,67.91) and (24.8,67.56) .. (23.83,66) -- cycle ;
%Shape: Ellipse [id:dp6892479271408569] 
\draw  [fill={rgb, 255:red, 0; green, 0; blue, 0 }  ,fill opacity=1 ] (138,66) .. controls (137.04,64.44) and (137.63,62.28) .. (139.31,61.17) .. controls (141,60.06) and (143.14,60.41) .. (144.1,61.97) .. controls (145.07,63.53) and (144.48,65.69) .. (142.79,66.8) .. controls (141.11,67.91) and (138.96,67.56) .. (138,66) -- cycle ;
%Shape: Ellipse [id:dp43393807057219114] 
\draw  [fill={rgb, 255:red, 0; green, 0; blue, 0 }  ,fill opacity=1 ] (160.14,21.45) .. controls (159.17,19.9) and (159.76,17.73) .. (161.45,16.62) .. controls (163.13,15.51) and (165.28,15.87) .. (166.24,17.42) .. controls (167.2,18.98) and (166.61,21.14) .. (164.93,22.25) .. controls (163.24,23.36) and (161.1,23.01) .. (160.14,21.45) -- cycle ;
%Shape: Ellipse [id:dp562165502170118] 
\draw  [fill={rgb, 255:red, 0; green, 0; blue, 0 }  ,fill opacity=1 ] (183.44,66) .. controls (182.47,64.44) and (183.06,62.28) .. (184.75,61.17) .. controls (186.43,60.06) and (188.58,60.41) .. (189.54,61.97) .. controls (190.5,63.53) and (189.91,65.69) .. (188.23,66.8) .. controls (186.54,67.91) and (184.4,67.56) .. (183.44,66) -- cycle ;

% Text Node
\draw (20.49,18.64) node [anchor=north west][inner sep=0.75pt]  [font=\normalsize,color={rgb, 255:red, 65; green, 117; blue, 5 }  ,opacity=1 ]  {$e_{0}$};
% Text Node
\draw (59.87,19.05) node [anchor=north west][inner sep=0.75pt]  [font=\normalsize,color={rgb, 255:red, 56; green, 109; blue, 172 }  ,opacity=1 ]  {$e_{1}$};
% Text Node
\draw (135.33,20.97) node [anchor=north west][inner sep=0.75pt]  [font=\normalsize]  {$e_{0}$};
% Text Node
\draw (153.15,68.81) node [anchor=north west][inner sep=0.75pt]  [font=\normalsize]  {$e_{1} '$};

\end{tikzpicture}
  \end{wrapfigure}
\noindent (i) (Proper Full Fold)  Let $\G'$ be the graph with $\mathcal{V}\G' = \mathcal{V}\G$ and $\mathcal{E}~\hspace{-0.17cm}\G'~ =~(\mathcal{E}\G~-~\{e_1\})~\cup~\{e_1'\}$, where $e_1'$ has $\iota(e_1'):=\tau(e_0)$ and $\tau(e_1'):=\tau(e_1)$. Let $f: \G \rightarrow \G'$ be given by:
$$\hspace{-7cm} f(e) = \begin{cases}e_0e_1' & \text{if }e=e_1 \\ e & \text{otherwise} \end{cases}$$
$f$ is called the \textit{proper full fold of $e_1$ over $e_0$.} Equivalently, subdivide $e_1 \in \mathcal{E}\G$: let $v'$ be a new vertex in the middle of $e_1$ and relabel $e_1$ as two edges $e_1''$ and $e_1'$, oriented so that $e_1$ is now equal to the edge path $e_1''e_1'$. Now, let $\G' = \G / e_1'' \sim e_0$, and let $f:\G \rightarrow \G'$ be the quotient map. \\
%\textcolor{white}{proper full  \\ fold}

 \begin{wrapfigure}{r}{4.5cm}
     \centering
    \tikzset{every picture/.style={line width=0.75pt}} %set default line width to 0.75pt        

\begin{tikzpicture}[x=0.75pt,y=0.75pt,yscale=-1,xscale=1]
%uncomment if require: \path (0,161); %set diagram left start at 0, and has height of 161

%Straight Lines [id:da818026981257939] 
\draw [color={rgb, 255:red, 0; green, 0; blue, 0 }  ,draw opacity=1 ][line width=1.5]    (143,63.5) -- (165,22.5) ;
\draw [shift={(149.84,50.75)}, rotate = 298.22] [color={rgb, 255:red, 0; green, 0; blue, 0 }  ,draw opacity=1 ][line width=1.5]    (14.21,-4.28) .. controls (9.04,-1.82) and (4.3,-0.39) .. (0,0) .. controls (4.3,0.39) and (9.04,1.82) .. (14.21,4.28)   ;
%Straight Lines [id:da2645857590769385] 
\draw [color={rgb, 255:red, 125; green, 172; blue, 71 }  ,draw opacity=0.77 ][line width=1.5]    (165.38,20.3) -- (141.89,63.28) ;
%Straight Lines [id:da406913118823498] 
\draw [color={rgb, 255:red, 74; green, 144; blue, 226 }  ,draw opacity=0.85 ][line width=1.5]    (166.62,22.62) -- (143.13,65.6) ;
%Shape: Ellipse [id:dp05938607667326279] 
\draw  [fill={rgb, 255:red, 0; green, 0; blue, 0 }  ,fill opacity=1 ] (139.8,65.57) .. controls (138.82,64.06) and (139.46,61.95) .. (141.25,60.88) .. controls (143.04,59.81) and (145.29,60.17) .. (146.28,61.69) .. controls (147.27,63.21) and (146.62,65.31) .. (144.83,66.38) .. controls (143.05,67.46) and (140.79,67.09) .. (139.8,65.57) -- cycle ;
%Shape: Ellipse [id:dp772314504341064] 
\draw  [fill={rgb, 255:red, 0; green, 0; blue, 0 }  ,fill opacity=1 ] (162.06,24.92) .. controls (161.07,23.4) and (161.72,21.3) .. (163.51,20.22) .. controls (165.3,19.15) and (167.55,19.51) .. (168.54,21.03) .. controls (169.53,22.55) and (168.88,24.65) .. (167.09,25.73) .. controls (165.3,26.8) and (163.05,26.44) .. (162.06,24.92) -- cycle ;
%Straight Lines [id:da6979851133628707] 
\draw    (89.96,47.39) -- (115.66,47.39) ;
\draw [shift={(117.66,47.39)}, rotate = 180] [color={rgb, 255:red, 0; green, 0; blue, 0 }  ][line width=0.75]    (10.93,-3.29) .. controls (6.95,-1.4) and (3.31,-0.3) .. (0,0) .. controls (3.31,0.3) and (6.95,1.4) .. (10.93,3.29)   ;
%Straight Lines [id:da18852661795091952] 
\draw [color={rgb, 255:red, 125; green, 172; blue, 71 }  ,draw opacity=1 ][line width=1.5]    (27.51,64.76) -- (48.51,21.76) ;
\draw [shift={(34.15,51.17)}, rotate = 296.03] [color={rgb, 255:red, 125; green, 172; blue, 71 }  ,draw opacity=1 ][line width=1.5]    (14.21,-4.28) .. controls (9.04,-1.82) and (4.3,-0.39) .. (0,0) .. controls (4.3,0.39) and (9.04,1.82) .. (14.21,4.28)   ;
%Straight Lines [id:da1429042636017186] 
\draw [color={rgb, 255:red, 74; green, 144; blue, 226 }  ,draw opacity=1 ][line width=1.5]    (71.51,64.76) -- (48.51,21.76) ;
\draw [shift={(64.16,51.02)}, rotate = 241.86] [color={rgb, 255:red, 74; green, 144; blue, 226 }  ,draw opacity=1 ][line width=1.5]    (14.21,-4.28) .. controls (9.04,-1.82) and (4.3,-0.39) .. (0,0) .. controls (4.3,0.39) and (9.04,1.82) .. (14.21,4.28)   ;
%Shape: Ellipse [id:dp09718669029672355] 
\draw  [fill={rgb, 255:red, 0; green, 0; blue, 0 }  ,fill opacity=1 ] (45.31,23.41) .. controls (44.35,21.86) and (44.94,19.7) .. (46.62,18.58) .. controls (48.31,17.47) and (50.45,17.83) .. (51.41,19.38) .. controls (52.38,20.94) and (51.79,23.1) .. (50.1,24.22) .. controls (48.42,25.33) and (46.27,24.97) .. (45.31,23.41) -- cycle ;
%Shape: Ellipse [id:dp2531049832451189] 
\draw  [fill={rgb, 255:red, 0; green, 0; blue, 0 }  ,fill opacity=1 ] (68.61,66.76) .. controls (67.65,65.2) and (68.24,63.04) .. (69.92,61.93) .. controls (71.61,60.81) and (73.75,61.17) .. (74.71,62.73) .. controls (75.67,64.28) and (75.09,66.45) .. (73.4,67.56) .. controls (71.72,68.67) and (69.57,68.31) .. (68.61,66.76) -- cycle ;
%Shape: Ellipse [id:dp8188602372733409] 
\draw  [fill={rgb, 255:red, 0; green, 0; blue, 0 }  ,fill opacity=1 ] (24.34,66.76) .. controls (23.38,65.2) and (23.97,63.04) .. (25.65,61.93) .. controls (27.34,60.81) and (29.48,61.17) .. (30.44,62.73) .. controls (31.41,64.28) and (30.82,66.45) .. (29.13,67.56) .. controls (27.45,68.67) and (25.3,68.31) .. (24.34,66.76) -- cycle ;

% Text Node
\draw (132.34,20.99) node [anchor=north west][inner sep=0.75pt]  [font=\normalsize]  {$e_{0} '$};
% Text Node
\draw (21,19.4) node [anchor=north west][inner sep=0.75pt]  [font=\normalsize,color={rgb, 255:red, 65; green, 117; blue, 5 }  ,opacity=1 ]  {$e_{0}$};
% Text Node
\draw (60.38,19.81) node [anchor=north west][inner sep=0.75pt]  [font=\normalsize,color={rgb, 255:red, 56; green, 109; blue, 172 }  ,opacity=1 ]  {$e_{1}$};

\end{tikzpicture}
  \end{wrapfigure}
\noindent (ii) (Complete Fold) Let $\G'$ be the graph resulting from identifying the vertices $\iota(e_0)$ and $\iota(e_1)$ and identifying the edges $e_0$ and $e_1$ into a new edge labelled $e_0'$. Let $f: \G \rightarrow \G'$ be given by:
$$\hspace{-7cm} f(e) = \begin{cases}e_0' & \text{if }e\in\{e_0,e_1\} \\ e & \text{otherwise} \end{cases}$$
$f$ is called the \textit{complete fold of $e_1$ and $e_0$}.  Equivalently, let $\G' = \G / e_1 \sim e_0$, and let $f:\G \rightarrow \G'$ be the quotient map. If $f$ is a fold in a fold decomposition of a homotopy equivalence, then $\tau(e_0) \neq \tau(e_1)$.\\

%\textcolor{white}{complete fold}
%\newpage
 \begin{wrapfigure}{r}{4.5cm}
     \centering
    \tikzset{every picture/.style={line width=0.75pt}} %set default line width to 0.75pt        

\begin{tikzpicture}[x=0.75pt,y=0.75pt,yscale=-1,xscale=1]
%uncomment if require: \path (0,162); %set diagram left start at 0, and has height of 162

%Straight Lines [id:da3415925152594681] 
\draw [color={rgb, 255:red, 0; green, 0; blue, 0 }  ,draw opacity=1 ][line width=1.5]    (159,40.5) -- (170,16.5) ;
\draw [shift={(161.33,35.41)}, rotate = 294.62] [color={rgb, 255:red, 0; green, 0; blue, 0 }  ,draw opacity=1 ][line width=1.5]    (11.37,-3.42) .. controls (7.23,-1.45) and (3.44,-0.31) .. (0,0) .. controls (3.44,0.31) and (7.23,1.45) .. (11.37,3.42)   ;
%Straight Lines [id:da6355808452122058] 
\draw [color={rgb, 255:red, 0; green, 0; blue, 0 }  ,draw opacity=1 ][line width=1.5]    (147,67.5) -- (158,41.5) ;
\draw [shift={(149.54,61.5)}, rotate = 292.93] [color={rgb, 255:red, 0; green, 0; blue, 0 }  ,draw opacity=1 ][line width=1.5]    (11.37,-3.42) .. controls (7.23,-1.45) and (3.44,-0.31) .. (0,0) .. controls (3.44,0.31) and (7.23,1.45) .. (11.37,3.42)   ;
%Straight Lines [id:da0720031406237609] 
\draw [color={rgb, 255:red, 0; green, 0; blue, 0 }  ,draw opacity=1 ][line width=1.5]    (195,65.5) -- (159,41.5) ;
\draw [shift={(183.82,58.05)}, rotate = 213.69] [color={rgb, 255:red, 0; green, 0; blue, 0 }  ,draw opacity=1 ][line width=1.5]    (12.79,-3.85) .. controls (8.13,-1.64) and (3.87,-0.35) .. (0,0) .. controls (3.87,0.35) and (8.13,1.64) .. (12.79,3.85)   ;
%Shape: Ellipse [id:dp7761786535418025] 
\draw  [fill={rgb, 255:red, 0; green, 0; blue, 0 }  ,fill opacity=1 ] (155.58,42.97) .. controls (154.17,41.48) and (154.43,39.26) .. (156.16,38.02) .. controls (157.89,36.77) and (160.43,36.97) .. (161.83,38.47) .. controls (163.24,39.96) and (162.98,42.17) .. (161.25,43.42) .. controls (159.52,44.66) and (156.98,44.46) .. (155.58,42.97) -- cycle ;
%Straight Lines [id:da8460112140006701] 
\draw [color={rgb, 255:red, 125; green, 172; blue, 71 }  ,draw opacity=0.77 ][line width=1.5]    (169.46,15.53) -- (147.78,63.06) ;
%Curve Lines [id:da13029221307729988] 
\draw [color={rgb, 255:red, 74; green, 144; blue, 226 }  ,draw opacity=0.8 ][line width=1.5]    (169.6,15.53) .. controls (155.91,45.08) and (162.75,41.22) .. (192.42,64.34) ;
%Shape: Ellipse [id:dp12422556788035166] 
\draw  [fill={rgb, 255:red, 0; green, 0; blue, 0 }  ,fill opacity=1 ] (143.57,69.46) .. controls (142.63,67.8) and (143.2,65.49) .. (144.85,64.3) .. controls (146.5,63.11) and (148.6,63.5) .. (149.55,65.16) .. controls (150.49,66.82) and (149.92,69.12) .. (148.27,70.31) .. controls (146.61,71.5) and (144.51,71.11) .. (143.57,69.46) -- cycle ;
%Shape: Ellipse [id:dp7930668928626334] 
\draw  [fill={rgb, 255:red, 0; green, 0; blue, 0 }  ,fill opacity=1 ] (166.39,18.07) .. controls (165.45,16.41) and (166.03,14.1) .. (167.68,12.92) .. controls (169.33,11.73) and (171.43,12.11) .. (172.37,13.77) .. controls (173.31,15.43) and (172.74,17.74) .. (171.09,18.93) .. controls (169.44,20.11) and (167.34,19.73) .. (166.39,18.07) -- cycle ;
%Shape: Ellipse [id:dp5643624475844304] 
\draw  [fill={rgb, 255:red, 0; green, 0; blue, 0 }  ,fill opacity=1 ] (191.5,67.9) .. controls (190.56,66.24) and (191.13,63.94) .. (192.78,62.75) .. controls (194.43,61.56) and (196.53,61.95) .. (197.48,63.6) .. controls (198.42,65.26) and (197.84,67.57) .. (196.19,68.76) .. controls (194.54,69.95) and (192.44,69.56) .. (191.5,67.9) -- cycle ;
%Straight Lines [id:da15322810297376077] 
\draw    (90.1,45.28) -- (116.28,45.28) ;
\draw [shift={(118.28,45.28)}, rotate = 180] [color={rgb, 255:red, 0; green, 0; blue, 0 }  ][line width=0.75]    (10.93,-3.29) .. controls (6.95,-1.4) and (3.31,-0.3) .. (0,0) .. controls (3.31,0.3) and (6.95,1.4) .. (10.93,3.29)   ;
%Straight Lines [id:da9086101712362538] 
\draw [color={rgb, 255:red, 125; green, 172; blue, 71 }  ,draw opacity=1 ][line width=1.5]    (27.51,63.76) -- (48.51,20.76) ;
\draw [shift={(34.15,50.17)}, rotate = 296.03] [color={rgb, 255:red, 125; green, 172; blue, 71 }  ,draw opacity=1 ][line width=1.5]    (14.21,-4.28) .. controls (9.04,-1.82) and (4.3,-0.39) .. (0,0) .. controls (4.3,0.39) and (9.04,1.82) .. (14.21,4.28)   ;
%Straight Lines [id:da8927048163100975] 
\draw [color={rgb, 255:red, 74; green, 144; blue, 226 }  ,draw opacity=1 ][line width=1.5]    (71.51,63.76) -- (48.51,20.76) ;
\draw [shift={(64.16,50.02)}, rotate = 241.86] [color={rgb, 255:red, 74; green, 144; blue, 226 }  ,draw opacity=1 ][line width=1.5]    (14.21,-4.28) .. controls (9.04,-1.82) and (4.3,-0.39) .. (0,0) .. controls (4.3,0.39) and (9.04,1.82) .. (14.21,4.28)   ;
%Shape: Ellipse [id:dp5163903870285347] 
\draw  [fill={rgb, 255:red, 0; green, 0; blue, 0 }  ,fill opacity=1 ] (45.31,22.41) .. controls (44.35,20.86) and (44.94,18.7) .. (46.62,17.58) .. controls (48.31,16.47) and (50.45,16.83) .. (51.41,18.38) .. controls (52.38,19.94) and (51.79,22.1) .. (50.1,23.22) .. controls (48.42,24.33) and (46.27,23.97) .. (45.31,22.41) -- cycle ;
%Shape: Ellipse [id:dp9975957422725217] 
\draw  [fill={rgb, 255:red, 0; green, 0; blue, 0 }  ,fill opacity=1 ] (68.61,65.76) .. controls (67.65,64.2) and (68.24,62.04) .. (69.92,60.93) .. controls (71.61,59.81) and (73.75,60.17) .. (74.71,61.73) .. controls (75.67,63.28) and (75.09,65.45) .. (73.4,66.56) .. controls (71.72,67.67) and (69.57,67.31) .. (68.61,65.76) -- cycle ;
%Shape: Ellipse [id:dp3865649819592216] 
\draw  [fill={rgb, 255:red, 0; green, 0; blue, 0 }  ,fill opacity=1 ] (24.34,65.76) .. controls (23.38,64.2) and (23.97,62.04) .. (25.65,60.93) .. controls (27.34,59.81) and (29.48,60.17) .. (30.44,61.73) .. controls (31.41,63.28) and (30.82,65.45) .. (29.13,66.56) .. controls (27.45,67.67) and (25.3,67.31) .. (24.34,65.76) -- cycle ;

% Text Node
\draw (143.52,11.91) node [anchor=north west][inner sep=0.75pt]  [font=\normalsize]  {$e_{0} '$};
% Text Node
\draw (126.85,38.87) node [anchor=north west][inner sep=0.75pt]  [font=\normalsize]  {$e_{0} ''$};
% Text Node
\draw (184.46,36.03) node [anchor=north west][inner sep=0.75pt]  [font=\normalsize]  {$e_{1} '$};
% Text Node
\draw (21,18.4) node [anchor=north west][inner sep=0.75pt]  [font=\normalsize,color={rgb, 255:red, 65; green, 117; blue, 5 }  ,opacity=1 ]  {$e_{0}$};
% Text Node
\draw (60.38,18.81) node [anchor=north west][inner sep=0.75pt]  [font=\normalsize,color={rgb, 255:red, 56; green, 109; blue, 172 }  ,opacity=1 ]  {$e_{1}$};

\end{tikzpicture}
  \end{wrapfigure}
\noindent (iii) (Partial Fold) Let $\G'$ be the graph with $\mathcal{V}\G' = \mathcal{V}\G \cup \{v'\}$ and $\mathcal{E}\G'~=~(\mathcal{E}\G~-~\{e_0,e_1\})~\cup~\{e_0',e_0'',e_1'\}$, where $e_0'$ joins $\iota(e_0)$ to $v'$, $e_0''$ joins $v'$ to $\tau(e_0)$, and $e_1'$ joins $v'$ to $\tau(e_1)$. Let $f: \G \rightarrow \G'$ be given by:
$$\hspace{-7cm} f(e) = \begin{cases}e_0'e_0'' & \text{if }e=e_0 \\ e_0'e_1' & \text{if }e=e_1 \\ e & \text{otherwise} \end{cases}$$
$f$ is called the \textit{partial fold of $e_1$ over $e_0$.} Equivalently,  subdivide $e_0 \in \mathcal{E}\G$: let $v'$ be a new vertex in the middle of $e_0$ and relabel $e_0$ as two edges $e_0'$ and $e_0''$, oriented so that $e_0$ is now equal to the edge path $e_0'e_0''$.  Subdivide $e_1 \in \mathcal{E}\G$: let $v''$ be a new vertex in the middle of $e_1$ and relabel $e_1$ as two edges $e_1''$ and $e_1'$, oriented so that $e_1'$ is now equal to the edge path $e_1''e_1'$. Now, let $\G' = \G / e_1'' \sim e_0'$, and let $f:\G \rightarrow \G'$ be the quotient map. \\

%\vspace{-0.2cm}

\begin{thm}[\cite{s83}]\label{stallings} Every surjective homotopy equivalence graph map $f:\G \rightarrow \G'$ can be decomposed as $f = h \circ f_m \circ \dots \circ f_2 \circ f_1$ where $\Gamma_1 = \Gamma$, each $f_i : \Gamma_i \rightarrow \Gamma_{i+1}$ is a fold, and $h:\Gamma_{m+1} \rightarrow \Gamma'$ is an graph isomorphism. \end{thm}

\noindent In particular, i.h.e. maps are surjective, and thus have such a fold decomposition. For instance, Example \ref{ahlp single fold} can be decomposed as a single proper full fold of $c_2$ over $\overline{b_1}$ and a graph isomorphism:\\

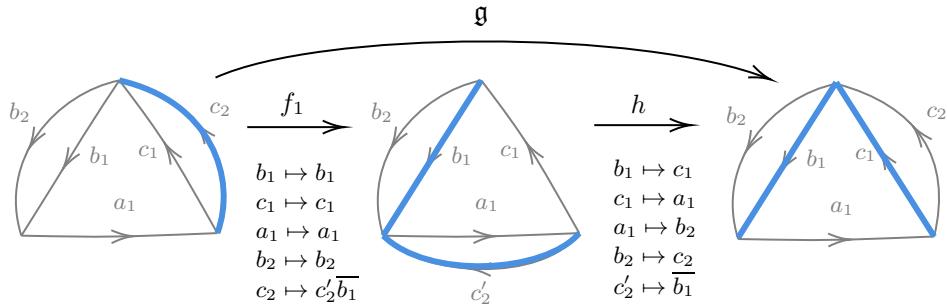
\begin{center}
\input{figures/single_fold_decomp}
\end{center}
% \begin{small} $$f_1(e):= \begin{cases} 
%\overline{b_1}c_2' & \text{if } e = c_2 \\
%e & \text{otherwise}
%\end{cases}
% $$ \end{small}
%\noindent and a graph isomorphism: \begin{small} $$h: \begin{cases} 
%b_1 & \mapsto c_1 \\
%c_1 & \mapsto a_1 \\
%a_1 & \mapsto b_2 \\
%b_2 & \mapsto c_2 \\
%c_2' & \mapsto \overline{b_1}.
%\end{cases}\\
% $$ \end{small}

%\vspace{0.6cm}
%\new page
\vspace{0.1cm}
\noindent We collect some known observations in the following lemma. 

\begin{lem}\label{helpful facts} Suppose $f:\G \rightarrow \G$ is an i.h.e. graph map with fold decomposition consisting of $m$ folds and a graph isomorphism $h:\G' \rightarrow \G$. Let $\lambda_f$ denote the greatest eigenvalue of $T(f)$ in modulus. Then there is a choice of positive length $\ell$ on each edge in $\G$ such that for every $e \in \mathcal{E}\G$, we have $\ell(f(e))=\lambda_f \ell(e)$ where $\ell(u) := \sum_{i=1}^k \ell(b_i)$ when $u=b_1b_2 \dots b_k$ is an edge path. Moreover, the following are equivalent: 
\begin{itemize}
\item[(i)] $m = 0$,
\item[(ii)]  there is a power $n \in \mathbb{N}$ such that $f^n$ is the identity on $\G$,
\item[(iii)] $\lambda_f = 1$,
\item[(iv)] $f$ is not expanding.

\end{itemize}
\end{lem}

\begin{proof}
 Suppose $T(f)$ is the transition matrix of $f$ with respect to an edge ordering $(e_1,\dots,e_n)$. Since $T(f)$ is irreducible, the Perron--Frobenius Theorem guarantees there is a left eigenvector $\vec{v}$ with positive entries such that $\vec{v} \ T(f)=\lambda_f \vec{v}$. Use the entries of $\vec{v}= [v_1,\dots,v_n]$ to assign the length $v_i$ to the corresponding edge $e_i$. Letting $a_i^1,\dots,a_i^n$ denote the entries of the $i-$th column of $T(f)$, we have
 \begin{align*}
 \ell(f(e_i)) & = \sum_{j=1}^{n} a_i^j \ell(e_j) \\
 & = \sum_{j=1}^{n} a_i^j v_j \\
 & = \lambda_f v_i.
 \end{align*}
 Hence $\ell(f(e)) = \lambda_f \ell(e)$ for each $e \in \mathcal{E}\G$. 
\begin{itemize} 
\item[(i) $\Rightarrow$ (ii):]  Suppose $m=0$. Then $f$ is a graph isomorphism and hence a bijection on the set of oriented edges of $\G$.  Thus there is a power $n$ such that $f^n$ is equal to the identity. \\

\item[(ii) $\Rightarrow$ (iii):] If $f^n$ is the identity, then  $(\lambda_f)^n=1$, so $|\lambda_f|=1$. The Perron--Frobenius theorem guarantees $\lambda_f$ is real, positive and greater than or equal to $1$. Thus $\lambda_f =1$. \\

\item[(iii) $\Rightarrow$ (iv):] Now suppose $\lambda_f=1$. Thus $\ell(f^n(e)) = \ell(e)$ for each $e \in \mathcal{E}\G$ and power $n \in \mathbb{N}$. Since the length of each edge is positive, $|f^n(e)|$ is bounded from above for all $n \in \mathbb{N}$. Hence $f$ is not expanding.\\

\item[(iv) $\Rightarrow$ (i):] Proceeding by contrapositive, suppose $m >0$.  If the fold decomposition consisted of only complete folds, then $|\mathcal{V}\G'|<|\mathcal{V}\G|$, contradicting that $h:\G' \rightarrow \G$ is a graph isomorphism. Thus there is at least one fold which is a proper full fold or a partial fold, and hence some edge $b \in \mathcal{E}\G$ with $|f(b)| > 1$.  Let $e \in \mathcal{E}\G$ be any edge. Since $f$ is irreducible, there is a power $k$ such that $f^{k}(e)$ traverses $b$, and a power $p$ such that $f^p(b)$ traverses $b$. Hence $f^{np}(f^k(e))$ traverses $b$ for each $n \in \mathbb{N}$. Since $|f(b)|>1$, we have  $|f^{np+k+1}(e)| >|f^{np+k}(e)|$ for each $n \in \mathbb{N}$. Since  $|f(u)| \geq |u|$ for any edge path $u$,
$$\{|f^{n}(e)|\}_{n=1}^{\infty}$$
is a non-decreasing sequence of integers which strictly increases for each $n \equiv k+1$ mod $p$. Therefore $|f^{n}(e)| \rightarrow \infty $ and hence $f$ is expanding.
\end{itemize} \end{proof}

\bigskip

\section{Folds and Mixing}\label{folds and mixing}

The following lemmas relating folds, mixing edges, and stacks will provide key facts for our lower bound and symmetry results. %For the duration of this section, suppose $f:\G \rightarrow \G$ is an irreducible expanding self graph map.

%\begin{obs}\label{injective} Suppose $f:\G \rightarrow \G$ is an irreducible self graph map. If a pair of distinct edges $a,b \in \mathcal{E}\G$ has $f(a)\in \{ f(b), \overline{f(b)} \}$, then the range of  \ $T(f)~:~\mathbb{R}^n~\rightarrow~\mathbb{R}^n$ is a non-trivial proper invariant subspace, contradicting that $T(f)$ is irreducible. Hence $f(a)\notin \{ f(b), \overline{f(b)} \}$.
%\end{obs}

%could state the following lemma for irreducible expanding graph maps. 

\begin{lem}\label{stack format} Suppose $f:\G \rightarrow \G$ is an expanding i.h.e. map.  Then each stack of $f$ has the form $\mathcal{K}=\{e, f(e), f^2(e), \dots, f^s(e)\}$ with only $f^s(e)$ either mixing or surplus. %with one of the following two properties: 
%\begin{itemize}
%\item[(i)] $f^s(e)$ is a mixing edge, or
%\item[(ii)] $f^s(e)$ is non-mixing and there is another edge $u \neq f^s(e)$ such that $f(u)=f^{s+1}(e)$. 
%\end{itemize}
%Moreover, if $f$ is periodic on the vertex set of $\G$, then the final edge in every stack is mixing.
\end{lem}

\begin{proof} 

\noindent Let $\mathcal{K}$ be a stack of $f$ and suppose $e \in \mathcal{K}$.  If $f^t(e)$ is non-mixing and non-surplus for all $0 \leq t \leq k$, then 
 $$\{e, f(e), \dots, f^k(e), f^{k+1}(e)\} \subseteq \mathcal{K}.$$ By the definition of a stack, these edges are distinct as unoriented edges, except possibly $f^{k+1}(e)\in\{e, \overline{e}\}$. Suppose $f^{k+1}(e)\in \{e ,\overline{e}\}$.  Then for any $b \in \{e,f(e),\dots,f^k(e)\}$, we have $f^n(b)$ or $f^n(\overline{b})$ is an edge in this same set. By irreducibility of $T(f)$, we must have $$\{e,f(e),\dots,f^k(e)\}= \mathcal{E}\G.$$ Thus $T(f)$ is a permutation matrix, so $\lambda_f =1$. By Lemma \ref{helpful facts}, this contradicts that $f$ is expanding. Thus $f^{k+1}(e) \notin \{e,\overline{e}\}$. \\%Theorem \ref{helpful facts}. 
 
 Since $\mathcal{E}\G$ is finite, eventually there is a first power $s$ such that $f^s(e)$ is either mixing or surplus. Suppose  $\mathcal{K}~-~\{e, f(e), \dots, f^{s}(e)\}~\neq~\emptyset$. Then there must be an edge $e'$ such that $f(e')= e$. Thus % for some $0 \leq t \leq s$. If $t >0$, then $f^{t-1}(e) = $ has property (ii), which is a contradiction. Hence $t=0$ and $f(e')=e$. Thus 
$$\{e', f(e'), f^2(e'), \dots,  f^{s+1}(e')\} \subseteq \mathcal{K}.$$
Once again, if $\mathcal{K} - \{e', f(e'), \dots, f^{s+1}(e')\} \neq \emptyset$, there is a $e''$ such that $f(e'')=e'$, so 
$$\{e'', f(e''), f^2(e''), \dots,  f^{s+2}(e'')\} \subseteq \mathcal{K}.$$
Since $\mathcal{E}\G$ is finite, this process eventually terminates, so $\mathcal{K}$ has the desired format.  \end{proof} 

\vspace{0.2cm}
 
 \begin{df}\label{root, final} (Root Edge, Final Edge) Given a stack $\mathcal{K} = \{e, f(e), f^2(e), \dots, f^s(e)\}$, we call $e$ the \textit{root edge} of $\mathcal{K}$ and $f^s(e)$ the \textit{final edge} of $\mathcal{K}$. 
 \end{df}

%\begin{df} Consider the number of extra edges in the image of $f$, $ex(f)$. 
%$$ex(f) := \Big{(} \sum_{e \in E\G} \text{word length of }f(e)\Big{)} - |\mathcal{E}\G|$$
%\end{df}

\begin{lem}\label{folds, edge images}
Suppose $f:\Gamma \rightarrow \Gamma$ is an expanding i.h.e. map with fold decomposition consisting of $m$ total folds and $p$ total stacks. Then
\begin{align} m \leq \sum_{e \in \mathcal{E}\G} \Big{(} |f(e)|-1\Big{)}. \label{eq:1} \end{align}
Moreover, if $f$ is periodic on the vertices of $\G$, then $p \leq m$. 
\end{lem}

\begin{proof}  Write $f = h \circ f_m \circ \dots \circ f_2 \circ f_1$ where $\G_1 = \G$, each $f_i: \G_i \rightarrow \G_{i+1}$ is a fold and $h:\G_{m+1} \rightarrow \G$ is a graph isomorphism. To keep track of the number of edges in the image as each fold $f_i$ is applied, let $T_0=0$ and 
$$T_i = \sum_{e \in \mathcal{E}\G} \Big{(} |(f_i \circ \dots \circ f_1) (e)|-1\Big{)}.$$

\noindent \textit{Claim:
\begin{itemize}
\item[(i)] If $f_i$ is a proper full fold, then $T_i \geq 1+T_{i-1}$ and $|\mathcal{V}\G_{i+1}| = |\mathcal{V}\G_{i}|$. \vspace{0.1cm}
\item[(ii)] If $f_i$ is a complete fold, then $T_i = T_{i-1}$  and $|\mathcal{V}\G_{i+1}| = |\mathcal{V}\G_{i}| -1$. \vspace{0.1cm}
\item[(iii)] If $f_i$ is a partial fold, then $T_i  \geq 2+T_{i-1}$  and $|\mathcal{V}\G_{i+1}| = |\mathcal{V}\G_{i}|+1$. \\
\end{itemize}}

\noindent Assuming the claim for now, we have $$T_m \geq \text{(number of proper full folds)} + 2\text{(number of partial folds)}$$
and
$$|V\G_{m+1}| = |V\G| + \text{(number of partial folds)} - \text{(number of complete folds)}.$$

Since $h: \Gamma_{m+1} \rightarrow \G$ is a graph isomorphism, $|\mathcal{V}\G_{m+1}| = |\mathcal{V}\G|$, so the number of complete folds must be equal to the number of partial folds. Further, for any edge path $u$,  we have $|h(u)|=|u|$, again since $h$ is a graph isomorphism. Therefore
\begin{align*}
 \sum_{e \in \mathcal{E}\G} \Big{(} |f(e)|-1\Big{)} & = T_m   \\ 
& \geq  \text{(number of proper full folds)} + 2\text{(number of partial folds)} \\
  & =  \text{(number of proper full folds)} + \text{(number of partial folds)} \\
  & \ \ \ \ \ \  \   + \text{(number of complete folds)} \\
  & = m.
\end{align*}
This completes the proof of equation \eqref{eq:1}. We now move on to proving claims (i), (ii) and (iii) and subsequently prove the statement that if $f$ is periodic on the vertices of $\G$, then $p \leq m$.\\

\noindent \textit{Proof of Claim (i):} Suppose $f_i: \G_i \rightarrow \G_{i+1}$ is a proper full fold of $e_1$ over $e_0$. By definition,  $|\mathcal{V}\G_{i+1}|=|\mathcal{V}\G_{i}|$ and
$$f_i(e) = \begin{cases}
e_0'e_1 & e = e_1 \\
e & \text{otherwise} \end{cases}$$
Let $u \in \mathcal{E}\G$. If  $(f_{i-1} \circ \dots \circ f_1)(u)$ traverses $e_1$ a total of $k$ times, then $|(f_{i} \circ \dots \circ f_1)(u)| = |(f_{i-1} \circ \dots \circ f_1)(u)| + k$. Since each $f_j$ is surjective, there must be at least one $u$ with $k >0$. Hence $T_i \geq T_{i-1} +1$. \hfill$\diamond$\\

\noindent \textit{Proof of Claim (ii):} Suppose $f_i: \G_i \rightarrow \G_{i+1}$ is a complete fold of $e_1$ and $e_0$. Since $f$ is a homotopy equivalence, $\tau(e_0) \neq \tau(e_1)$. Thus $|\mathcal{V}\G_{i+1}|=|\mathcal{V}\G_{i}|-1$. By definition,
$$f_i(e) = \begin{cases}e_0' & e \in \{e_0,e_1\} \\
e & \text{otherwise} \end{cases}$$
For all $u \in \mathcal{E}\G$, we have $|(f_{i} \circ \dots \circ f_1)(u)| = |(f_{i-1} \circ \dots \circ f_1)(u)| $, so $T_i=T_{i+1}$. \hfill $\diamond$\\

\noindent \textit{Proof of Claim (iii):} Suppose $f_i: \G_i \rightarrow \G_{i+1}$ is a partial fold of $e_1$ over $e_0$. By definition, $|\mathcal{V}\G_{i+1}|=|\mathcal{V}\G_{i}|+1$ and
$$f_i(e) = \begin{cases}e_0' e_0''& e=e_0 \\
e_0'e_1' & e = e_1 \\
e & \text{otherwise} \end{cases}$$
Let $u \in \mathcal{E}\G$. If $(f_{i-1} \circ \dots \circ f_1)(u)$ traverses $e_0$ and $e_1$ a total of $k$ times, then $|(f_{i} \circ \dots \circ f_1)(u)| = |(f_{i-1} \circ \dots \circ f_1)(u)| + k$. Since each $f_j$ is surjective, there must be at least one $u$ with $(f_{i-1} \circ \dots \circ f_1)(u)$ traversing $e_0$ at least once, and at least one $u$ with $(f_{i-1} \circ \dots \circ f_1)(u)$ traversing $e_1$ at least once. Hence $T_i \geq T_{i-1} +2$.  \hfill$\diamond$\\

\vspace{0.1cm}
Now, suppose $f$ is periodic on the vertices of $\G$. Suppose distinct edges $e_1,e_2 \in \mathcal{E}\G$ are surplus and $f(e_1)=f(e_2)$. Since $f$ is a bijection on the vertices, we must have $\iota(e_1)=\iota(e_2)$ and $\tau(e_1)=\tau(e_2)$. Hence $e_1 \overline{e_2}$ is a closed loop $\G$ which is not null-homotopic.  However, $f(e_1\overline{e_2})=f(e_1)\overline{f(e_1)}$ is null-homotopic, contradicting that $f$ is a homotopy equivalence. Therefore there are no surplus edges, and hence by Lemma \ref{stack format}, the final edge in each stack is mixing. Let $\alpha_1, \dots, \alpha_p$ denote these final mixing edges. We will make an assignment of each $\alpha_k$ to a fold $f_{i_k}$ in the following way:\\

Recursively label $f_i(\alpha_k)$ as $\alpha_k \in \mathcal{E}\G_{i+1}$ whenever $|f_i(\alpha_k)|=1$. This agrees with the labelling determined in Definition \ref{folds}. If $\alpha_k$ nor $\overline{\alpha_k}$ is never properly folded over an edge, nor involved in a partial fold, then $|f(\alpha_k)|=1$ contradicting that $\alpha_k$ is mixing. Thus, possibly replacing $\alpha_k$ with $\overline{\alpha_k}$, there must exist a first fold $f_{i_k}$ and some $e_0 \in \mathcal{E}\G_{i_k}$ such that 
\begin{itemize}
\item[(i)] $f_{i_k}$ is a proper full fold of $\alpha_k$ over $e_0$ and $f_{i_k}(\alpha_k)=\alpha_k'e_0$, or
\item[(ii)] $f_{i_k}$ is a partial fold of $\alpha_k$ over $e_0$ and $f_{i_k}(\alpha_k)=\alpha_k'e_0'$, or
\item[(iii)] $f_{i_k}$ is a partial fold of $e_0$ over $\alpha_k$ and $f_{i_k}(\alpha_k)=e_0''e_0$.\\
\end{itemize}
%If $f_{i_k}$ falls into case (i), then for any $j \neq k$, $i_k \neq i_j$, as follows. Suppose $i_k=i_j$.  Then $f_{i_k}$ is case (ii) and $f_{i_j}$  case (iii) (or vice versa). If $t \notin \{k,j\}$, $i_t \neq i_k$. 

To each proper full fold, either one or zero mixing edges are assigned. To each partial fold, either two, one, or zero mixing edges are assigned.  As argued above, the number of partial folds is equal to the number of complete folds. Since all $p$ mixing edges are assigned to some proper full fold or partial fold, there are at least $p$ folds.  
\end{proof}

%\begin{df} Given a graph $\Gamma$, we call a graph $G$ a \textit{supergraph} of $\Gamma$ if $\Gamma$ is a subgraph of $G$ and $\mathcal{V}(G)=V(\Gamma)$. The \textit{symmetry score of $\Gamma$} is 
%$$Sym(\Gamma) = \min \{\text{number of edge orbits of }\varphi  \ | \ G \text{ is a supergraph of }\Gamma \text{ and }\varphi \in \text{Aut}(G) \}$$
%\end{df} 

\bigskip

\section{Stack Graphs}\label{stack graphs}

To prove Theorem \hyperref[lower bound]{A} , we develop a tool called the stack graph to measure how the stacks of a graph map interact with each other. Alternatively, combining Lemma \ref{folds, edge images} with Lemma \ref{hamsong} (\cite{hamsong}) yields a proof of Theorem \hyperref[lower bound]{A} for i.h.e. maps with primitive transition matrices, which avoids the need for stack graphs. % One could skip this section and proceed to proof 2 of Theorem \hyperref[lower bound]{A}. 

\vspace{0.1cm}
For the duration of this section, let  $f: \G \rightarrow \G$ be an irreducible expanding %homotopy equivalence 
self graph map with stacks $\mathcal{K}_1, \dots ,\mathcal{K}_p$. For each $1 \leq i \leq p$, let $n_i$ be the number of edges in stack $\mathcal{K}_i$ and $\alpha_i$ the final edge in stack $\mathcal{K}_i$. Let $n$ be the total number of edges in $\G$ and note that $n= \sum_{i=1}^p n_i$.  

 \begin{df}\label{stack graph} (Stack Graph, Weight $\omega$)  The \textit{stack graph of $f$}, denoted $\mathcal{SG}(f)$, is a directed graph with vertex set $\mathcal{V} (\mathcal{SG}(f))=\{\mathcal{K}_1, \dots,\mathcal{K}_p\}$ and directed edges:
 $$\mathcal{E}^+\mathcal{SG}(f) = \{[\mathcal{K}_i,\mathcal{K}_j] \ | \ f(\alpha_{i}) \text{ contains an edge in }\mathcal{K}_j\}.$$
 \end{df}
 \noindent We assign a weight $\omega$ to the vertices of $\mathcal{SG}(f)$:
  $$\omega(\mathcal{K}_i) := |f(\alpha_{i})|-1$$
  Note that $\omega(\mathcal{K}_i)=0$ if and only if the final edge $\alpha_i$ is surplus, instead of mixing.  
\begin{obs}\label{weight} Any non-final edge $e$ is non-mixing, and hence has $|f(e)|=1$. When $f$ is an expanding i.h.e. map, by Lemma \ref{folds, edge images} we have
\vspace{-0.2cm}
\begin{align*}\sum_{j=1}^p \omega(\mathcal{K}_j) \ & =  \ \ \sum_{j=1}^{p}\Big{(} |f(\alpha_{j})| - 1\Big{)} \\
& = \sum_{e \in \mathcal{E}\G}\Big{(} |f(e)| - 1\Big{)}\ \geq m,
\end{align*}
where $m$ is the number of folds in the fold decomposition of $f$. 
\end{obs}
\begin{df} (Length $s$, Directed ball of size $d$) We assign a length $s$ to the edges of $\mathcal{SG}(f)$:
 $$s([\mathcal{K}_i,\mathcal{K}_j]) := \min\{s \ | \ f^s(\alpha_i) \text{ traverses } \alpha_j\}$$%n_j +1 -  \max\{q \ | \ e_j^q \text{ or } \overline{e_j^q} \text{ appears in }f(e_{i})\}. $$
\noindent Observe that by definition of  $\mathcal{E}^+\mathcal{SG}(f)$, $s([\mathcal{K}_i,\mathcal{K}_j]) \leq n_j$. For any number $d$ and $\mathcal{K}_i \in \mathcal{V}(\mathcal{SG}(f))$, let the \textit{directed ball of size $d$ at $\mathcal{K}_i$}, be
 $$B_d(\mathcal{K}_i) = \{\mathcal{K}_j \in  \mathcal{V}(\mathcal{SG}(f)) \ |\ \text{there is a directed edge path }P\text{ in }\mathcal{SG}(f) \text{ from }\mathcal{K}_i \text{ to }\mathcal{K}_j \text{ with }s(P) \leq d\},$$
where $P=E_1\dots E_k$ must only traverse edges with positive orientation and $s(P) := \sum_{i=1}^k s(E_i)$
\end{df}

\vspace{0.2cm}
\begin{multicols}{2}
\begin{ex}\label{stack graph example} Consider the irreducible expanding self graph map $\mathfrak{f}:\G \rightarrow \G$, written in stack format to the right. \\ %$\mathfrak{f}$ is not a homotopy equivalence, but for example-sake, one can still compute $\mathcal{SG}(\mathfrak{f})$. 

Below and to the right is the stack graph of $\mathfrak{f}$, $\mathcal{SG}(\mathfrak{f})$ with length of edges labeled, and the weight of each vertex in $\mathcal{SG}(\mathfrak{f})$. For example, $a_3$ is the final edge in stack $a$ and $$\mathfrak{f}^3(a_3) =b_3a_3d_1b_3a_3b_4b_1$$ contains the final edge in stacks $a$, $b$, and $d$. There are directed paths of length $3$ in $\mathcal{SG}(\mathfrak{f})$ from $a$ to $a$, $b$, and $d$. In contrast, there is no directed path of length $3$ from $a$ to $c$.  

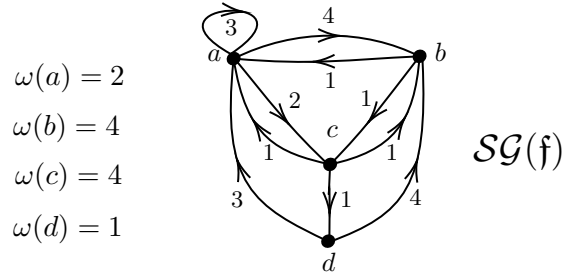
\columnbreak 
\textcolor{white}{\ } \\
\vspace{-3cm}
\tikzset{every picture/.style={line width=0.75pt}} %set default line width to 0.75pt        

\hspace{-0.7cm}\begin{tikzpicture}[x=0.75pt,y=0.75pt,yscale=-1,xscale=1]
%uncomment if require: \path (0,170); %set diagram left start at 0, and has height of 170

%Curve Lines [id:da3089650333839902] 
\draw    (75.4,20.75) .. controls (34.9,98.16) and (75.94,21.31) .. (34.69,97.22) ;
\draw [shift={(58.58,53.07)}, rotate = 116.98] [color={rgb, 255:red, 0; green, 0; blue, 0 }  ][line width=0.75]    (10.93,-3.29) .. controls (6.95,-1.4) and (3.31,-0.3) .. (0,0) .. controls (3.31,0.3) and (6.95,1.4) .. (10.93,3.29)   ;
%Curve Lines [id:da5677289640472669] 
\draw    (34.69,97.22) .. controls (116.12,95.54) and (33.19,97.29) .. (116.12,95.97) ;
\draw [shift={(68.46,96.53)}, rotate = 359.04] [color={rgb, 255:red, 0; green, 0; blue, 0 }  ][line width=0.75]    (10.93,-3.29) .. controls (6.95,-1.4) and (3.31,-0.3) .. (0,0) .. controls (3.31,0.3) and (6.95,1.4) .. (10.93,3.29)   ;
%Curve Lines [id:da2884436359634339] 
\draw    (116.12,95.97) .. controls (75.08,21.31) and (115.27,95.54) .. (75.4,20.75) ;
\draw [shift={(98.73,64.46)}, rotate = 240.73] [color={rgb, 255:red, 0; green, 0; blue, 0 }  ][line width=0.75]    (10.93,-3.29) .. controls (6.95,-1.4) and (3.31,-0.3) .. (0,0) .. controls (3.31,0.3) and (6.95,1.4) .. (10.93,3.29)   ;
%Curve Lines [id:da7211858481715636] 
\draw    (75.4,20.75) .. controls (105.86,38.78) and (114.41,56.24) .. (116.12,95.97) ;
\draw [shift={(103.96,45.83)}, rotate = 60.05] [color={rgb, 255:red, 0; green, 0; blue, 0 }  ][line width=0.75]    (10.93,-3.29) .. controls (6.95,-1.4) and (3.31,-0.3) .. (0,0) .. controls (3.31,0.3) and (6.95,1.4) .. (10.93,3.29)   ;
%Curve Lines [id:da5175010511636169] 
\draw [line width=0.75]    (34.29,96.63) .. controls (-3.57,105.15) and (23.78,65.85) .. (34.59,95.28) ;
%Curve Lines [id:da6914867823362085] 
\draw    (35.88,98.39) .. controls (37.46,130.47) and (-1.01,106.02) .. (34.29,96.63) ;
%Curve Lines [id:da7697534488709188] 
\draw [line width=0.75]    (75.37,22.27) .. controls (76.7,-17.3) and (105.86,17.82) .. (76.57,22.91) ;
%Curve Lines [id:da9321895365866053] 
\draw    (73.3,23.41) .. controls (42.48,17.01) and (75.24,-14.98) .. (75.37,22.27) ;
%Curve Lines [id:da5538211643767741] 
\draw [line width=0.75]    (116.09,95.48) .. controls (142.66,124.32) and (95.6,120.47) .. (114.8,95.9) ;
%Curve Lines [id:da832153166824444] 
\draw    (116.76,93.19) .. controls (142.91,75.35) and (142.08,121.62) .. (116.09,95.48) ;
%Shape: Ellipse [id:dp4078386056516199] 
\draw  [fill={rgb, 255:red, 0; green, 0; blue, 0 }  ,fill opacity=1 ] (73.65,22.95) .. controls (73.06,21.97) and (73.42,20.6) .. (74.46,19.9) .. controls (75.5,19.19) and (76.83,19.42) .. (77.43,20.4) .. controls (78.02,21.39) and (77.66,22.75) .. (76.62,23.46) .. controls (75.57,24.16) and (74.25,23.93) .. (73.65,22.95) -- cycle ;
%Shape: Ellipse [id:dp9103076850994045] 
\draw  [fill={rgb, 255:red, 0; green, 0; blue, 0 }  ,fill opacity=1 ] (32.91,97.88) .. controls (32.32,96.9) and (32.68,95.53) .. (33.72,94.83) .. controls (34.76,94.13) and (36.09,94.35) .. (36.69,95.34) .. controls (37.28,96.32) and (36.92,97.69) .. (35.88,98.39) .. controls (34.83,99.1) and (33.51,98.87) .. (32.91,97.88) -- cycle ;
%Shape: Ellipse [id:dp5194164540622117] 
\draw  [fill={rgb, 255:red, 0; green, 0; blue, 0 }  ,fill opacity=1 ] (113.84,96.31) .. controls (113.24,95.32) and (113.6,93.96) .. (114.65,93.25) .. controls (115.69,92.55) and (117.02,92.78) .. (117.61,93.76) .. controls (118.21,94.74) and (117.84,96.11) .. (116.8,96.81) .. controls (115.76,97.52) and (114.43,97.29) .. (113.84,96.31) -- cycle ;
%Curve Lines [id:da8474042945792992] 
\draw [line width=0.75]    (134,72.5) .. controls (402,193.5) and (406,-76.5) .. (131,38.5) ;
\draw [shift={(131,38.5)}, rotate = 337.31] [color={rgb, 255:red, 0; green, 0; blue, 0 }  ][line width=0.75]    (10.93,-3.29) .. controls (6.95,-1.4) and (3.31,-0.3) .. (0,0) .. controls (3.31,0.3) and (6.95,1.4) .. (10.93,3.29)   ;

% Text Node
\draw (67.91,79) node [anchor=north west][inner sep=0.75pt]  [font=\footnotesize]  {$a_{1}$};
% Text Node
\draw (80.79,57.09) node [anchor=north west][inner sep=0.75pt]  [font=\footnotesize]  {$a_{3}$};
% Text Node
\draw (56.3,56.59) node [anchor=north west][inner sep=0.75pt]  [font=\footnotesize]  {$a_{2}$};
% Text Node
\draw (107.84,33.65) node [anchor=north west][inner sep=0.75pt]  [font=\footnotesize]  {$d_{1}$};
% Text Node
\draw (0.11,84.82) node [anchor=north west][inner sep=0.75pt]  [font=\footnotesize]  {$b_{2}$};
% Text Node
\draw (7.05,99.51) node [anchor=north west][inner sep=0.75pt]  [font=\footnotesize]  {$c_{1}$};
% Text Node
\draw (90.51,5.19) node [anchor=north west][inner sep=0.75pt]  [font=\footnotesize]  {$b_{3}$};
% Text Node
\draw (46.52,4.32) node [anchor=north west][inner sep=0.75pt]  [font=\footnotesize]  {$c_{2}$};
% Text Node
\draw (135.39,83.9) node [anchor=north west][inner sep=0.75pt]  [font=\footnotesize]  {$b_{1}$};
% Text Node
\draw (128.56,103.75) node [anchor=north west][inner sep=0.75pt]  [font=\footnotesize]  {$b_{4}$};
% Text Node
\draw (177,24.4) node [anchor=north west][inner sep=0.75pt]  [font=\footnotesize]  {$a_{1} \mapsto a_{2} \mapsto a_{3} \mapsto b_{1} a_{1} c_{1} \ $};
% Text Node
\draw (151,44.4) node [anchor=north west][inner sep=0.75pt]  [font=\footnotesize]  {$b_{1} \mapsto b_{2} \mapsto b_{3} \mapsto b_{4} \mapsto c_{1} a_{2} c_{2} a_{3} a_{1}$};
% Text Node
\draw (210,63.4) node [anchor=north west][inner sep=0.75pt]  [font=\footnotesize]  {$c_{1} \mapsto c_{2} \mapsto d_{1} b_{3} a_{3} b_{4} b_{1}$};
% Text Node
\draw (240,79.4) node [anchor=north west][inner sep=0.75pt]  [font=\footnotesize]  {$d_{1} \mapsto \overline{a_{1}} b_{1}$};
% Text Node
\draw (344,45.4) node [anchor=north west][inner sep=0.75pt]  [font=\large]  {$\mathfrak{f}$};
% Text Node
\draw (7,45.4) node [anchor=north west][inner sep=0.75pt]  [font=\Large]  {$\Gamma $};

\end{tikzpicture}
\vspace{-2.2cm}
\tikzset{every picture/.style={line width=0.75pt}} %set default line width to 0.75pt        

\hspace{0.4cm}\begin{tikzpicture}[x=0.75pt,y=0.75pt,yscale=-1,xscale=1]
%uncomment if require: \path (0,181); %set diagram left start at 0, and has height of 181

%Shape: Ellipse [id:dp06503626371936355] 
\draw  [fill={rgb, 255:red, 0; green, 0; blue, 0 }  ,fill opacity=1 ] (221.88,32.01) .. controls (221.05,30.7) and (221.56,28.88) .. (223.01,27.95) .. controls (224.45,27.01) and (226.3,27.32) .. (227.12,28.62) .. controls (227.95,29.93) and (227.45,31.75) .. (226,32.68) .. controls (224.55,33.62) and (222.71,33.32) .. (221.88,32.01) -- cycle ;
%Shape: Ellipse [id:dp19498868704253458] 
\draw  [fill={rgb, 255:red, 0; green, 0; blue, 0 }  ,fill opacity=1 ] (128.1,33.17) .. controls (127.28,31.86) and (127.78,30.04) .. (129.23,29.11) .. controls (130.68,28.18) and (132.52,28.48) .. (133.35,29.78) .. controls (134.17,31.09) and (133.67,32.91) .. (132.22,33.84) .. controls (130.77,34.78) and (128.93,34.48) .. (128.1,33.17) -- cycle ;
%Shape: Ellipse [id:dp7983568366590479] 
\draw  [fill={rgb, 255:red, 0; green, 0; blue, 0 }  ,fill opacity=1 ] (176.77,86.57) .. controls (175.95,85.26) and (176.45,83.44) .. (177.9,82.51) .. controls (179.35,81.57) and (181.19,81.87) .. (182.02,83.18) .. controls (182.84,84.49) and (182.34,86.31) .. (180.89,87.24) .. controls (179.44,88.17) and (177.6,87.87) .. (176.77,86.57) -- cycle ;
%Shape: Ellipse [id:dp6085332436840034] 
\draw  [fill={rgb, 255:red, 0; green, 0; blue, 0 }  ,fill opacity=1 ] (175.77,125.28) .. controls (174.95,123.98) and (175.45,122.16) .. (176.9,121.23) .. controls (178.35,120.29) and (180.19,120.59) .. (181.02,121.9) .. controls (181.84,123.21) and (181.34,125.02) .. (179.89,125.96) .. controls (178.44,126.89) and (176.6,126.59) .. (175.77,125.28) -- cycle ;
%Curve Lines [id:da3913898446584174] 
\draw    (224.5,30.32) .. controls (196.82,16.53) and (162.14,15.9) .. (130.73,31.48) ;
\draw [shift={(184.29,20)}, rotate = 180.94] [color={rgb, 255:red, 0; green, 0; blue, 0 }  ][line width=0.75]    (10.93,-3.29) .. controls (6.95,-1.4) and (3.31,-0.3) .. (0,0) .. controls (3.31,0.3) and (6.95,1.4) .. (10.93,3.29)   ;
%Curve Lines [id:da09895150625473059] 
\draw    (180.89,87.24) .. controls (158,66) and (149,53) .. (129.23,29.11) ;
\draw [shift={(158.51,64.48)}, rotate = 228.94] [color={rgb, 255:red, 0; green, 0; blue, 0 }  ][line width=0.75]    (10.93,-3.29) .. controls (6.95,-1.4) and (3.31,-0.3) .. (0,0) .. controls (3.31,0.3) and (6.95,1.4) .. (10.93,3.29)   ;
%Curve Lines [id:da5435804472466901] 
\draw    (130.73,31.48) .. controls (172,35) and (187,32) .. (224.5,30.32) ;
\draw [shift={(170.76,33.08)}, rotate = 358.59] [color={rgb, 255:red, 0; green, 0; blue, 0 }  ][line width=0.75]    (10.93,-3.29) .. controls (6.95,-1.4) and (3.31,-0.3) .. (0,0) .. controls (3.31,0.3) and (6.95,1.4) .. (10.93,3.29)   ;
%Curve Lines [id:da15079643210531612] 
\draw    (179.39,84.87) .. controls (199,65) and (206,54) .. (224.5,30.32) ;
\draw [shift={(198.61,63.66)}, rotate = 308.94] [color={rgb, 255:red, 0; green, 0; blue, 0 }  ][line width=0.75]    (10.93,-3.29) .. controls (6.95,-1.4) and (3.31,-0.3) .. (0,0) .. controls (3.31,0.3) and (6.95,1.4) .. (10.93,3.29)   ;
%Curve Lines [id:da06474865215360559] 
\draw    (179,124) .. controls (179.56,70.12) and (178,119) .. (179.39,84.87) ;
\draw [shift={(179.11,109.7)}, rotate = 270.11] [color={rgb, 255:red, 0; green, 0; blue, 0 }  ][line width=0.75]    (10.93,-3.29) .. controls (6.95,-1.4) and (3.31,-0.3) .. (0,0) .. controls (3.31,0.3) and (6.95,1.4) .. (10.93,3.29)   ;
%Curve Lines [id:da9580091557673736] 
\draw    (224.5,30.32) .. controls (222,64) and (211,78) .. (179.39,84.87) ;
\draw [shift={(216.98,61.49)}, rotate = 122.52] [color={rgb, 255:red, 0; green, 0; blue, 0 }  ][line width=0.75]    (10.93,-3.29) .. controls (6.95,-1.4) and (3.31,-0.3) .. (0,0) .. controls (3.31,0.3) and (6.95,1.4) .. (10.93,3.29)   ;
%Curve Lines [id:da15960670930894616] 
\draw    (130.73,31.48) .. controls (137,72) and (153,79) .. (179.39,84.87) ;
\draw [shift={(140.52,62.46)}, rotate = 54.88] [color={rgb, 255:red, 0; green, 0; blue, 0 }  ][line width=0.75]    (10.93,-3.29) .. controls (6.95,-1.4) and (3.31,-0.3) .. (0,0) .. controls (3.31,0.3) and (6.95,1.4) .. (10.93,3.29)   ;
%Curve Lines [id:da8805539015701058] 
\draw    (130.73,31.48) .. controls (124,96) and (141,106) .. (179,124) ;
\draw [shift={(132.41,82.38)}, rotate = 70.71] [color={rgb, 255:red, 0; green, 0; blue, 0 }  ][line width=0.75]    (10.93,-3.29) .. controls (6.95,-1.4) and (3.31,-0.3) .. (0,0) .. controls (3.31,0.3) and (6.95,1.4) .. (10.93,3.29)   ;
%Curve Lines [id:da7132297882781105] 
\draw    (226,32.68) .. controls (229,96) and (217,110) .. (179,124) ;
\draw [shift={(223.3,82.63)}, rotate = 106.41] [color={rgb, 255:red, 0; green, 0; blue, 0 }  ][line width=0.75]    (10.93,-3.29) .. controls (6.95,-1.4) and (3.31,-0.3) .. (0,0) .. controls (3.31,0.3) and (6.95,1.4) .. (10.93,3.29)   ;
%Curve Lines [id:da5963388630251218] 
\draw    (129.23,29.11) .. controls (176,-2) and (81,2) .. (129.23,29.11) ;
\draw [shift={(135.54,7.86)}, rotate = 186.02] [color={rgb, 255:red, 0; green, 0; blue, 0 }  ][line width=0.75]    (10.93,-3.29) .. controls (6.95,-1.4) and (3.31,-0.3) .. (0,0) .. controls (3.31,0.3) and (6.95,1.4) .. (10.93,3.29)   ;

% Text Node
\draw (116.11,25.31) node [anchor=north west][inner sep=0.75pt]  [font=\small]  {$a$};
% Text Node
\draw (230.69,21.91) node [anchor=north west][inner sep=0.75pt]  [font=\small]  {$b$};
% Text Node
\draw (176.09,63.7) node [anchor=north west][inner sep=0.75pt]  [font=\small]  {$c$};
% Text Node
\draw (173.77,127.68) node [anchor=north west][inner sep=0.75pt]  [font=\small]  {$d$};
% Text Node
\draw (174,4.4) node [anchor=north west][inner sep=0.75pt]  [font=\footnotesize]  {$4$};
% Text Node
\draw (175,37.4) node [anchor=north west][inner sep=0.75pt]  [font=\footnotesize]  {$1$};
% Text Node
\draw (157,47.4) node [anchor=north west][inner sep=0.75pt]  [font=\footnotesize]  {$2$};
% Text Node
\draw (193,46.4) node [anchor=north west][inner sep=0.75pt]  [font=\footnotesize]  {$1$};
% Text Node
\draw (206,73.4) node [anchor=north west][inner sep=0.75pt]  [font=\footnotesize]  {$1$};
% Text Node
\draw (144,73.4) node [anchor=north west][inner sep=0.75pt]  [font=\footnotesize]  {$1$};
% Text Node
\draw (183,97.4) node [anchor=north west][inner sep=0.75pt]  [font=\footnotesize]  {$1$};
% Text Node
\draw (218,96.4) node [anchor=north west][inner sep=0.75pt]  [font=\footnotesize]  {$4$};
% Text Node
\draw (128,97.4) node [anchor=north west][inner sep=0.75pt]  [font=\footnotesize]  {$3$};
% Text Node
\draw (19,32.4) node [anchor=north west][inner sep=0.75pt]    {$\omega ( a) =2$};
% Text Node
\draw (18,57.4) node [anchor=north west][inner sep=0.75pt]    {$\omega ( b) =4$};
% Text Node
\draw (19,82.4) node [anchor=north west][inner sep=0.75pt]    {$\omega ( c) =4$};
% Text Node
\draw (18,108.4) node [anchor=north west][inner sep=0.75pt]    {$\omega ( d) =1$};
% Text Node
\draw (250,67.4) node [anchor=north west][inner sep=0.75pt]  [font=\Large]  {$\mathcal{SG}(\mathfrak{f})$};
% Text Node
\draw (125,10.4) node [anchor=north west][inner sep=0.75pt]  [font=\footnotesize]  {$3$};

\end{tikzpicture}
\end{ex}
\end{multicols}
%\vspace{-0.3cm} \noindent %By analyzing the direction map of $f$, this is indeed a train track map. The transition matrix is easily seen to be irreducible. 
\vspace{0.2cm}

\begin{lem}\label{stack paths} If there is a directed path $P$ in $\mathcal{SG}(f)$ from $\mathcal{K}_i$ to $\mathcal{K}_j$ with $s(P)=d$, then $f^d(\alpha_i)$ traverses $\alpha_j$.  
\end{lem}

\begin{proof} Suppose a directed path $P$ with $s(P)=d$ has vertices $\mathcal{K}_{1}, \mathcal{K}_{2}, \dots, \mathcal{K}_{k}$ and let $s_i = s([\mathcal{K}_{i},\mathcal{K}_{i+1}])$. Hence $d = \sum_{i=1}^{k}s_{i}$. By definition of $s$, $f^{s_i}(\alpha_i)$ traverses $\alpha_{i+1}$. Hence $f^d(\alpha_{1})=f^{s_k}\circ \dots \circ f^{s_1}(\alpha_1)$ traverses $\alpha_k$. 
\end{proof}
 
 \begin{lem}\label{stack graph connected} $\mathcal{SG}(f)$ is strongly connected and for any $\mathcal{K}_i \in \mathcal{V}(\mathcal{SG}(f))$, we have $$\mathcal{V}(\mathcal{SG}(f)) \subseteq B_{n-n_i}(\mathcal{K}_i).$$
 \end{lem}
 
 \begin{proof}  Let $\mathcal{K}_i, \mathcal{K}_j \in \mathcal{V}(\mathcal{SG}(f))$. Since $f$ is irreducible, there is a power $s$ such that $f^s(\alpha_i)$ traverses $\alpha_j$. 
 \begin{itemize}
 \item Let $b_s$ be either $\alpha_j$ or $\overline{\alpha_j}$, whichever appears in $f^s(\alpha_i)$. 
 \item Let $b_{s-1}$ be a single edge in $f^{s-1}(\alpha_i)$ such that $b_s$ appears in $f(b_{s-1})$. 
 \item For $2 \leq t \leq s$, let $b_{s-t}$ be a single edge in $f^{s-t}(\alpha_i)$ such that $b_{s-t+1}$ appears in $f(b_{s-t})$.  
 \end{itemize}
 Hence $b_0 = \alpha_i$, and $f(b_t)$ contains $b_{t+1}$ for all $0 \leq t \leq s-1$. Whenever $b_t$ is a non-final edge, $f(b_t) = b_{t+1}$, so both are in the same stack. Whenever $b_t$ is a final edge, $f(b_{t})$ containing $b_{t+1}$ implies there is an edge in $\mathcal{SG}(f)$ from the the stack containing $b_t$ to the stack containing $b_{t+1}$. Following the sequence of stacks containing the edges $\{b_t\}_{t=0}^s$ gives a directed path in $\mathcal{SG}(f)$ from $\mathcal{K}_i$ to $\mathcal{K}_j$. Thus, $\mathcal{SG}(f)$ is strongly connected. 
 
 \vspace{0.1cm} 
 Let $\mathcal{K}_i, \mathcal{K}_j \in \mathcal{V}(\mathcal{SG}(f))$. If $\mathcal{K}_j =\mathcal{K}_i$, it is immediate that $\mathcal{K}_j \in B_{n-n_i}(\mathcal{K}_i)$. Suppose $\mathcal{K}_j \neq \mathcal{K}_i$. Since $\mathcal{SG}(f)$ is strongly connected, there is a path $P$ in $\mathcal{SG}(f)$ from $\mathcal{K}_i$ to $\mathcal{K}_j$. Choose $P$ so that every vertex in $P$ appears only once. Since each vertex in $P$ appears only once, we have at most one edge with terminal vertex $\mathcal{K}$ for each $\mathcal{K} \in \mathcal{V}(\mathcal{SG}(f))$. Moreover, since $P$ starts at $\mathcal{K}_i$ and ends at $\mathcal{K}_j \neq \mathcal{K}_i$, no edge in $P$ has terminal vertex $\mathcal{K}_i$. Observe that for any edge $E \in \mathcal{E}^+\mathcal{SG}(f)$, $s(E) \leq n_t$ where $\mathcal{K}_t$ is the terminal vertex of $E$. Thus,
  \begin{align*} s(P) = &\  \sum_{E \in P} s(E) \leq  \sum_{t \neq i} n_t   =  \ n-n_i. 
 \end{align*}
 
 \noindent Therefore $\mathcal{K}_j \in B_{n-n_i}(\mathcal{K}_i)$. Since $j$ is arbitrary, $\mathcal{V}(\mathcal{SG}(f) )\subseteq B_{n-n_i}(\mathcal{K}_i)$.   \end{proof}
 
 \begin{lem}\label{weight and word length}For any $d \in \mathbb{Z}_{\geq0}$, we have
 $$|f^{d+1}(\alpha_i)| \geq 1+ \sum_{\mathcal{K}_j \in B_d(\mathcal{K}_i)} \omega(\mathcal{K}_j).$$
 \end{lem}
 
 \begin{proof} 
 We prove this by induction on $d$. When $d=0$, $B_0(\mathcal{K}_i) = \{\mathcal{K}_i\}$, so 
 \begin{align*}|f(\alpha_i)| &= 1+|f(\alpha_i)|-1\\
 & =1+ \sum_{\mathcal{K}_j \in B_0(S)} \omega(\mathcal{K}_j).
 \end{align*}
 Now, let $d \geq 1$ and suppose the inequality holds for $d-1$. Let  $B_d(\mathcal{K}_i)-B_{d-1}(\mathcal{K}_i)=\{\mathcal{K}_{t_1},\dots,\mathcal{K}_{t_k}\}$. Then for each $t_q$, there is a directed path from $\mathcal{K}_i$ to $\mathcal{K}_{t_q}$ with length exactly $d$, so by Lemma \ref{stack paths}, $f^{d}(\alpha_i)$ traverses $\alpha_{t_q}$.
   
 Let $\delta = |f^d(\alpha_i)|$ and let $\alpha_{t_1},\dots,\alpha_{t_k},b_{k+1},\dots,b_{\delta}$ denote the edges appearing in $f^d(\alpha_i)$ (with multiplicity). Thus, by our induction hypothesis,
 \begin{align*}|f^{d+1}(\alpha_i)| & = |f(\alpha_{t_1})|+\dots+|f(\alpha_{t_k})|+ |f(b_{k+1})| + \dots + |f(b_\delta)| \\
 & \geq (\omega(\mathcal{K}_{t_1}) +1) + \dots + (\omega(\mathcal{K}_{t_k}) + 1) + (\delta-k) \\
 & = \delta + \sum_{q=1}^k \omega(\mathcal{K}_{t_q})\\
& = |f^d(\alpha_i)|  + \sum_{t=1}^k \omega(\mathcal{K}_{t_q}) \\
  & \geq 1+ \sum_{\mathcal{K}_t \in B_{d-1}(\mathcal{K}_i)} \omega(\mathcal{K}_t) + \sum_{q=1}^k \omega(\mathcal{K}_{t_q})  \\
 & =1+ \sum_{\mathcal{K}_t \in B_d(\mathcal{K}_i)} \omega(\mathcal{K}_t)
 \end{align*}
 This completes the proof of the lemma.  \end{proof}

\bigskip
\section{Lower Bound Proof}\label{lower bound proof}
 
\begin{mainthmA}\label{lower bound}
\textit{ Suppose $f:\Gamma \rightarrow \Gamma$ is an irreducible homotopy equivalence self graph map with fold decomposition consisting of $m$ total folds. Let $n = |\mathcal{E}\Gamma|$. Then
$$(m+1)^{\frac{1}{n}} \leq \lambda_f$$
where $\lambda_f$ is the largest eigenvalue of the transition matrix of $f$.}
\end{mainthmA}
 \smallskip 
 \begin{proof} If $f$ is not expanding, then by Lemma \ref{helpful facts} we have $m=0$ and $\lambda_f=1$, so the inequality holds. We now assume $f$ is expanding. \\

\vspace{-0.2cm}
Let $\lambda= \lambda_f$ and let $\ell$ be the metric on $\G$ from Lemma \ref{helpful facts}, %from Theorem \ref{bestvina}, 
so that $f$ is uniformly $\lambda-$expanding on $(\G,\ell)$. %\footnote{As we have restated it here, Theorem \ref{bestvina} requires in addition that $f$ is a train track map. However, by the Perron-Frobenius theorem, such a metric $\ell$ exists as long $f$ is an irreducible self graph map.}.
Let $e \in \mathcal{E}\G$ be an edge with the shortest length $\ell(e)$. Uniformly scale $\ell$ so that $\ell(e)=1$. 

\vspace{0.1cm}
We claim that $e$ must be the root edge in some stack of $f$. Otherwise, $e = f(a)$ for some edge $a$. Since $f$ is uniformly $\lambda-$expanding, $\ell(e) = \lambda \ell (a)$. Since $\lambda >1$, $\ell (e)> \ell (a)$, contradicting that $e$ is the shortest edge. 

\vspace{0.1cm}
Without loss of generality, suppose $e$ is the root edge in stack $\mathcal{K}_1$. Let $n_1$ be the number of edges in $\mathcal{K}_1$, so $f^{n_1-1}(e)$ is the final edge of  $\mathcal{K}_1$. 

\vspace{0.1cm}
\noindent By Lemma \ref{stack graph connected}, $\mathcal{V}(\mathcal{SG}(f)) \subseteq B_{n-n_1}(\mathcal{K}_1)$. Thus by Lemma \ref{weight and word length} with $d=n-n_1$, 
 $$|f^{n}(e)|=|f^{(n-n_1)+1}(f^{n_1-1}(e))| \geq 1+  \sum_{j=1}^p \omega(\mathcal{K}_j),$$
where $p$ is the number of stacks in $f$. By observation \ref{weight}, 
$$\sum_{j=1}^p \omega(\mathcal{K}_j)\geq m$$
Since every edge has length greater than or equal to $\ell (e) =1$, 
 \begin{align*} \lambda^n  = \ell (f^n(e))&  \geq |f^n(e)| \\
 &  \geq 1+ \sum_{j=1}^p \omega(\mathcal{K}_j)  \geq m+1
 \end{align*}
 Therefore, $(m+1)^{\frac{1}{n}} \leq \lambda_f.$
  \end{proof}
  
  Using the following lemma, (Lemma 3.1 in \cite{hamsong}), we provide an alternative proof of Theorem \hyperref[lower bound]{A} for irreducible homotopy equivalence self graph map with primitive transition matrices. In particular, if $f$ is an i.t.t. representative of a fully irreducible outer automorphism, then $T(f)$ is primitive (Lemma 2.4(2) in \cite{k14}).  
  
  \begin{lem}\label{hamsong} \cite{hamsong} Suppose $M$ is a non-negative integral primitive $n \times n$ matrix with $\lambda > 1$ its largest eigenvalue. Then
$$\lambda^n \geq |M|-n+1$$
where $|M|$ denotes the sum of the entries of M.
  \end{lem}
  
  \vspace{0.1cm}
  \noindent \textit{Alternative Proof of Theorem \hyperref[lower bound]{A} for i.h.e. maps with primitive transition matrix:} \\
  
  \vspace{-0.2cm}
\noindent Suppose $f$ is an irreducible homotopy equivalence self graph map with $T(f)$ primitive. Since  $|T(f)| = \sum_{e \in \mathcal{E}(\G)} |f(e)|$, and $T(f)$ is non-negative and integral, by Lemma \ref{hamsong}  and Lemma \ref{folds, edge images}, 
\begin{align*} \lambda^n  & \geq \Big{(}\sum_{e \in \mathcal{E}\G} (|f(e)|) \Big{)} -n +1 \\
& =  \Big{(}\sum_{e \in \mathcal{E}\G} (|f(e)| - 1) \Big{)} +1 \\
&  \geq m+1. \qedhere
\end{align*}
 Therefore, $(m+1)^{\frac{1}{n}} \leq \lambda_f.$ \hfill $\square$ \\

\bigskip
\vspace{0.5cm}
%\vspace{0.2cm}

%\newpage 
\section{Latent Symmetry}\label{latent symmetry}

In order for a graph to admit an i.h.e. map with very few folds in its fold decomposition, the graph isomorphism following the folds needs to sufficiently mix the edges. The stack score is designed to measure how much mixing the graph isomorphism can possibly admit, with a smaller stack score indicating more mixing is possible in the graph isomorphism. \\

\begin{df}\label{stack score}(Stack Score) A graph $G$ is a \textit{supergraph} of $\Gamma$ if $\Gamma$ is a subgraph of $G$. % and $\mathcal{V}G=\mathcal{V}\Gamma$. 
 Given a supergraph $G$ of $\G$ with $\mathcal{V}G=\mathcal{V}\G$, and $\psi \in $ Aut$(G)$, we define an equivalence relation $\sim_{\psi}$ on $\mathcal{E}\G$ generated by $a \sim_\psi \psi(a)$ whenever $\psi(a) \in \mathcal{E}\G$. %Equivalently, for $a,b \in \mathcal{E}\G$, 
%$a \sim_\psi b$
%if and only if there is a power $n \in \mathbb{Z}$ of $\psi$ such that $\psi^n(a)=b$ and for each power $k$ such that $|n-k|<n$, $\psi^k(a) \in \mathcal{E}\G$. 
The \textit{stack score} of a graph $\G$ is 

 \vspace{-0.1cm} 
$$\hspace{-0.2cm}\mathfrak{S}(\G):=\min\{ \text{number of }\sim_{\psi} \text{equivalence classes} \ | \ G \text{ is a supergraph of }\Gamma \text{ with }\mathcal{V}G = \mathcal{V}\G \text{ and }\psi \in \text{Aut}(G) \}$$

\end{df} 
\vspace{0.1cm}
Similarly, let $\mathfrak{O}(\G)$ be the minimum number of $\psi$ edge orbits over all pairs $(G,\psi)$, where $G$ is a supergraph of $\G$ with $\mathcal{V}G = \mathcal{V}\G$ and $\psi \in$ Aut$(G)$. Then $\mathfrak{O}(\G)$ is a similar graph invariant  to $\mathfrak{S}(\G)$. While $\mathfrak{O}(\G)$ is slightly easier to conceptualize and compute, we have $$\mathfrak{O}(\G) \leq \mathfrak{S}(\G)$$ and there are cases when the inequality is strict. Below, Example \ref{stack score ex} gives a graph $\G$ with $\mathfrak{O}(\G)=2$ and $\mathfrak{S}(\G)=3$. \\
%$$\min\{ \text{number of }\psi \text{ edge orbits in }G  \ | \ G \text{ is a supergraph of }\Gamma \text{ with }\mathcal{V}G = \mathcal{V}\G \text{ and }\psi \in \text{Aut}(G) \}\leq \mathfrak{S}(\G)$$
%\newpage 
%Observe that for any graph $\G$, we have $Sym(\G) \leq \mathfrak{S}(\G)$. \\
%\begin{multicols}{2}%\newpage 
\begin{ex}\label{stack score ex} 

\begin{multicols}{2}
Consider the graph $\G$ along with a supergraph $G$ as pictured to the right. Let $\psi_1 \in $ Aut$(G)$ rotate vertices in $G$ clockwise by one and send  
$$\vspace{-0.2cm} x_1 \mapsto x_2 \mapsto x_3 \mapsto x_4 \mapsto x_5 \mapsto x_1$$ 
and $$ c_i \mapsto d_i \mapsto e_i \mapsto a_{i} \mapsto b_i \mapsto c_{i+1}$$

\noindent for $1 \leq i \leq 3$, with the exception that $b_3~\mapsto~c_1$. Then $[c_1]_{\sim_{\psi_1}}=\{c_1, d_1, e_1, a_1, b_1, c_2, d_2\}$ and $a_2, a_3, x_2,x_4$ are each their own equivalence class.

 \columnbreak \textcolor{white}{\ } \vspace{-0.3cm}
\tikzset{every picture/.style={line width=0.75pt}} %set default line width to 0.75pt        

\begin{tikzpicture}[x=0.75pt,y=0.75pt,yscale=-1,xscale=1]

%Shape: Polygon [id:dp6076631107149031] 
\draw  [color={rgb, 255:red, 0; green, 0; blue, 0 }  ,draw opacity=1 ] (126.02,68.79) -- (109.13,120.5) -- (55.4,120.21) -- (39.07,68.33) -- (82.71,36.55) -- cycle ;
%Flowchart: Connector [id:dp19241575514622222] 
\draw  [color={rgb, 255:red, 0; green, 0; blue, 0 }  ,draw opacity=1 ] (74.19,23.84) .. controls (74.19,16.82) and (78,11.12) .. (82.71,11.12) .. controls (87.41,11.12) and (91.23,16.82) .. (91.23,23.84) .. controls (91.23,30.86) and (87.41,36.55) .. (82.71,36.55) .. controls (78,36.55) and (74.19,30.86) .. (74.19,23.84) -- cycle ;
%Curve Lines [id:da7044347667113351] 
\draw [color={rgb, 255:red, 0; green, 0; blue, 0 }  ,draw opacity=1 ]   (39.07,68.33) .. controls (22.69,85.94) and (23.72,118.76) .. (55.41,120.21) ;
%Curve Lines [id:da6405090274926868] 
\draw [color={rgb, 255:red, 0; green, 0; blue, 0 }  ,draw opacity=1 ]   (126.01,68.79) .. controls (132.66,90.1) and (129.31,109.1) .. (109.14,120.5) ;
%Curve Lines [id:da04901374649774737] 
\draw [color={rgb, 255:red, 0; green, 0; blue, 0 }  ,draw opacity=1 ]   (39.07,68.33) .. controls (6.1,84.67) and (20.16,130.15) .. (55.41,120.21) ;
%Curve Lines [id:da35779034575578916] 
\draw [color={rgb, 255:red, 0; green, 0; blue, 0 }  ,draw opacity=1 ]   (82.71,36.55) .. controls (107.45,32.17) and (121.95,45.3) .. (126.01,68.79) ;
%Flowchart: Connector [id:dp5473956890459415] 
\draw  [color={rgb, 255:red, 0; green, 0; blue, 0 }  ,draw opacity=1 ] (122.96,125.63) .. controls (126.36,131.15) and (125.68,137.79) .. (121.45,140.47) .. controls (117.22,143.15) and (111.04,140.85) .. (107.64,135.34) .. controls (104.24,129.82) and (104.91,123.18) .. (109.14,120.5) .. controls (113.37,117.82) and (119.56,120.12) .. (122.96,125.63) -- cycle ;
%Shape: Polygon [id:dp6215791260521595] 
\draw   (279.23,67.72) -- (263.13,119.43) -- (211.91,119.15) -- (196.35,67.26) -- (237.96,35.47) -- cycle ;
%Flowchart: Connector [id:dp0808899679078321] 
\draw   (230.46,22.77) .. controls (230.46,15.75) and (233.82,10.06) .. (237.95,10.06) .. controls (242.08,10.06) and (245.43,15.75) .. (245.43,22.77) .. controls (245.43,29.79) and (242.08,35.48) .. (237.95,35.48) .. controls (233.82,35.48) and (230.46,29.79) .. (230.46,22.77) -- cycle ;
%Curve Lines [id:da9920684650179228] 
\draw    (196.36,67.26) .. controls (189.82,83.6) and (193.01,106) .. (211.93,119.14) ;
%Curve Lines [id:da032637820156208974] 
\draw    (279.22,67.72) .. controls (285.56,89.03) and (282.37,108.03) .. (263.14,119.43) ;
%Curve Lines [id:da4204763909927802] 
\draw    (196.36,67.26) .. controls (163.91,85.94) and (173.21,120.95) .. (211.93,119.14) ;
%Curve Lines [id:da42884755796399787] 
\draw    (237.72,36.4) .. controls (258.76,34.06) and (274.71,44.91) .. (279.22,67.72) ;
%Flowchart: Connector [id:dp04100704755512696] 
\draw   (275.16,125.32) .. controls (278.41,130.84) and (278.28,137.14) .. (274.88,139.39) .. controls (271.48,141.65) and (266.1,139.01) .. (262.86,133.5) .. controls (259.62,127.99) and (259.74,121.69) .. (263.14,119.43) .. controls (266.54,117.17) and (271.92,119.81) .. (275.16,125.32) -- cycle ;
%Flowchart: Connector [id:dp6254405713032747] 
\draw   (289.82,59.51) .. controls (295.92,59.09) and (301.07,62.08) .. (301.32,66.19) .. controls (301.58,70.31) and (296.83,73.98) .. (290.73,74.41) .. controls (284.63,74.83) and (279.48,71.83) .. (279.22,67.72) .. controls (278.97,63.61) and (283.71,59.93) .. (289.82,59.51) -- cycle ;
%Flowchart: Connector [id:dp9030415179699962] 
\draw   (213.41,133.25) .. controls (210.58,139.01) and (205.33,142.04) .. (201.68,140.02) .. controls (198.03,137.99) and (197.37,131.67) .. (200.19,125.91) .. controls (203.02,120.14) and (208.28,117.11) .. (211.93,119.14) .. controls (215.57,121.17) and (216.24,127.49) .. (213.41,133.25) -- cycle ;
%Flowchart: Connector [id:dp014193805202823873] 
\draw   (189.55,55.13) .. controls (195.06,57.97) and (198.11,63.4) .. (196.36,67.26) .. controls (194.6,71.12) and (188.72,71.94) .. (183.2,69.1) .. controls (177.69,66.26) and (174.65,60.83) .. (176.4,56.97) .. controls (178.15,53.11) and (184.04,52.29) .. (189.55,55.13) -- cycle ;
%Curve Lines [id:da02774225342342307] 
\draw    (196.36,67.26) .. controls (204.52,46.46) and (214.67,38.24) .. (237.72,36.4) ;
%Curve Lines [id:da9917221038364579] 
\draw    (211.93,119.14) .. controls (233.86,129.75) and (242.16,131.11) .. (263.23,119.37) ;
%Curve Lines [id:da9170432166288784] 
\draw    (196.36,67.26) .. controls (187.68,29.04) and (198.01,15.91) .. (237.72,36.4) ;
%Curve Lines [id:da3680349928720752] 
\draw    (237.72,36.4) .. controls (269.3,19.19) and (289.97,22.48) .. (279.22,67.72) ;
%Curve Lines [id:da36911295845248326] 
\draw    (279.22,67.72) .. controls (302,94.5) and (305,110.5) .. (263.23,119.37) ;
%Curve Lines [id:da22515085056301842] 
\draw    (211.93,119.14) .. controls (229.01,153.78) and (246.57,154.87) .. (263.23,119.37) ;

% Text Node
\draw (72.5,163.95) node [anchor=north west][inner sep=0.75pt]  [font=\large]  {$\Gamma $};
% Text Node
\draw (229.66,165.12) node [anchor=north west][inner sep=0.75pt]  [font=\large]  {$G$};
% Text Node
\draw (49.74,84.77) node [anchor=north west][inner sep=0.75pt]  [font=\scriptsize]  {$a_{1}$};
% Text Node
\draw (31.14,88.05) node [anchor=north west][inner sep=0.75pt]  [font=\scriptsize]  {$a_{2}$};
% Text Node
\draw (4.27,93.53) node [anchor=north west][inner sep=0.75pt]  [font=\scriptsize]  {$a_{3}$};
% Text Node
\draw (59.04,53.04) node [anchor=north west][inner sep=0.75pt]  [font=\scriptsize]  {$b_{1}$};
% Text Node
\draw (94.15,51.95) node [anchor=north west][inner sep=0.75pt]  [font=\scriptsize]  {$c_{1}$};
% Text Node
\draw (109.65,26.78) node [anchor=north west][inner sep=0.75pt]  [font=\scriptsize]  {$c_{2}$};
% Text Node
\draw (100.36,81.49) node [anchor=north west][inner sep=0.75pt]  [font=\scriptsize]  {$d_{1}$};
% Text Node
\draw (130.33,91.34) node [anchor=north west][inner sep=0.75pt]  [font=\scriptsize]  {$d_{2}$};
% Text Node
\draw (77.62,104.47) node [anchor=north west][inner sep=0.75pt]  [font=\scriptsize]  {$e_{1}$};
% Text Node
\draw (58,10.37) node [anchor=north west][inner sep=0.75pt]  [font=\scriptsize]  {$x_{2}$};
% Text Node
\draw (126.24,124.47) node [anchor=north west][inner sep=0.75pt]  [font=\scriptsize]  {$x_{4}$};
% Text Node
\draw (204.39,83.68) node [anchor=north west][inner sep=0.75pt]  [font=\scriptsize]  {$a_{1}$};
% Text Node
\draw (179.59,90.24) node [anchor=north west][inner sep=0.75pt]  [font=\scriptsize]  {$a_{2}$};
% Text Node
\draw (159.96,92.43) node [anchor=north west][inner sep=0.75pt]  [font=\scriptsize]  {$a_{3}$};
% Text Node
\draw (213.69,51.95) node [anchor=north west][inner sep=0.75pt]  [font=\scriptsize]  {$b_{1}$};
% Text Node
\draw (248.8,50.85) node [anchor=north west][inner sep=0.75pt]  [font=\scriptsize]  {$c_{1}$};
% Text Node
\draw (263.27,28.97) node [anchor=north west][inner sep=0.75pt]  [font=\scriptsize]  {$c_{2}$};
% Text Node
\draw (255.02,80.4) node [anchor=north west][inner sep=0.75pt]  [font=\scriptsize]  {$d_{1}$};
% Text Node
\draw (280.85,89.15) node [anchor=north west][inner sep=0.75pt]  [font=\scriptsize]  {$d_{2}$};
% Text Node
\draw (232.27,103.37) node [anchor=north west][inner sep=0.75pt]  [font=\scriptsize]  {$e_{1}$};
% Text Node
\draw (212.66,9.27) node [anchor=north west][inner sep=0.75pt]  [font=\scriptsize]  {$x_{2}$};
% Text Node
\draw (280.89,123.37) node [anchor=north west][inner sep=0.75pt]  [font=\scriptsize]  {$x_{4}$};
% Text Node
\draw (159.96,46.48) node [anchor=north west][inner sep=0.75pt]  [font=\scriptsize]  {$x_{1}$};
% Text Node
\draw (303.58,57.42) node [anchor=north west][inner sep=0.75pt]  [font=\scriptsize]  {$x_{3}$};
% Text Node
\draw (181.66,130.73) node [anchor=north west][inner sep=0.75pt]  [font=\scriptsize]  {$x_{5}$};
% Text Node
\draw (231.24,127.44) node [anchor=north west][inner sep=0.75pt]  [font=\scriptsize]  {$e_{2}$};
% Text Node
\draw (232.27,146.05) node [anchor=north west][inner sep=0.75pt]  [font=\scriptsize]  {$e_{3}$};
% Text Node
\draw (298.58,94.62) node [anchor=north west][inner sep=0.75pt]  [font=\scriptsize]  {$d_{3}$};
% Text Node
\draw (275.67,14.74) node [anchor=north west][inner sep=0.75pt]  [font=\scriptsize]  {$c_{3}$};
% Text Node
\draw (197.16,30.06) node [anchor=north west][inner sep=0.75pt]  [font=\scriptsize]  {$b_{2}$};
% Text Node
\draw (181.66,12.56) node [anchor=north west][inner sep=0.75pt]  [font=\scriptsize]  {$b_{3}$};

\end{tikzpicture}

 \end{multicols}
 \newpage
 %\vspace{-0.3cm} 
\noindent \vspace{-0.25cm} \hspace{-0.25cm} Below, $(\G, \sim_{\psi_1})$ shows $\G$ with edges colored and dashed to distinguish the $\sim_{\psi_1}$ equivalence classes. 
 \begin{multicols}{2}
\textcolor{white}{ \ } \vspace{-0.1cm}
\hspace{-0.2cm}\input{figures/stack_score_example_1}
 \columnbreak
 
 \noindent Let $\psi_2 \in $ Aut$(G)$ rotate vertices in $G$ clockwise by two and send  
$$\vspace{-0.2cm} x_1 \mapsto x_3 \mapsto x_5 \mapsto x_2 \mapsto x_4 \mapsto x_1$$ 
and $$ d_i \mapsto a_i \mapsto c_i \mapsto e_{i} \mapsto b_i \mapsto d_{i+1}$$

\noindent for $1 \leq i \leq 3$, with the exception that $b_3~\mapsto~d_1$. Then  $[d_1]_{\sim_{\psi_2}}=\{d_1,a_1,c_1,e_1,b_1,d_2,a_2,c_2\}$, $[x_2 ]_{\sim_{\psi_2}}=\{x_2,x_4\}$, and $[a_3]_{\sim_{\psi_2}}=\{a_3\}$. % \columnbreak

\end{multicols}

%\vspace{0.1cm}
 \noindent Above, $(\G, \sim_{\psi_2})$ shows $\G$ with edges colored and dashed to distinguish the $\sim_{\psi_2}$ equivalence classes. For this graph $\G$, $\sim_{\psi_2}$ gives the minimal number of equivalence classes, so $\mathfrak{S}(\G)=3$.\\

\end{ex}

%Even though $\psi_2$ has two total edge orbits, there are edges not in $\G$ in between powers. \\

\vspace{-0.5cm}
\begin{mainthmB}\label{symmetry and folds}
\textit{Any irreducible expanding homotopy equivalence self graph map $f:\Gamma \rightarrow \Gamma$ which is periodic on the vertex set of $\G$ must have at least $\mathfrak{S}(\Gamma)$ folds. }
\end{mainthmB}

\begin{proof}

Suppose $f$ has $p$ stacks of sizes $n_1,n_2,\dots,n_p$ and root edges $e_1,\dots,e_p$. Then $f$ is given by:
$$f : \begin{cases} e_1 \mapsto f(e_1)\mapsto  \dots \mapsto f^{n_1-1}(e_1) \mapsto f^{n_1}(e_1) \\
 e_2 \mapsto f(e_2) \mapsto \dots \mapsto f^{n_2-1}(e_2) \mapsto f^{n_2}(e_2) \\
 \ \ \ \ \ \ \ \ \ \  \ \ \ \ \ \ \ \ \ \ \ \vdots \\
 e_p \mapsto f(e_p) \mapsto  \dots \mapsto f^{n_p-1}(e_p) \mapsto f^{n_p}(e_p) \\
\end{cases}$$
\vspace{0.1cm}

For each $i \in \{1 ,\dots, p\}$, let $v_i:=\iota(e_i)$ and $w_i:=\tau(e_i)$. Since $f$ is periodic on $\mathcal{V}\G$, there is some power $k_i$ of $f$ such that $f^{k_i}(v_i)=v_i$ and some power $t_i$ such that $f^{t_i}(w_i)=w_i$. Let $q_i$ be a multiple of $k_it_i$ such that $n_i \leq q_i$. Build a supergraph $G$ of $\G$ by adding edges $b_i^j$ for $n_i \leq j \leq q_i-1$  joining $f^{j}(v_i)$ to $f^j(w_i)$. Define $\psi$ on  the vertices of $G$ by $\psi_V=f_V$ and on the edges of $G$ by
%\begin{multicols}{2}
%\textcolor{white}{Theorem B Figure} \vspace{-1.5cm}
%$$\psi : \begin{cases} 
%e_1 \mapsto f(e_1)\mapsto  \dots \mapsto f^{n_1-1}(e_1) \mapsto b_1^{n_1} \mapsto \dots \mapsto b_1^{q_1-1} \mapsto e_1 \\
%e_2 \mapsto f(e_2)\mapsto  \dots \mapsto f^{n_2-1}(e_2) \mapsto b_2^{n_2} \mapsto \dots \mapsto b_2^{q_2-1} \mapsto e_2 \\
% \ \ \ \ \ \ \ \ \ \  \  \ \ \ \ \ \ \ \ \ \ \ \ \ \ \ \ \ \ \ \ \ \vdots \\
% e_p \mapsto f(e_p)\mapsto  \dots \mapsto f^{n_p-1}(e_p) \mapsto b_p^{n_p} \mapsto \dots \mapsto b_p^{q_p-1} \mapsto e_p \\
%\end{cases}$$
%\vspace{0.1cm}
%\columnbreak 
% \vspace{-4cm} \textcolor{white}{Theorem B Figure}

\tikzset{every picture/.style={line width=0.75pt}} %set default line width to 0.75pt        

\hspace{-0.7cm}\begin{tikzpicture}[x=0.75pt,y=0.75pt,yscale=-1,xscale=1]
%uncomment if require: \path (0,254); %set diagram left start at 0, and has height of 254

%Straight Lines [id:da1472094981565628] 
\draw [color={rgb, 255:red, 65; green, 117; blue, 5 }  ,draw opacity=1 ]   (567.14,28.96) -- (567.14,67.17) ;
%Straight Lines [id:da7398432864519398] 
\draw [color={rgb, 255:red, 65; green, 117; blue, 5 }  ,draw opacity=1 ]   (615.35,42.72) -- (584.49,71.76) ;
%Straight Lines [id:da8115594146126464] 
\draw [color={rgb, 255:red, 65; green, 117; blue, 5 }  ,draw opacity=1 ]   (643.31,77.87) -- (597.99,86.28) ;
%Straight Lines [id:da7891095790739282] 
\draw [color={rgb, 255:red, 65; green, 117; blue, 5 }  ,draw opacity=1 ]   (570.03,146.65) -- (570.03,109.2) ;
%Straight Lines [id:da295946974112395] 
\draw [color={rgb, 255:red, 139; green, 29; blue, 185 }  ,draw opacity=1 ]   (552.67,104.62) -- (526.64,136.72) ;
%Straight Lines [id:da5473633041806119] 
\draw [color={rgb, 255:red, 139; green, 29; blue, 185 }  ,draw opacity=1 ]   (540.14,90.86) -- (496.75,107.68) ;
%Straight Lines [id:da5214834823155656] 
\draw [color={rgb, 255:red, 139; green, 29; blue, 185 }  ,draw opacity=1 ]   (549.78,71.76) -- (521.82,40.42) ;
%Shape: Ellipse [id:dp0031949419168491033] 
\draw  [fill={rgb, 255:red, 0; green, 0; blue, 0 }  ,fill opacity=1 ] (565.65,30.87) .. controls (565.06,30.12) and (565.33,29.13) .. (566.25,28.65) .. controls (567.17,28.16) and (568.39,28.38) .. (568.97,29.13) .. controls (569.56,29.88) and (569.29,30.88) .. (568.37,31.36) .. controls (567.45,31.84) and (566.23,31.62) .. (565.65,30.87) -- cycle ;
%Shape: Ellipse [id:dp0279443290603969] 
\draw  [fill={rgb, 255:red, 0; green, 0; blue, 0 }  ,fill opacity=1 ] (565.65,67.55) .. controls (565.06,66.81) and (565.33,65.81) .. (566.25,65.33) .. controls (567.17,64.85) and (568.39,65.07) .. (568.97,65.81) .. controls (569.56,66.56) and (569.29,67.56) .. (568.37,68.04) .. controls (567.45,68.52) and (566.23,68.3) .. (565.65,67.55) -- cycle ;
%Shape: Ellipse [id:dp8355082812196635] 
\draw  [fill={rgb, 255:red, 0; green, 0; blue, 0 }  ,fill opacity=1 ] (583.97,72.14) .. controls (583.38,71.39) and (583.65,70.39) .. (584.57,69.91) .. controls (585.49,69.43) and (586.71,69.65) .. (587.29,70.4) .. controls (587.88,71.15) and (587.61,72.15) .. (586.69,72.63) .. controls (585.77,73.11) and (584.55,72.89) .. (583.97,72.14) -- cycle ;
%Shape: Ellipse [id:dp28722051646204116] 
\draw  [fill={rgb, 255:red, 0; green, 0; blue, 0 }  ,fill opacity=1 ] (613.86,43.86) .. controls (613.27,43.11) and (613.54,42.12) .. (614.46,41.64) .. controls (615.38,41.16) and (616.6,41.37) .. (617.18,42.12) .. controls (617.77,42.87) and (617.5,43.87) .. (616.58,44.35) .. controls (615.66,44.83) and (614.44,44.61) .. (613.86,43.86) -- cycle ;
%Shape: Ellipse [id:dp27948311043965157] 
\draw  [fill={rgb, 255:red, 0; green, 0; blue, 0 }  ,fill opacity=1 ] (596.5,87.42) .. controls (595.92,86.67) and (596.19,85.68) .. (597.11,85.2) .. controls (598.02,84.72) and (599.24,84.94) .. (599.83,85.68) .. controls (600.41,86.43) and (600.14,87.43) .. (599.22,87.91) .. controls (598.31,88.39) and (597.09,88.17) .. (596.5,87.42) -- cycle ;
%Shape: Ellipse [id:dp38359695653933845] 
\draw  [fill={rgb, 255:red, 0; green, 0; blue, 0 }  ,fill opacity=1 ] (641.82,78.25) .. controls (641.24,77.5) and (641.51,76.51) .. (642.42,76.03) .. controls (643.34,75.55) and (644.56,75.76) .. (645.14,76.51) .. controls (645.73,77.26) and (645.46,78.26) .. (644.54,78.74) .. controls (643.62,79.22) and (642.41,79) .. (641.82,78.25) -- cycle ;
%Shape: Ellipse [id:dp08561284397451274] 
\draw  [fill={rgb, 255:red, 0; green, 0; blue, 0 }  ,fill opacity=1 ] (568.54,111.11) .. controls (567.96,110.37) and (568.23,109.37) .. (569.14,108.89) .. controls (570.06,108.41) and (571.28,108.63) .. (571.86,109.38) .. controls (572.45,110.12) and (572.18,111.12) .. (571.26,111.6) .. controls (570.34,112.08) and (569.13,111.86) .. (568.54,111.11) -- cycle ;
%Shape: Ellipse [id:dp06623359991340072] 
\draw  [fill={rgb, 255:red, 0; green, 0; blue, 0 }  ,fill opacity=1 ] (568.54,147.03) .. controls (567.96,146.28) and (568.23,145.29) .. (569.14,144.81) .. controls (570.06,144.33) and (571.28,144.54) .. (571.86,145.29) .. controls (572.45,146.04) and (572.18,147.04) .. (571.26,147.52) .. controls (570.34,148) and (569.13,147.78) .. (568.54,147.03) -- cycle ;
%Shape: Ellipse [id:dp5067151029722932] 
\draw  [fill={rgb, 255:red, 0; green, 0; blue, 0 }  ,fill opacity=1 ] (550.22,105.77) .. controls (549.64,105.02) and (549.91,104.02) .. (550.82,103.54) .. controls (551.74,103.06) and (552.96,103.28) .. (553.54,104.03) .. controls (554.13,104.78) and (553.86,105.77) .. (552.94,106.25) .. controls (552.02,106.73) and (550.81,106.51) .. (550.22,105.77) -- cycle ;
%Shape: Ellipse [id:dp19837753251436196] 
\draw  [fill={rgb, 255:red, 0; green, 0; blue, 0 }  ,fill opacity=1 ] (525.15,137.1) .. controls (524.57,136.35) and (524.84,135.35) .. (525.75,134.87) .. controls (526.67,134.39) and (527.89,134.61) .. (528.47,135.36) .. controls (529.06,136.11) and (528.79,137.1) .. (527.87,137.59) .. controls (526.95,138.07) and (525.74,137.85) .. (525.15,137.1) -- cycle ;
%Shape: Ellipse [id:dp9497173823223393] 
\draw  [fill={rgb, 255:red, 0; green, 0; blue, 0 }  ,fill opacity=1 ] (537.69,92.01) .. controls (537.1,91.26) and (537.37,90.26) .. (538.29,89.78) .. controls (539.21,89.3) and (540.42,89.52) .. (541.01,90.27) .. controls (541.59,91.02) and (541.32,92.02) .. (540.41,92.5) .. controls (539.49,92.98) and (538.27,92.76) .. (537.69,92.01) -- cycle ;
%Shape: Ellipse [id:dp958302943736971] 
\draw  [fill={rgb, 255:red, 0; green, 0; blue, 0 }  ,fill opacity=1 ] (496.22,108.06) .. controls (495.64,107.31) and (495.91,106.31) .. (496.83,105.83) .. controls (497.75,105.35) and (498.96,105.57) .. (499.55,106.32) .. controls (500.13,107.07) and (499.86,108.06) .. (498.94,108.54) .. controls (498.03,109.03) and (496.81,108.81) .. (496.22,108.06) -- cycle ;
%Shape: Ellipse [id:dp5314487311240246] 
\draw  [fill={rgb, 255:red, 0; green, 0; blue, 0 }  ,fill opacity=1 ] (548.29,72.14) .. controls (547.71,71.39) and (547.98,70.39) .. (548.9,69.91) .. controls (549.81,69.43) and (551.03,69.65) .. (551.62,70.4) .. controls (552.2,71.15) and (551.93,72.15) .. (551.01,72.63) .. controls (550.09,73.11) and (548.88,72.89) .. (548.29,72.14) -- cycle ;
%Shape: Ellipse [id:dp08827951084030516] 
\draw  [fill={rgb, 255:red, 0; green, 0; blue, 0 }  ,fill opacity=1 ] (520.33,41.57) .. controls (519.75,40.82) and (520.02,39.82) .. (520.93,39.34) .. controls (521.85,38.86) and (523.07,39.08) .. (523.65,39.83) .. controls (524.24,40.58) and (523.97,41.58) .. (523.05,42.06) .. controls (522.13,42.54) and (520.91,42.32) .. (520.33,41.57) -- cycle ;
%Curve Lines [id:da5109791384816773] 
\draw    (574.85,21.32) .. controls (596.72,19.87) and (601.37,25.25) .. (609.95,31.65) ;
\draw [shift={(611.49,32.78)}, rotate = 215.5] [color={rgb, 255:red, 0; green, 0; blue, 0 }  ][line width=0.75]    (10.93,-3.29) .. controls (6.95,-1.4) and (3.31,-0.3) .. (0,0) .. controls (3.31,0.3) and (6.95,1.4) .. (10.93,3.29)   ;

% Text Node
\draw (550.51,40.7) node [anchor=north west][inner sep=0.75pt]  [font=\footnotesize,color={rgb, 255:red, 65; green, 117; blue, 5 }  ,opacity=1 ]  {$e_{1}$};
% Text Node
\draw (575.28,42.99) node [anchor=north west][inner sep=0.75pt]  [font=\footnotesize,color={rgb, 255:red, 65; green, 117; blue, 5 }  ,opacity=1 ]  {$f( e_{1})$};
% Text Node
\draw (595.42,65.04) node [anchor=north west][inner sep=0.75pt]  [font=\footnotesize,color={rgb, 255:red, 65; green, 117; blue, 5 }  ,opacity=1 ]  {$f^{2}( e_{1})$};
% Text Node
\draw (623.85,92.23) node [anchor=north west][inner sep=0.75pt]  [rotate=-103.52]  {$\dotsc $};
% Text Node
\draw (574.9,113.83) node [anchor=north west][inner sep=0.75pt]  [font=\footnotesize,color={rgb, 255:red, 65; green, 117; blue, 5 }  ,opacity=1 ]  {$f^{n_{1} -1}( e_{1})$};
% Text Node
\draw (540.76,117.01) node [anchor=north west][inner sep=0.75pt]  [font=\footnotesize,color={rgb, 255:red, 139; green, 29; blue, 185 }  ,opacity=1 ]  {$b_{1}^{n_{1}}$};
% Text Node
\draw (514.53,97.43) node [anchor=north west][inner sep=0.75pt]  [font=\footnotesize,color={rgb, 255:red, 139; green, 29; blue, 185 }  ,opacity=1 ]  {$b_{1}^{n_{1} +1}$};
% Text Node
\draw (526.14,72.1) node [anchor=north west][inner sep=0.75pt]  [color={rgb, 255:red, 139; green, 29; blue, 185 }  ,opacity=1 ,rotate=-92.58]  {$\dotsc $};
% Text Node
\draw (509.73,56.93) node [anchor=north west][inner sep=0.75pt]  [font=\footnotesize,color={rgb, 255:red, 139; green, 29; blue, 185 }  ,opacity=1 ]  {$b_{1}^{q_{1} -1}$};
% Text Node
\draw (594.11,2.67) node [anchor=north west][inner sep=0.75pt]    {$\psi $};
% Text Node
\draw (12,26.4) node [anchor=north west][inner sep=0.75pt]  [font=\normalsize]  {$ \begin{array}{l}
\psi \ :\ \begin{cases}
e_{1} \ \mapsto f( e_{1}) \mapsto \ \dotsc \ \mapsto \ f^{n_{1} -1}( e_{1}) \mapsto b_{1}^{n_{1}} \ \mapsto \dotsc \ \mapsto \ b_{1}^{q_{1} -1} \mapsto e_{1}\\
e_{2} \ \mapsto f( e_{2}) \mapsto \ \dotsc \ \mapsto \ f^{n_{2} -1}( e_{2}) \mapsto b_{2}^{n_{2}} \ \mapsto \dotsc \ \mapsto \ b_{2}^{q_{2} -1} \mapsto e_{2}\\
\ \ \ \ \ \ \ \ \ \ \ \ \ \ \ \ \ \ \ \ \ \ \ \ \ \ \ \ \ \ \ \ \ \ \ \ \ \ \ \ \ \ \ \ \ \ \ \ \ \ \ \ \ \ \ \ \ \vdots \\
e_{p} \ \mapsto f( e_{p}) \mapsto \ \dotsc \ \mapsto \ f^{n_{p} -1}( e_{p}) \mapsto b_{p}^{n_{p}} \ \mapsto \dotsc \ \mapsto \ b_{p}^{q_{p} -1} \mapsto e_{p}
\end{cases}\\
\ \ 
\end{array}$};
% Text Node
\draw (561.12,16.48) node [anchor=north west][inner sep=0.75pt]  [font=\scriptsize]  {$v_{1}$};
% Text Node
\draw (561.07,69.69) node [anchor=north west][inner sep=0.75pt]  [font=\scriptsize]  {$w_{1}$};

\end{tikzpicture}
\vspace{0.2cm}
%\end{multicols}

\noindent We claim $\psi$ an automorphism of $G$. By definition, $\psi_E$ is bijection from $\mathcal{E}G$ to itself. We also have $\psi_V = f_V$ is a bijection by our hypothesis on $f$. It remains to show that $\psi$ is a graph map. For any edge $b_i^j$, with $n_i \leq j \leq q_i-2$ and $ 1 \leq i \leq p$, we have
 \begin{align*} \psi(\iota(b_i^j)) & = \psi(f^j(v_i)) \\
 & = f^{j+1}(v_i) \\
& = \iota(b_i^{j+1}) \\
& = \iota(\psi(b_i^j)). \end{align*} 
For any edge $b_i^{q_i-1}$ with $1 \leq i \leq p$, we have
\begin{align*}
\psi(\iota(b_i^{q_i-1})) & = \psi(f^{q_i-1}(v_i)) \\
 & = f^{q_i}(v_i) \\
 & = v_i \\
& = \iota(e_i) \\
& = \iota(\psi(b_i^{q_i-1})). \end{align*} 
For the edges $f^{n_i-1}(e_i)$ with $1 \leq i \leq p$, we have
\begin{align*}
\psi(\iota(f^{n_i-1}(e_i))) & = \psi(f^{n_i-1}(v_i)) \\
 & = f^{n_i}(v_i) \\
& = \iota(b_i^1) \\
& = \iota(\psi(f^{n_i-1}(e_i))). \end{align*} 

Thus $\psi(\iota(e))=\iota(\psi(e))$ for every edge $e \in \mathcal{E}G$. Similarly, $\psi(\tau(e))=\tau(\psi(e))$, so indeed $\psi$ is a graph map. 

Observe that $\sim_{\psi}$ partitions $\mathcal{E}\G$ into exactly $p$ equivalence classes. 
Hence $\mathfrak{S}(\G) \leq p$. Since $f$ is periodic on the vertices of $\G$, by Lemma \ref{folds, edge images}, $p$ is less than or equal to the number of folds in the fold decomposition of $f$. Hence $f$ has at least $\mathfrak{S}(\G)$ many folds. 
\end{proof}
 
\noindent Theorems \hyperref[lower bound]{A} and \hyperref[symmetry and folds]{B} immediately give the following corollary. 

%\vspace{0.1cm}
\begin{cor}\label{main cor}\textit{Let $f :\G \rightarrow \G$ be an  irreducible expanding homotopy equivalence self graph map which is periodic on the vertex set of $\G$. Let $n=|\mathcal{E}\G|$. Then }
$$(\mathfrak{S}(\G)+1)^{\frac{1}{n}} \leq \lambda_f,$$
where $\mathfrak{S}(\G)$ is the stack score of $\G$ and $\lambda_f$ is the leading eigenvalue of the transition matrix of $f$. 
%where $\mathfrak{S}(\G)$ is the stack score of $\G$. }
\end{cor}
 
 \smallskip
%\vspace{0.5cm}
% \newpage
 \section{Single Fold Maps}
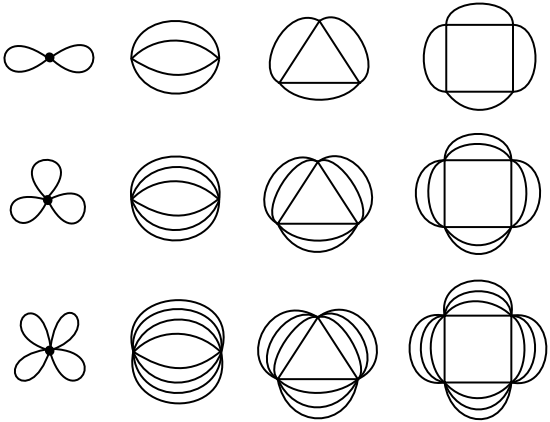
\label{single fold maps}
 
% \subsection{Polygonal Graphs} \textcolor{white}{hazel is a very good dog}

 \begin{multicols}{2}
 \input{figures/polygonals}
\columnbreak

\begin{df}(Polygonal Graph) Let $P_{s,k}$ be a graph with vertex set $\mathcal{V}P_{s,k} = \{v_0, \dots, v_{s-1}\}$ and edges
  $$\mathcal{E}P_{s,k} = \{e_{i}^{j} : \  1 \leq j \leq k, \ 0 \leq i \leq s-1, \},$$
where an edge $e_i^{j}$ joins $v_i$ to $v_{i+1}$, with vertex subscripts taken modulo $s$. We call $P_{s,k}$ the  \textit{$s$-gonal graph of depth $k$}. A  \textit{side} of $P_{s,k}$ is $$\mathfrak{s}_i := \{e_i^{j} \ | \ 1 \leq j \leq k\} \subseteq \mathcal{E}P_{s,k} $$ 
The sides of $P_{s,k}$ partition $\mathcal{E}P_{s,k}$. 
\end{df}
 \end{multicols}
 
% \newpage

Observe that each polygonal graph has an edge transitive automorphism. Hence $\mathfrak{S}(P_{s,k})=1$ for any $s,k \in \mathbb{N}$. The following lemma provides a converse to this statement in the special case that a graph $G$ has an automorphism which is both edge and vertex transitive.

 \begin{lem}\label{transitive} If $G$ is a connected graph and there exists a $\psi \in Aut(G)$ such that the cyclic subgroup of Aut$(G)$ generated by $\psi$, denoted $\langle \psi \rangle$, acts transitively on both $\mathcal{V}G$ and $\mathcal{E}G$, then $G$ is isomorphic to some polygonal graph $P_{s,k}$. 
 \end{lem}
 
 \begin{proof} Let $\mathcal{V}G = \{v_0, \dots, v_{s-1}\}$. Since $\langle \psi \rangle $ is transitive on $\mathcal{V}G$, we can assume the vertices are labeled so that $\psi(v_i)= v_{i+1}$, with subscripts taken modulo $s$. Suppose $e$ is an edge joining $v_0$ to $v_j$. Thus for any power $m$, $\psi^m(e)$ is an edge joining $v_{m}$ to $v_{j+m}$. Since $\langle \psi \rangle$ is transitive on $\mathcal{E}G$, $$\{\psi^m(e) | m \in \mathbb{Z}\} = \mathcal{E}G.$$ Hence each $a \in \mathcal{E}G$ joins $v_i$ to $v_{j+i}$ for some $i$. In other words, there is an edge between $v_{i_1}$ and $v_{i_2}$ if and only if $|i_1-i_2| = j$.
 
 Suppose there are precisely $k$ distinct edges in $G$ joining $v_0$ to $v_j$. Since $\psi$ is an automorphism, there must be exactly $k$ edges joining $\psi^m(v_0)=v_m$ to $\psi^m(v_j)=v_{m+j}$ for each power $m$.  To summarize, given any two vertices $v_{i_1}$ and $v_{i_2}$, there are exactly $k$ edges joining $v_{i_1}$ to $v_{i_2}$ if $|i_1-i_2|=j$, and zero edges joining $v_{i_1}$ to $v_{i_2}$ otherwise. Since $G$ is connected, $G$ is isomorphic to $P_{s,k}$.
 \end{proof}
 
 \vspace{0.1cm}
 The following lemma classifies the structure of connected subgraphs of polygonal graphs with stack score equal to 1. In particular, the number of edges in each side of the polygonal graph which are also in the subgraph can vary by at most 1. \\
 
 \begin{lem}\label{sides} Suppose $\G$ is a connected subgraph of $P_{s,k}$ for $s \geq 3$ and there exists an edge transitive automorphism $\psi~\in~\text{Aut}(P_{s,k})$ and an edge $e \in \mathcal{E}\G$ 
 %$$\psi: e \mapsto \psi(e) \mapsto \psi^2(e) \mapsto \dots \mapsto ^{n-1}(e) \mapsto b_1 \mapsto \dots \mapsto b_j  \mapsto e$$
such that \begin{align} \{e,\psi(e),\dots, & \psi^{n-1}(e)\} = \mathcal{E}\G.  \label{eq:2} \end{align} %and $$\{\psi^n(e), \psi^{n+1}(e), \dots, \psi^{sk-1}(e)\} = \mathcal{E}P_{s,k} - \mathcal{E} \G.$$ 
Let $\mathfrak{s}_0, \dots, \mathfrak{s}_{s-1}$ denote the sides of $P_{s,k}$ and write $n=sm+t$ for $m \in \{1, \dots, k\}$ and $t \in \{0, \dots, m-1\}$. Then
 \begin{itemize}
 \item[(i)] there are precisely $t$ sides such that $|\mathfrak{s}_i \cap \mathcal{E}\G| = m+1$, and
 \item[(ii)] the remaining $s-t$ sides have $|\mathfrak{s}_i \cap \mathcal{E}\G| = m$. 
 \end{itemize}
  \end{lem}
 \begin{proof}

By the definition of a graph automorphism, $\psi(\iota(a)) = \iota(\psi(a))$ for every $a \in \mathcal{E}^{\pm}P_{s,k}$. Thus $\psi$ descends to a bijection on the sides of $P_{s,k}$. Relabel the sides of $P_{s,k}$ so that $e \in \mathfrak{s}_0$ and $\psi$ on the sides is given by%$\mathfrak{s}_0, \dots , \mathfrak{s}_{s-1}$, so that $\psi$ maps edges in side $\mathfrak{s}_i$ to edges in side $\mathfrak{s}_{i+1}$,
$$\psi: \mathfrak{s}_0 \mapsto \mathfrak{s}_1 \mapsto \dots \mapsto \mathfrak{s}_{s-1} \mapsto \mathfrak{s}_0.$$
% Let $e \in \mathcal{E}\G$ be the root edge of the single stack in $\G$. Without loss of generality, assume $e \in \mathfrak{s_0}$. 
% $$\psi: \mathfrak{s}_0 \mapsto \mathfrak{s}_1 \mapsto \mathfrak{s}_2 \mapsto \mathfrak{s}_0.$$
%As mentioned earlier, by the proof of  \hyperref[symmetry and folds]{B}, we can write $\psi$ on the edges of $G$ as%Assume $ \psi$ sends $v_0 \mapsto v_1 \mapsto v_2 \mapsto v_0$. 
%$$\psi: e \mapsto f(e) \mapsto f^2(e) \mapsto \dots \mapsto f^{n-1}(e) \mapsto b_1 \mapsto \dots \mapsto b_j  \mapsto e$$
%where $\{e,f(e),\dots, f^{n-1}(e)\} = \mathcal{E}\G$ and $\{b_1, \dots, b_j \} = \mathcal{E}G - \mathcal{E} \G$. Thus,
Hence by  \eqref{eq:2}, $$\mathfrak{s}_i \cap \mathcal{E}\G  = \{\psi^{j}(e) \ |  \ j \in \{0,1,\dots,n-1\} \ \text{and} \ j \equiv i \text{ mod }s \}.$$
Therefore, 
$$|\mathfrak{s}_i \cap \mathcal{E}\G|=\begin{cases} m+1 & \text{ if } 1 \leq i \leq t-1 \\
m &\text{ it } t \leq i \leq s-1.
\end{cases}$$
This completes the proof of the lemma.
\end{proof}

\begin{df}\label{almost 3}(Almost 3-gonal graphs) For any $k \in \mathbb{N}$, we define two graphs called the \textit{almost 3-gonal graphs of depth $k$.} \begin{itemize} 
\item[(i)] Let $\Delta_k^-$ be $P_{3,k}$ with edge $e_0^k$ removed. Note that the choice of removed edge does not change the isomorphism class of $\Delta_k^-$. We have $$\text{Rank}(\Delta_k^-)=3k-3.$$
\item[(ii)] Let  $\Delta_k^{+}$ be $P_{3,k+1}$ with edges $e_0^{k+1}$ and $e_1^{k+1}$ removed. The choice of removed edges from two distinct sides of $P_{3,k+1}$ does not change the isomorphism class of $\Delta_k^+$. We have $$\text{Rank}(\Delta_k^+)=3k-1.$$
\end{itemize}
\end{df}

\begin{df} (Rose) For any $r \in \mathbb{N}$, the \textit{rose with $r$ petals} is $R_r=P_{1,r}$. We have Rank$(R_r)=r$. 

\end{df}

\newpage
% \subsection{Classifying Single Fold Train Track Maps}
 
 \begin{mainthmC}\label{single folds}\vspace{-0.2cm}
\begin{multicols}{2}
 \textit{Suppose $\G$ is a connected rank $r$ graph and $f: \G \rightarrow \G$ is a single fold irreducible homotopy equivalence self graph map. %Suppose $f: \G \rightarrow \G$ is a single fold irreducible train track map representing an irreducible $\varphi \in$ Out$(F_r)$. 
 Then $\G$ is isomorphic to one of the graphs to the right for some $k \geq 2$. }
 
 \vspace{0.2cm}
 
 \noindent \textit{In particular:
\begin{enumerate}
\item[(i)] if  $r \equiv 0$ mod $3$, then $\G \cong G \in \{R_r, \Delta_k^{-}\}$,
%\vspace{0.1cm}
\item[(ii)] if $r \equiv 1$ mod $3$, then $\G \cong R_r$, and
%\vspace{0.1cm}
\item[(iii)] if $r \equiv 2$ mod $3$, then  $\G \cong G \in \{R_r, \Delta_k^+\}$,\end{enumerate}
\vspace{0.1cm}
for appropriate values of $k$. }
\columnbreak
\textcolor{white}{ \ }\\
\tikzset{every picture/.style={line width=0.75pt}} %set default line width to 0.75pt        

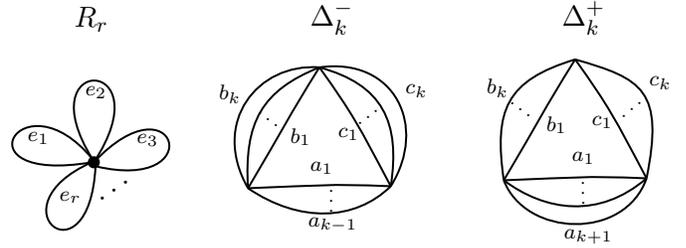
\hspace{-1.3cm} \vspace{0.3cm} \begin{tikzpicture}[x=0.75pt,y=0.75pt,yscale=-1,xscale=1]

%Curve Lines [id:da6924922789062755] 
\draw    (56.64,90.7) .. controls (23,37) and (93,38) .. (57.28,92.25) ;
%Curve Lines [id:da8706651017574671] 
\draw    (57.28,92.25) .. controls (99,48) and (116,108) .. (58.14,94.26) ;
%Shape: Ellipse [id:dp9558946616333523] 
\draw  [fill={rgb, 255:red, 0; green, 0; blue, 0 }  ,fill opacity=1 ][line width=0.75]  (58.98,90.32) .. controls (60.03,91.24) and (60.11,92.85) .. (59.17,93.92) .. controls (58.23,94.98) and (56.62,95.1) .. (55.58,94.17) .. controls (54.53,93.25) and (54.45,91.64) .. (55.39,90.58) .. controls (56.33,89.51) and (57.94,89.4) .. (58.98,90.32) -- cycle ;
%Curve Lines [id:da10834350526073577] 
\draw [line width=0.75]    (56.28,92.25) .. controls (-13,104) and (21,47) .. (56.64,90.7) ;
%Curve Lines [id:da084381339408659] 
\draw    (58.14,94.26) .. controls (55,155) and (5,115) .. (56.28,92.25) ;
%Curve Lines [id:da913529146454332] 
\draw    (135.04,105.36) .. controls (153,75) and (155,70) .. (171.04,44.36) ;
%Curve Lines [id:da016851898324729664] 
\draw    (135.04,105.36) .. controls (172,104) and (174,103) .. (207.04,104.36) ;
%Curve Lines [id:da6588945710687599] 
\draw    (171.04,44.36) .. controls (188,68) and (191,74) .. (207.04,104.36) ;
%Curve Lines [id:da23858804344759243] 
\draw    (135.04,105.36) .. controls (133,72) and (142,56) .. (171.04,44.36) ;
%Curve Lines [id:da936644372761162] 
\draw    (171.04,44.36) .. controls (203,53) and (212,72) .. (207.04,104.36) ;
%Curve Lines [id:da2378019011576622] 
\draw    (135.04,105.36) .. controls (161,123) and (189,122) .. (207.04,104.36) ;
%Curve Lines [id:da22821103469159265] 
\draw    (135.04,105.36) .. controls (116,72) and (140,42) .. (171.04,44.36) ;
%Curve Lines [id:da8896807307874222] 
\draw    (171.04,44.36) .. controls (209,39) and (226,74) .. (207.04,104.36) ;
%Curve Lines [id:da44969352087609327] 
\draw    (264.04,101.36) .. controls (282,71) and (284,66) .. (300.04,40.36) ;
%Curve Lines [id:da09284919830997662] 
\draw    (264.04,101.36) .. controls (301,100) and (303,99) .. (336.04,100.36) ;
%Curve Lines [id:da2677240522212432] 
\draw    (300.04,40.36) .. controls (317,64) and (320,70) .. (336.04,100.36) ;
%Curve Lines [id:da756640741133167] 
\draw    (264.04,101.36) .. controls (259,66) and (265,51) .. (300.04,40.36) ;
%Curve Lines [id:da9225502503516558] 
\draw    (300.04,40.36) .. controls (334,50) and (345,56) .. (336.04,100.36) ;
%Curve Lines [id:da6843141925934617] 
\draw    (264.04,101.36) .. controls (285,117) and (312,121) .. (336.04,100.36) ;
%Curve Lines [id:da1496828074210348] 
\draw    (264.04,101.36) .. controls (275,130) and (325,132) .. (336.04,100.36) ;

% Text Node
\draw (22.5,75.4) node [anchor=north west][inner sep=0.75pt]  [font=\scriptsize]  {$e_{1}$};
% Text Node
\draw (51.5,52.4) node [anchor=north west][inner sep=0.75pt]  [font=\scriptsize]  {$e_{2}$};
% Text Node
\draw (77.5,76.4) node [anchor=north west][inner sep=0.75pt]  [font=\scriptsize]  {$e_{3}$};
% Text Node
\draw (38.5,105.4) node [anchor=north west][inner sep=0.75pt]  [font=\scriptsize]  {$e_{r}$};
% Text Node
\draw (57.26,110.42) node [anchor=north west][inner sep=0.75pt]  [font=\large,rotate=-326.41]  {$\dotsc $};
% Text Node
\draw (46,12.4) node [anchor=north west][inner sep=0.75pt]  [font=\normalsize]  {$R_{r}$};
% Text Node
\draw (165,9.4) node [anchor=north west][inner sep=0.75pt]  [font=\normalsize]  {$\Delta _{k}^{-}$};
% Text Node
\draw (165,90.4) node [anchor=north west][inner sep=0.75pt]  [font=\scriptsize]  {$a_{1}$};
% Text Node
\draw (164,118.4) node [anchor=north west][inner sep=0.75pt]  [font=\scriptsize]  {$a_{k-1}$};
% Text Node
\draw (179.03,103.73) node [anchor=north west][inner sep=0.75pt]  [font=\tiny,rotate=-89.49]  {$\dotsc $};
% Text Node
\draw (179,73.4) node [anchor=north west][inner sep=0.75pt]  [font=\scriptsize]  {$c_{1}$};
% Text Node
\draw (203.24,66.23) node [anchor=north west][inner sep=0.75pt]  [font=\tiny,rotate=-142.02]  {$\dotsc $};
% Text Node
\draw (213,51.4) node [anchor=north west][inner sep=0.75pt]  [font=\scriptsize]  {$c_{k}$};
% Text Node
\draw (155,73.4) node [anchor=north west][inner sep=0.75pt]  [font=\scriptsize]  {$b_{1}$};
% Text Node
\draw (150.39,77.07) node [anchor=north west][inner sep=0.75pt]  [font=\tiny,rotate=-216.17]  {$\dotsc $};
% Text Node
\draw (119,51.4) node [anchor=north west][inner sep=0.75pt]  [font=\scriptsize]  {$b_{k}$};
% Text Node
\draw (292,9.4) node [anchor=north west][inner sep=0.75pt]  [font=\normalsize]  {$\Delta _{k}^{+}$};
% Text Node
\draw (297,85.4) node [anchor=north west][inner sep=0.75pt]  [font=\scriptsize]  {$a_{1}$};
% Text Node
\draw (306.03,99.73) node [anchor=north west][inner sep=0.75pt]  [font=\tiny,rotate=-89.49]  {$\dotsc $};
% Text Node
\draw (307,67.4) node [anchor=north west][inner sep=0.75pt]  [font=\scriptsize]  {$c_{1}$};
% Text Node
\draw (336,47.4) node [anchor=north west][inner sep=0.75pt]  [font=\scriptsize]  {$c_{k}$};
% Text Node
\draw (336.24,62.23) node [anchor=north west][inner sep=0.75pt]  [font=\tiny,rotate=-142.02]  {$\dotsc $};
% Text Node
\draw (284,67.4) node [anchor=north west][inner sep=0.75pt]  [font=\scriptsize]  {$b_{1}$};
% Text Node
\draw (254,48.4) node [anchor=north west][inner sep=0.75pt]  [font=\scriptsize]  {$b_{k}$};
% Text Node
\draw (277.69,71.99) node [anchor=north west][inner sep=0.75pt]  [font=\tiny,rotate=-216.17]  {$\dotsc $};
% Text Node
\draw (293,124.4) node [anchor=north west][inner sep=0.75pt]  [font=\scriptsize]  {$a_{k+1}$};

\end{tikzpicture}
\end{multicols}
\end{mainthmC}

To prove this theorem, we first we argue that $\G$ satisfies the hypotheses of Lemma \ref{sides}. Next, we show that $\G$ must be a subgraph $P_{1,k}$ or $P_{3,k}$. Finally, we determine which subgraphs of $P_{3,k}$ are admissible. 

\begin{proof} We can write $f = h \circ f_1$, where $f_1: \G \rightarrow \G'$ is a fold and $h:\G' \rightarrow \G$ is a graph isomorphism. Since $\G'$ must be isomorphic to $\G$, the fold $f_1$ must be a proper full fold, as complete and partial folds change the number of vertices of $\G'$. Hence $f$ must be periodic on the vertex set. Moreover, since $f$ has a fold, $f$ is expanding. Thus by Theorem \hyperref[symmetry and folds]{B}, $\mathfrak{S}(\G) = 1$.

%$f$ has a single stack and the edges can be labeled so that $$f: e_1 \mapsto e_2 \mapsto \dots \mapsto e_n \mapsto f(e_n)$$ where $f(e_n)$ is a path with two edges. 

By the definition of a stack score, there exists a supergraph $G$ of $\G$ and an automorphism $\psi~\in~\text{Aut}(G)$ such that $\sim_{\psi}$ partitions the edges of $\G$ into a single set. By the proof of Theorem \hyperref[symmetry and folds]{B}, we can assume $\psi$ can be written:
\begin{align} \psi: e \mapsto f(e) \mapsto f^2(e) \mapsto \dots \mapsto f^{n-1}(e) \mapsto b_1 \mapsto \dots \mapsto b_j  \mapsto e  \label{eq:3} \end{align}
where $\{e,f(e),\dots, f^{n-1}(e)\} = \mathcal{E}\G$ and $\{b_1, \dots, b_j \} = \mathcal{E}G - \mathcal{E} \G$. 
%Thus (by potentially removing unnecessary edges in $G$), we may assume the edges of $G$ lie in a single $\psi$ orbit. %Argue we can choose $\psi$ so that $\psi =f $ on non-mixing edges. 
 Hence $\psi_V = f_V$, and  $\langle \psi \rangle$ acts transitively on $\mathcal{E}G$. \\%Either this action is also transitive on $\mathcal{V}G$, or not. \\

%\vspace{0.25cm}
\noindent \textit{Claim:} $\langle \psi \rangle$ also acts transitively on $\mathcal{V}G$. \\

%\vspace{0.25cm}
\noindent \textit{Proof of Claim:} Suppose $\langle \psi \rangle $ does not act transitively on $\mathcal{V}G$. By Theorem 2.1 in \cite{graphbook} %\footnote{Theorem 2.1 in \cite{graphbook} assumes $G$ does not have multiple edges or self-loops, but the proof of the theorem is exactly the same if these are allowed.}, 
$G$ is bipartite and the action of $\langle \psi \rangle $ on $\mathcal{V}G$ has two orbits, $X$ and $Y$, which form the partition of $\mathcal{V}G$. Suppose $f_1$ is a proper full fold of $e_1$ over $e_0$. %and let $C$ be a cycle in $\G$ containing $e_1$. Let $k$ be the number of edges in the cycle $C$. Since $\G$ is bipartite, $k$ is even. Since $f_1(e)=e$ for every edge $e \notin \{e_1, \overline{e_1} \}$,  $f_1(C)$ is a cycle of length $k+1$. 
Assume $\iota(e_1)=\iota(e_0) \in X$ and $\tau(e_1),\tau(e_0) \in Y$.  %As a connected subgraph of $G$ with the same vertex set, $\G$ must also be bipartite. 

%\newpage

\begin{multicols}{2} 

 Since $f_1$ is the identity on $\mathcal{V}\G$, $\psi_V=f_V$, and the sets $X$ and $Y$ are invariant under $\psi$, we have 
 $$\iota(h(e_1')),  \tau(h(e_1') \in Y.$$ 
 However, $X$ and $Y$ form the bipartition of $\mathcal{V}G$, so this is a contradiction. Hence $\langle \psi \rangle$ acts transitively on $\mathcal{V}G$.  \hfill $\diamond$
 
 \columnbreak 
 
 \tikzset{every picture/.style={line width=0.75pt}} %set default line width to 0.75pt        

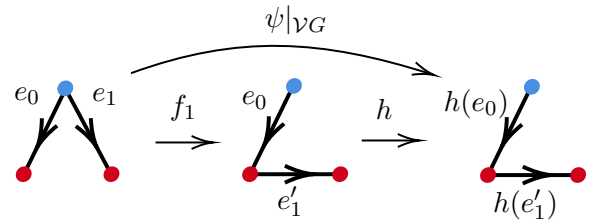
\begin{tikzpicture}[x=0.75pt,y=0.75pt,yscale=-1,xscale=1]
%uncomment if require: \path (0,150); %set diagram left start at 0, and has height of 150

%Straight Lines [id:da8839681078122423] 
\draw [color={rgb, 255:red, 0; green, 0; blue, 0 }  ,draw opacity=1 ][line width=1.5]    (27,92) -- (48,49) ;
\draw [shift={(33.64,78.41)}, rotate = 296.03] [color={rgb, 255:red, 0; green, 0; blue, 0 }  ,draw opacity=1 ][line width=1.5]    (14.21,-4.28) .. controls (9.04,-1.82) and (4.3,-0.39) .. (0,0) .. controls (4.3,0.39) and (9.04,1.82) .. (14.21,4.28)   ;
%Straight Lines [id:da05811286404870719] 
\draw [color={rgb, 255:red, 0; green, 0; blue, 0 }  ,draw opacity=1 ][line width=1.5]    (71,92) -- (48,49) ;
\draw [shift={(63.65,78.26)}, rotate = 241.86] [color={rgb, 255:red, 0; green, 0; blue, 0 }  ,draw opacity=1 ][line width=1.5]    (14.21,-4.28) .. controls (9.04,-1.82) and (4.3,-0.39) .. (0,0) .. controls (4.3,0.39) and (9.04,1.82) .. (14.21,4.28)   ;
%Shape: Ellipse [id:dp8927226104475046] 
\draw  [color={rgb, 255:red, 74; green, 144; blue, 226 }  ,draw opacity=1 ][fill={rgb, 255:red, 74; green, 144; blue, 226 }  ,fill opacity=1 ] (44.8,50.65) .. controls (43.84,49.1) and (44.43,46.94) .. (46.11,45.82) .. controls (47.8,44.71) and (49.95,45.07) .. (50.91,46.62) .. controls (51.87,48.18) and (51.28,50.34) .. (49.6,51.46) .. controls (47.91,52.57) and (45.77,52.21) .. (44.8,50.65) -- cycle ;
%Straight Lines [id:da03119479752252241] 
\draw    (93.37,75.97) -- (120.49,75.97) ;
\draw [shift={(122.49,75.97)}, rotate = 180] [color={rgb, 255:red, 0; green, 0; blue, 0 }  ][line width=0.75]    (10.93,-3.29) .. controls (6.95,-1.4) and (3.31,-0.3) .. (0,0) .. controls (3.31,0.3) and (6.95,1.4) .. (10.93,3.29)   ;
%Straight Lines [id:da44846541515250715] 
\draw [color={rgb, 255:red, 0; green, 0; blue, 0 }  ,draw opacity=1 ][line width=1.5]    (141,92) -- (163,47) ;
\draw [shift={(148.13,77.41)}, rotate = 296.05] [color={rgb, 255:red, 0; green, 0; blue, 0 }  ,draw opacity=1 ][line width=1.5]    (14.21,-4.28) .. controls (9.04,-1.82) and (4.3,-0.39) .. (0,0) .. controls (4.3,0.39) and (9.04,1.82) .. (14.21,4.28)   ;
%Straight Lines [id:da15704206982822733] 
\draw [color={rgb, 255:red, 0; green, 0; blue, 0 }  ,draw opacity=1 ][line width=1.5]    (187,92) -- (141,92) ;
\draw [shift={(172.8,92)}, rotate = 180] [color={rgb, 255:red, 0; green, 0; blue, 0 }  ,draw opacity=1 ][line width=1.5]    (14.21,-4.28) .. controls (9.04,-1.82) and (4.3,-0.39) .. (0,0) .. controls (4.3,0.39) and (9.04,1.82) .. (14.21,4.28)   ;
%Shape: Ellipse [id:dp5037101429255253] 
\draw  [color={rgb, 255:red, 208; green, 2; blue, 27 }  ,draw opacity=1 ][fill={rgb, 255:red, 208; green, 2; blue, 27 }  ,fill opacity=1 ] (68.1,94) .. controls (67.14,92.44) and (67.73,90.28) .. (69.41,89.17) .. controls (71.1,88.06) and (73.24,88.41) .. (74.21,89.97) .. controls (75.17,91.53) and (74.58,93.69) .. (72.9,94.8) .. controls (71.21,95.91) and (69.06,95.56) .. (68.1,94) -- cycle ;
%Shape: Ellipse [id:dp7081561753084569] 
\draw  [color={rgb, 255:red, 208; green, 2; blue, 27 }  ,draw opacity=1 ][fill={rgb, 255:red, 208; green, 2; blue, 27 }  ,fill opacity=1 ] (23.83,94) .. controls (22.87,92.44) and (23.46,90.28) .. (25.15,89.17) .. controls (26.83,88.06) and (28.98,88.41) .. (29.94,89.97) .. controls (30.9,91.53) and (30.31,93.69) .. (28.63,94.8) .. controls (26.94,95.91) and (24.8,95.56) .. (23.83,94) -- cycle ;
%Shape: Ellipse [id:dp8381743296778406] 
\draw  [color={rgb, 255:red, 208; green, 2; blue, 27 }  ,draw opacity=1 ][fill={rgb, 255:red, 208; green, 2; blue, 27 }  ,fill opacity=1 ] (138,94) .. controls (137.04,92.44) and (137.63,90.28) .. (139.31,89.17) .. controls (141,88.06) and (143.14,88.41) .. (144.1,89.97) .. controls (145.07,91.53) and (144.48,93.69) .. (142.79,94.8) .. controls (141.11,95.91) and (138.96,95.56) .. (138,94) -- cycle ;
%Shape: Ellipse [id:dp8106135932598615] 
\draw  [color={rgb, 255:red, 74; green, 144; blue, 226 }  ,draw opacity=1 ][fill={rgb, 255:red, 74; green, 144; blue, 226 }  ,fill opacity=1 ] (160.14,49.45) .. controls (159.17,47.9) and (159.76,45.73) .. (161.45,44.62) .. controls (163.13,43.51) and (165.28,43.87) .. (166.24,45.42) .. controls (167.2,46.98) and (166.61,49.14) .. (164.93,50.25) .. controls (163.24,51.36) and (161.1,51.01) .. (160.14,49.45) -- cycle ;
%Shape: Ellipse [id:dp8950018864271942] 
\draw  [color={rgb, 255:red, 208; green, 2; blue, 27 }  ,draw opacity=1 ][fill={rgb, 255:red, 208; green, 2; blue, 27 }  ,fill opacity=1 ] (183.44,94) .. controls (182.47,92.44) and (183.06,90.28) .. (184.75,89.17) .. controls (186.43,88.06) and (188.58,88.41) .. (189.54,89.97) .. controls (190.5,91.53) and (189.91,93.69) .. (188.23,94.8) .. controls (186.54,95.91) and (184.4,95.56) .. (183.44,94) -- cycle ;
%Straight Lines [id:da0742531516491367] 
\draw [color={rgb, 255:red, 0; green, 0; blue, 0 }  ,draw opacity=1 ][line width=1.5]    (261,92.43) -- (283,47.43) ;
\draw [shift={(268.13,77.83)}, rotate = 296.05] [color={rgb, 255:red, 0; green, 0; blue, 0 }  ,draw opacity=1 ][line width=1.5]    (14.21,-4.28) .. controls (9.04,-1.82) and (4.3,-0.39) .. (0,0) .. controls (4.3,0.39) and (9.04,1.82) .. (14.21,4.28)   ;
%Straight Lines [id:da15807747627443036] 
\draw [color={rgb, 255:red, 0; green, 0; blue, 0 }  ,draw opacity=1 ][line width=1.5]    (307,92.43) -- (261,92.43) ;
\draw [shift={(292.8,92.43)}, rotate = 180] [color={rgb, 255:red, 0; green, 0; blue, 0 }  ,draw opacity=1 ][line width=1.5]    (14.21,-4.28) .. controls (9.04,-1.82) and (4.3,-0.39) .. (0,0) .. controls (4.3,0.39) and (9.04,1.82) .. (14.21,4.28)   ;
%Shape: Ellipse [id:dp048656370027181595] 
\draw  [color={rgb, 255:red, 208; green, 2; blue, 27 }  ,draw opacity=1 ][fill={rgb, 255:red, 208; green, 2; blue, 27 }  ,fill opacity=1 ] (258,94.43) .. controls (257.04,92.87) and (257.63,90.71) .. (259.31,89.6) .. controls (261,88.48) and (263.14,88.84) .. (264.1,90.4) .. controls (265.07,91.95) and (264.48,94.12) .. (262.79,95.23) .. controls (261.11,96.34) and (258.96,95.98) .. (258,94.43) -- cycle ;
%Shape: Ellipse [id:dp8962309069102907] 
\draw  [color={rgb, 255:red, 74; green, 144; blue, 226 }  ,draw opacity=1 ][fill={rgb, 255:red, 74; green, 144; blue, 226 }  ,fill opacity=1 ] (280.14,49.88) .. controls (279.17,48.32) and (279.76,46.16) .. (281.45,45.05) .. controls (283.13,43.93) and (285.28,44.29) .. (286.24,45.85) .. controls (287.2,47.4) and (286.61,49.57) .. (284.93,50.68) .. controls (283.24,51.79) and (281.1,51.43) .. (280.14,49.88) -- cycle ;
%Shape: Ellipse [id:dp9807700002713167] 
\draw  [color={rgb, 255:red, 208; green, 2; blue, 27 }  ,draw opacity=1 ][fill={rgb, 255:red, 208; green, 2; blue, 27 }  ,fill opacity=1 ] (303.44,94.43) .. controls (302.47,92.87) and (303.06,90.71) .. (304.75,89.6) .. controls (306.43,88.48) and (308.58,88.84) .. (309.54,90.4) .. controls (310.5,91.95) and (309.91,94.12) .. (308.23,95.23) .. controls (306.54,96.34) and (304.4,95.98) .. (303.44,94.43) -- cycle ;
%Straight Lines [id:da43032845291285704] 
\draw    (197.37,74.97) -- (224.49,74.97) ;
\draw [shift={(226.49,74.97)}, rotate = 180] [color={rgb, 255:red, 0; green, 0; blue, 0 }  ][line width=0.75]    (10.93,-3.29) .. controls (6.95,-1.4) and (3.31,-0.3) .. (0,0) .. controls (3.31,0.3) and (6.95,1.4) .. (10.93,3.29)   ;
%Curve Lines [id:da3627064090257597] 
\draw    (81,41) .. controls (135.45,21.2) and (177.16,23.94) .. (232.32,41.46) ;
\draw [shift={(234,42)}, rotate = 197.82] [color={rgb, 255:red, 0; green, 0; blue, 0 }  ][line width=0.75]    (10.93,-3.29) .. controls (6.95,-1.4) and (3.31,-0.3) .. (0,0) .. controls (3.31,0.3) and (6.95,1.4) .. (10.93,3.29)   ;

% Text Node
\draw (20.49,46.64) node [anchor=north west][inner sep=0.75pt]  [font=\normalsize,color={rgb, 255:red, 0; green, 0; blue, 0 }  ,opacity=1 ]  {$e_{0}$};
% Text Node
\draw (59.87,47.05) node [anchor=north west][inner sep=0.75pt]  [font=\normalsize,color={rgb, 255:red, 0; green, 0; blue, 0 }  ,opacity=1 ]  {$e_{1}$};
% Text Node
\draw (135.33,48.97) node [anchor=north west][inner sep=0.75pt]  [font=\normalsize]  {$e_{0}$};
% Text Node
\draw (153.15,96.81) node [anchor=north west][inner sep=0.75pt]  [font=\normalsize]  {$e_{1} '$};
% Text Node
\draw (99,50.4) node [anchor=north west][inner sep=0.75pt]    {$f_{1}$};
% Text Node
\draw (237.33,47.4) node [anchor=north west][inner sep=0.75pt]  [font=\normalsize]  {$h( e_{0})$};
% Text Node
\draw (262.15,98.23) node [anchor=north west][inner sep=0.75pt]  [font=\normalsize]  {$h( e_{1} ')$};
% Text Node
\draw (203,52.4) node [anchor=north west][inner sep=0.75pt]    {$h$};
% Text Node
\draw (147,5.4) node [anchor=north west][inner sep=0.75pt]    {$\psi_V$};

\end{tikzpicture}
 
 \end{multicols}

%However, $\G'$ must be isomorphic to $\G$, so this is a contradiction. Hence $\langle \psi \rangle$ must act transitively on both $\mathcal{E}G$ and $\mathcal{V}G$. 

% Suppose the initial vertex of $e_1$ lies in $X$. Since $X$ is a $\psi$ orbit, the initial vertex of each edge $a_i$ also lies in $X$. Since $Y$ must be non-empty, the terminal vertex of every edge $a_i$ must lie in $Y$. Thus $f(e_n)$ is a path from a vertex in $X$ to a vertex in $Y$. However, on a bipartite graph, no such length two path can exist, since every length two path begins and ends at vertices in the same set of the partition. Hence $\langle \psi \rangle$ must act transitively on both $\mathcal{E}(G)$ and $\mathcal{V}(G)$. \\

By Lemma \ref{transitive}, $G$ is an $s$-gonal graph of depth $k$, for some $s, k \in \mathbb{N}$. Hence by \eqref{eq:3}, $\G$ satisfies the hypotheses of Lemma \ref{sides}.\\ 

We now argue that in fact $G$ is either a $1$-gonal graph (and hence isomorphic to a rose $R_k$) or a $3$-gonal graph. \\

\noindent \textit{Claim:} If $s \geq 4$ then $\G'$ cannot be isomorphic to $\G$. \\
 
 \noindent \textit{Proof of Claim:} Suppose $s \geq 4$. %and $\G \subseteq G$ is a connected subgraph with $\mathcal{V}\G=\mathcal{V}G$.
 A single proper full fold between edges in $\G$ in the same side yields a graph $\G'$ with a self loop, and hence $\G'$ is not isomorphic to $\G$. Otherwise, the single fold $f_1: \G \rightarrow \G'$ must be between edges in adjacent sides.  Without loss of generality, suppose $f_1$ is the proper full fold of an edge $a$ from $v_1$ to $v_0$ over an edge $b$ from $v_1$ to $v_2$.  By definition of proper full fold, $\G'$ has an edge $a'$ from $v_2$ to $v_0$. Since the valence of every vertex in $\G$ is at least 3, Lemma \ref{sides} guarantees that for each side $\mathfrak{s}_i$ of $G$, we have $$|\mathfrak{s}_i \cap \mathcal{E}\G| \geq 1.$$ 
 
 \begin{multicols}{2}
Therefore, there must be an edge $c \in  \mathcal{E} \G$ from $v_2$ to $v_3$. Observe that in $\G'$, the vertex $v_2$ is adjacent to vertices $v_0,v_1,$ and $v_3$. Observe that every vertex in a subgraph of an $s$-gonal graph is adjacent to at most two vertices. Hence $\G'$ cannot be isomorphic to $\G$. \hfill $\diamond$
  
\columnbreak

\tikzset{every picture/.style={line width=0.75pt}} %set default line width to 0.75pt        

\begin{tikzpicture}[x=0.75pt,y=0.75pt,yscale=-1,xscale=1]
%uncomment if require: \path (0,351); %set diagram left start at 0, and has height of 351

%Straight Lines [id:da06669082037148955] 
\draw [color={rgb, 255:red, 0; green, 0; blue, 0 }  ,draw opacity=1 ][line width=1.5]    (44.25,82.94) -- (60.59,35.1) ;
\draw [shift={(49.57,67.35)}, rotate = 288.86] [color={rgb, 255:red, 0; green, 0; blue, 0 }  ,draw opacity=1 ][line width=1.5]    (14.21,-6.37) .. controls (9.04,-2.99) and (4.3,-0.87) .. (0,0) .. controls (4.3,0.87) and (9.04,2.99) .. (14.21,6.37)   ;
%Shape: Ellipse [id:dp3454973235665748] 
\draw  [fill={rgb, 255:red, 0; green, 0; blue, 0 }  ,fill opacity=1 ] (59.62,38.07) .. controls (58.15,37.57) and (57.4,35.84) .. (57.94,34.21) .. controls (58.47,32.57) and (60.09,31.64) .. (61.56,32.14) .. controls (63.02,32.63) and (63.78,34.36) .. (63.24,36) .. controls (62.71,37.63) and (61.08,38.56) .. (59.62,38.07) -- cycle ;
%Straight Lines [id:da3405668404319657] 
\draw [color={rgb, 255:red, 0; green, 0; blue, 0 }  ,draw opacity=1 ][line width=1.5]    (101.34,36.43) -- (60.59,35.1) ;
\draw [shift={(89.76,36.06)}, rotate = 181.87] [color={rgb, 255:red, 0; green, 0; blue, 0 }  ,draw opacity=1 ][line width=1.5]    (14.21,-6.37) .. controls (9.04,-2.99) and (4.3,-0.87) .. (0,0) .. controls (4.3,0.87) and (9.04,2.99) .. (14.21,6.37)   ;
%Shape: Ellipse [id:dp621778308903032] 
\draw  [fill={rgb, 255:red, 0; green, 0; blue, 0 }  ,fill opacity=1 ] (98.36,35.52) .. controls (98.83,34.05) and (100.54,33.26) .. (102.19,33.77) .. controls (103.84,34.27) and (104.79,35.87) .. (104.33,37.35) .. controls (103.86,38.82) and (102.15,39.61) .. (100.5,39.1) .. controls (98.85,38.6) and (97.89,37) .. (98.36,35.52) -- cycle ;
%Straight Lines [id:da15012183573560223] 
\draw [color={rgb, 255:red, 0; green, 0; blue, 0 }  ,draw opacity=1 ][line width=1.5]    (117.49,83.52) -- (101.34,36.43) ;
\draw [shift={(112.27,68.3)}, rotate = 251.07] [color={rgb, 255:red, 0; green, 0; blue, 0 }  ,draw opacity=1 ][line width=1.5]    (14.21,-6.37) .. controls (9.04,-2.99) and (4.3,-0.87) .. (0,0) .. controls (4.3,0.87) and (9.04,2.99) .. (14.21,6.37)   ;
%Shape: Ellipse [id:dp7205483484653741] 
\draw  [fill={rgb, 255:red, 0; green, 0; blue, 0 }  ,fill opacity=1 ] (45.51,85.79) .. controls (44.09,86.4) and (42.38,85.62) .. (41.68,84.05) .. controls (40.98,82.47) and (41.56,80.7) .. (42.98,80.09) .. controls (44.4,79.47) and (46.12,80.25) .. (46.82,81.83) .. controls (47.52,83.4) and (46.93,85.18) .. (45.51,85.79) -- cycle ;
%Shape: Ellipse [id:dp6428737440148855] 
\draw  [fill={rgb, 255:red, 0; green, 0; blue, 0 }  ,fill opacity=1 ] (120.44,84.54) .. controls (119.92,85.99) and (118.18,86.72) .. (116.55,86.15) .. controls (114.92,85.59) and (114.02,83.95) .. (114.54,82.49) .. controls (115.06,81.04) and (116.81,80.32) .. (118.43,80.88) .. controls (120.06,81.44) and (120.96,83.08) .. (120.44,84.54) -- cycle ;
%Straight Lines [id:da9850329953502712] 
\draw    (154,59) -- (213,59) ;
\draw [shift={(215,59)}, rotate = 180] [color={rgb, 255:red, 0; green, 0; blue, 0 }  ][line width=0.75]    (10.93,-3.29) .. controls (6.95,-1.4) and (3.31,-0.3) .. (0,0) .. controls (3.31,0.3) and (6.95,1.4) .. (10.93,3.29)   ;
%Straight Lines [id:da8037944514067354] 
\draw [color={rgb, 255:red, 0; green, 0; blue, 0 }  ,draw opacity=1 ][line width=1.5]    (245.25,84.94) -- (302.34,38.43) ;
\draw [shift={(266.97,67.24)}, rotate = 320.84] [color={rgb, 255:red, 0; green, 0; blue, 0 }  ,draw opacity=1 ][line width=1.5]    (14.21,-6.37) .. controls (9.04,-2.99) and (4.3,-0.87) .. (0,0) .. controls (4.3,0.87) and (9.04,2.99) .. (14.21,6.37)   ;
%Shape: Ellipse [id:dp6896938662080281] 
\draw  [fill={rgb, 255:red, 0; green, 0; blue, 0 }  ,fill opacity=1 ] (260.62,40.07) .. controls (259.15,39.57) and (258.4,37.84) .. (258.94,36.21) .. controls (259.47,34.57) and (261.09,33.64) .. (262.56,34.14) .. controls (264.02,34.63) and (264.78,36.36) .. (264.24,38) .. controls (263.71,39.63) and (262.08,40.56) .. (260.62,40.07) -- cycle ;
%Straight Lines [id:da8402430612198517] 
\draw [color={rgb, 255:red, 0; green, 0; blue, 0 }  ,draw opacity=1 ][line width=1.5]    (302.91,38.61) -- (261.59,37.1) ;
\draw [shift={(291.04,38.18)}, rotate = 182.09] [color={rgb, 255:red, 0; green, 0; blue, 0 }  ,draw opacity=1 ][line width=1.5]    (14.21,-6.37) .. controls (9.04,-2.99) and (4.3,-0.87) .. (0,0) .. controls (4.3,0.87) and (9.04,2.99) .. (14.21,6.37)   ;
%Shape: Ellipse [id:dp7232689702613717] 
\draw  [fill={rgb, 255:red, 0; green, 0; blue, 0 }  ,fill opacity=1 ] (299.36,37.52) .. controls (299.83,36.05) and (301.54,35.26) .. (303.19,35.77) .. controls (304.84,36.27) and (305.79,37.87) .. (305.33,39.35) .. controls (304.86,40.82) and (303.15,41.61) .. (301.5,41.1) .. controls (299.85,40.6) and (298.89,39) .. (299.36,37.52) -- cycle ;
%Straight Lines [id:da8678643331484233] 
\draw [color={rgb, 255:red, 0; green, 0; blue, 0 }  ,draw opacity=1 ][line width=1.5]    (318.49,86.52) -- (303.35,39.09) ;
\draw [shift={(313.6,71.19)}, rotate = 252.29] [color={rgb, 255:red, 0; green, 0; blue, 0 }  ,draw opacity=1 ][line width=1.5]    (14.21,-6.37) .. controls (9.04,-2.99) and (4.3,-0.87) .. (0,0) .. controls (4.3,0.87) and (9.04,2.99) .. (14.21,6.37)   ;
%Shape: Ellipse [id:dp8435024846904187] 
\draw  [fill={rgb, 255:red, 0; green, 0; blue, 0 }  ,fill opacity=1 ] (246.51,87.79) .. controls (245.09,88.4) and (243.38,87.62) .. (242.68,86.05) .. controls (241.98,84.47) and (242.56,82.7) .. (243.98,82.09) .. controls (245.4,81.47) and (247.12,82.25) .. (247.82,83.83) .. controls (248.52,85.4) and (247.93,87.18) .. (246.51,87.79) -- cycle ;
%Shape: Ellipse [id:dp9589728101794119] 
\draw  [fill={rgb, 255:red, 0; green, 0; blue, 0 }  ,fill opacity=1 ] (321.44,87.54) .. controls (320.92,88.99) and (319.18,89.72) .. (317.55,89.15) .. controls (315.92,88.59) and (315.02,86.95) .. (315.54,85.49) .. controls (316.06,84.04) and (317.81,83.32) .. (319.43,83.88) .. controls (321.06,84.44) and (321.96,86.08) .. (321.44,87.54) -- cycle ;

% Text Node
\draw (23,77.4) node [anchor=north west][inner sep=0.75pt]  [font=\footnotesize]  {$v_{0}$};
% Text Node
\draw (53,17.4) node [anchor=north west][inner sep=0.75pt]  [font=\footnotesize]  {$v_{1}$};
% Text Node
\draw (106,22.4) node [anchor=north west][inner sep=0.75pt]  [font=\footnotesize]  {$v_{2}$};
% Text Node
\draw (34,44.4) node [anchor=north west][inner sep=0.75pt]  [color={rgb, 255:red, 0; green, 0; blue, 0 }  ,opacity=1 ]  {$a$};
% Text Node
\draw (78,10.4) node [anchor=north west][inner sep=0.75pt]  [color={rgb, 255:red, 0; green, 0; blue, 0 }  ,opacity=1 ]  {$b$};
% Text Node
\draw (118,40.4) node [anchor=north west][inner sep=0.75pt]  [color={rgb, 255:red, 0; green, 0; blue, 0 }  ,opacity=1 ]  {$c$};
% Text Node
\draw (266,69.4) node [anchor=north west][inner sep=0.75pt]  [color={rgb, 255:red, 0; green, 0; blue, 0 }  ,opacity=1 ]  {$a'$};
% Text Node
\draw (177,35.4) node [anchor=north west][inner sep=0.75pt]    {$f_{1}$};
% Text Node
\draw (126,80.4) node [anchor=north west][inner sep=0.75pt]  [font=\footnotesize]  {$v_{3}$};
% Text Node
\draw (254,19.4) node [anchor=north west][inner sep=0.75pt]  [font=\footnotesize]  {$v_{1}$};
% Text Node
\draw (307,24.4) node [anchor=north west][inner sep=0.75pt]  [font=\footnotesize]  {$v_{2}$};
% Text Node
\draw (327,82.4) node [anchor=north west][inner sep=0.75pt]  [font=\footnotesize]  {$v_{3}$};
% Text Node
\draw (280,11.4) node [anchor=north west][inner sep=0.75pt]  [color={rgb, 255:red, 0; green, 0; blue, 0 }  ,opacity=1 ]  {$b$};
% Text Node
\draw (319,45.4) node [anchor=north west][inner sep=0.75pt]  [color={rgb, 255:red, 0; green, 0; blue, 0 }  ,opacity=1 ]  {$c$};
% Text Node
\draw (225,80.4) node [anchor=north west][inner sep=0.75pt]  [font=\footnotesize]  {$v_{0}$};

\end{tikzpicture}

\end{multicols}
  
%If there are no edges in $\mathcal{E}\G$ from $v_2$ to $v_3$, then in order for the valence $v_2$ to be at least $3$, there must be at least $3$ edges from $v_2$ to $v_3$.  However 
 
Since $h:\G' \rightarrow \G$ is a graph isomorphism, $\G'$ must be isomorphic to $\G$. Hence $1 \leq s \leq 3$. \\

 % Recall that a \textit{cycle} in a graph is an edge path with the first vertex equal to the last, and no other vertex repeated. If $s \geq 4$, no subgraphs of $G$ have a cycle of length $3$ nor a cycle of length $s-1$.  However, if $\G$ is a connected subgraph of $G$, then any single proper full fold between edges in adjacent sides yields a graph $\G'$ which has either a cycle of length $3$ or a cycle of length $s-1$. Thus, $\G'$ cannot be isomorphic to $\G$, contradicting that $h$ is a graph isomorphism. Hence $1\leq s \leq 3$. \\
 If $s=1$, then $\G\cong R_k$. Any subgraph of $R_k$ is another rose $R_j$ for some $j \leq k$. Since the rank of $R_k$ is equal to $k$, we can build a rose with any rank. \\
 
 If $s=2$, then $G$ is a (1,1)-bipartite graph. As a connected non-empty subgraph of $G$, the graph $\G$ is also a (1,1)-bipartite graph. Any single proper full fold in $\G$ yields an edge $e_1'$ with $\iota(e_1')=\tau(e_1')$. Hence $\G'$ is not bipartite, and thus not isomorphic to $\G$, a contradiction. Hence $s \in \{1 , 3\}$.\\

Suppose $s=3$. Then $G\cong P_{3,k}$. % Using an isomorphism from $P_{3,k}$ to $G$, partition the edges of $G$ into sides $\mathfrak{s}_0,\mathfrak{s}_1,$ and $\mathfrak{s}_2$. Let $e \in \mathcal{E}\G$ be the root edge of the single stack in $\G$. Without loss of generality, assume $e \in \mathfrak{s_0}$. If $\psi(e) \in \mathfrak{s}_0$, then $\psi$ would not act transitively on the vertices of $G$. Thus, without loss of generality we may assume $\psi(e) \in \mathfrak{s}_1$. Similarly, we may assume $\psi^2(e) \in \mathfrak{s}_2$ and $\psi^3(e) \in \mathfrak{s}_0$. By continuity, $\psi$ maps every edge in $\mathfrak{s}_0$ to an edge in $\mathfrak{s}_1$, every edge in $\mathfrak{s}_1$ to an edge in $\mathfrak{s}_2$, and every edge in $\mathfrak{s}_2$ to an edge in $\mathfrak{s}_0$. More succinctly, 
% $$\psi: \mathfrak{s}_0 \mapsto \mathfrak{s}_1 \mapsto \mathfrak{s}_2 \mapsto \mathfrak{s}_0.$$
%As mentioned earlier, by the proof of  \hyperref[symmetry and folds]{B}, we can write $\psi$ on the edges of $G$ as%Assume $ \psi$ sends $v_0 \mapsto v_1 \mapsto v_2 \mapsto v_0$. 
%$$\psi: e \mapsto f(e) \mapsto f^2(e) \mapsto \dots \mapsto f^{n-1}(e) \mapsto b_1 \mapsto \dots \mapsto b_j  \mapsto e$$
%where $\{e,f(e),\dots, f^{n-1}(e)\} = \mathcal{E}\G$ and $\{b_1, \dots, b_j \} = \mathcal{E}G - \mathcal{E} \G$. Thus,
%\begin{align*}
%\mathfrak{s}_0 \cap \mathcal{E}\G & = \{f^{m}(e) \ |  \ m \in \{0,1,\dots,n-1\} \ \text{and} \ m \equiv 0 \text{ mod }3 \}, \\
%\mathfrak{s}_1 \cap \mathcal{E}\G & =\{f^{m}(e) \ | \ m \in \{0,1,\dots,n-1\} \ \text{and} \ m \equiv 1 \text{ mod }3 \}, \text{ and} \\
%\mathfrak{s}_2 \cap \mathcal{E}\G & =\{f^{m}(e) \ | \ m \in \{0,1,\dots,n-1\} \ \text{and} \ m \equiv 2 \text{ mod }3 \}.
%\end{align*}
By Lemma \ref{sides}, up to relabeling of the sides $\mathfrak{s}_i$, we have the following three cases: \\
\begin{itemize}
\item [(i)]If $n = 3m$ for some $m \in \mathbb{N}$, then $$(|(\mathfrak{s}_0 \cap \mathcal{E}\G)|,|(\mathfrak{s}_1 \cap \mathcal{E}\G)|,|(\mathfrak{s}_2 \cap \mathcal{E}\G)|)  = (m,m,m).$$ 
Hence $\G \cong P_{3,m}$. \\
\item[(ii)]If $n = 3m+1$ for some $m \in \mathbb{N}$, then $$(|(\mathfrak{s}_0 \cap \mathcal{E}\G)|,|(\mathfrak{s}_1 \cap \mathcal{E}\G)|,|(\mathfrak{s}_2 \cap \mathcal{E}\G)|)  = (m+1,m,m).$$
Hence $\G \cong \Delta_{m}^+.$ \\
\item[(iii)] If $n = 3m+2$ for some $m \in \mathbb{N}$, then $$(|(\mathfrak{s}_0 \cap \mathcal{E}\G)|,|(\mathfrak{s}_1 \cap \mathcal{E}\G)|,|(\mathfrak{s}_2 \cap \mathcal{E}\G)|)  = (m+1,m+1,m).$$
Hence $\G \cong \Delta_{m+1}^-$. \\
\end{itemize} 
Observe that when $s=3$, we have $|\mathcal{V}\G|=3$. Hence by the Euler characteristic formula, the rank $r$ of $\G$ is computed as \begin{align*} r = & \ |\mathcal{E}\G| - |\mathcal{V}\G| +1\\  = & \ n-2. \end{align*} Thus, the above cases correspond to $r \equiv 1,2,0$ mod $3$ respectively. \\
%Side depths equal to $(k,k,k-1)$ or $(k,k,k+1)$ imply that $\G$ is isomorphic to $\Delta_k^-$ or $\Delta_k^+$. Note that $k \geq 2$ is necessary for valence requirements.  

Now we need only rule out the possibility that $\G$ is isomorphic to $P_{3,m}$. In this case, any single proper full fold yields a graph with a self loop or a $3-$gonal graph with side depths $(m,m-1,m+1)$. Hence $\G'$ is not isomorphic to $P_{3,m}$, a contradiction. \end{proof}

\bigskip
%\newpage

\section{Further Observations and Questions}\label{further}

\subsection{Unique Minimizer in Out$(F_3)$} We have the following application of Theorems \hyperref[lower bound]{A} and \hyperref[single folds]{C}. 

\begin{cor}\label{unique min} The element $\varphi \in $ Out$(F_3)$ given by $\varphi: x \mapsto y \mapsto z \mapsto zx^{-1}$ defines the unique Out$(F_3)-$conjugacy class of infinite order irreducible elements realizing the minimal stretch factor $\lambda \approx 1.167$, the largest real root of $x^5-x-1$. 
\end{cor}

 \begin{proof} The element $\varphi$ is Example \ref{ahlp single fold}. It is shown in \cite{wiggd1} that $\varphi$ has stretch factor $\lambda(\varphi) \approx 1.167$, the largest real root of $x^5-x-1$. % and this stretch factor is minimal among fully irreducible elements of Out$(F_3)$. 
 Suppose $\phi \in $ Out$(F_3)$ is an infinite order irreducible element with $\lambda(\phi) \leq \lambda(\varphi)$. Let $f:\G \rightarrow \G$ be an irreducible train track representative of $\phi$ on a connected rank $3$ graph $\G$.  Since $\phi$ is infinite order, $\lambda_f>1$ by Theorem \ref{bestvina}. Thus by Lemma \ref{helpful facts}, $f$ must have at least one fold in its fold decomposition. Since 
 $$\lambda_f \leq \lambda(\varphi) < 2^{\frac{1}{4}} < 3^{\frac{1}{6}},$$ by Theorem \hyperref[lower bound]{A}, $f$ must have exactly one fold in its fold decomposition and $\G$ must have at least 5 edges. As the vertices of $\G$ have valence at least 3 and $\G$ has rank 3, an Euler characteristic argument shows $\G$ can have no more than 6 edges. Hence by Theorem \hyperref[single folds]{C}, $\G \cong \Delta_2^-$.  
 
Suppose $f=h \circ f_1 $ is a fold decomposition, so $f_1:\G \rightarrow \G'$ is a proper full fold and $h: \G' \rightarrow \G$ is a graph isomorphism. Up to relabeling the edges, the only proper full fold on $\Delta_2^-$ which yields an isomorphic graph is the proper full fold of $c_2$ over $\overline{b_1}$. Without loss of generality, suppose $\G = \Delta_2^-$, give $\G$ the labels in Example \ref{ahlp single fold}, and assume $f_1$ is the proper full fold of $c_2$ over $\overline{b_1}$. By continuity, we must have $h(c_1) \in \{a_1,\overline{a_1}\}$.

%\begin{itemize}
%\item[1)]
Suppose $h(c_1) = \overline{a_1}$. If $h(a_1) =\overline{c_1}$, then $f(c_1)=\overline{a_1}$ and $f(a_1)=\overline{c_1}$, so $f$ is reducible. This leaves two ways $h$ could map the remaining edges:
\begin{itemize}
\item[(i)]$h: a_1 \mapsto \overline{c_2}$, $c_2' \mapsto c_1$, $b_1 \mapsto  \overline{b_1}$, and $b_2 \mapsto \overline{b_2}$. \\
In this case $f(b_1)=\overline{b_1}$, so $f$ is reducible.
\item[(ii)] $h: a_1 \mapsto \overline{c_2}$. $c_2' \mapsto c_1$. $b_1 \mapsto  \overline{b_2}$. $b_2 \mapsto \overline{b_1}$. \\
In this case, $f(b_1) = \overline{b_2}$ and $f(b_2) = \overline{b_1}$, so again $f$ is reducible. 

\end{itemize}
 Thus $h(c_1)\neq \overline{a_1}$, so we must have $h(c_1)=a_1$. Then $h$ maps the remaining edges in one of the following four ways:
\begin{itemize}
\item[(i)] $h: b_1 \mapsto c_1, \ b_2 \mapsto c_2$, $ \ a_1 \mapsto  b_2$,  and $c_2' \mapsto \overline{b_1}$. \\
In this case,  $f$ is equal to $\mathfrak{g}$ in Example \ref{ahlp single fold} and hence $\phi$ is Out$(F_3)-$conjugate to $\varphi$. 
\item[(ii)] $h: b_1 \mapsto c_2$, $ \ b_2 \mapsto c_1$, $ \ a_1 \mapsto  b_2$, and $c_2' \mapsto \overline{b_1}$. \\
In this case, we have $f(a_1)= b_2$, $f(b_2) = c_1$ and $f(c_1)=a_1$, so $f$ is reducible.
\item[(iii)]  $h: b_1 \mapsto c_1$, $ \ b_2 \mapsto c_2$, $ \ a_1 \mapsto  b_1 $, $ \ c_2' \mapsto \overline{b_2}$. \\
In this case, we have $f(a_1) = b_1$, $f(b_1)=c_1$, and $f(c_1)=a_1$, so $f$ is reducible.. 
\item[(iv)]  $h: b_1 \mapsto c_2$, $ \ b_2 \mapsto c_1$, $ \ a_1 \mapsto  b_1 $,  $ \ c_2' \mapsto \overline{b_2}$. \\
In this case, we have $f: b_2 \mapsto c_1 \mapsto a_1 \mapsto b_1 \mapsto c_2 \mapsto \overline{b_2}\overline{c_2}$. Then $\lambda_f$ is equal to the largest root of $x^5-x^4-1$, which is larger than $\lambda(\varphi)$.\\
\end{itemize}

%\end{itemize}

Therefore, if $\phi$ is an infinite order irreducible element of Out$(F_3)$ with $\lambda(\phi) \leq \lambda(\varphi)$, then $\phi$ is Out$(F_3)-$conjugate to $\varphi$, and hence has $\lambda(\phi)=\lambda(\varphi)$.  \end{proof}

\smallskip
 %\newpage 
 \subsection{Single Fold Irreducible Train Track on a Disconnected Graph} The hypothesis that $\G$ is connected in Theorem \hyperref[single folds]{C} is in fact necessary. 
 
 \begin{ex}\label{single fold disconnected} Let $\G$ be the graph consisting of the union of two disjoint copies $\Delta_2^-$. For the first copy of $\Delta_2^-$, use the same labels for edges as in Example \ref{ahlp single fold}, and use $a_1', b_1',b_2', c_1',$ and $c_2'$ as edge labels for the second copy of $\Delta_2^-$. Now define $f:\G \rightarrow \G$ by
 $$f: \begin{cases} 
 b_1   \mapsto b_1' \mapsto c_1 \\
 c_1  \mapsto c_1' \mapsto a_1 \\
 a_1  \mapsto a_1' \mapsto b_2 \\
 b_2 \mapsto b_2' \mapsto c_2 \\
 c_2 \mapsto c_2' \mapsto \overline{b_1} \overline{c_1}
 \end{cases}
 $$
 Then $f$ is a single fold irreducible train track map and the leading eigenvalue of $T(f)$ is $\lambda^{\frac{1}{2}}$ for $\lambda$ equal to the largest root of $x^5 - x -1$. By taking $n$ copies of $\Delta_2^-$, this example can be generalized to build a single fold irreducible train track map with leading eigenvalue $\lambda^{\frac{1}{n}}$.  %This example can be tweaked to build a relative train track map on a connected graph which induces a reducible outer automorphism with stretch factor $\lambda^{\frac{1}{n}}$. %
 However, when $\G$ is disconnected, homotopy equivalences on $\G$ don't correspond to outer automorphisms of $F_r$. 
 \end{ex}
 
 \bigskip
 
 \subsection{Candidate for Minimal Rank 4 Stretch Factor} By Theorem \hyperref[single folds]{C}, the only single fold i.t.t. maps on connected rank 4 graphs are on $R_4$. Among the single folds on $R_4$, the map sending $e_1 \mapsto e_2 \mapsto e_3 \mapsto e_4 \mapsto e_1e_2$ has the smallest stretch factor, which is the largest root of $x^4-x-1$, approximately $1.221$. However, this is not minimal in Out$(F_4)$. 
 
 \begin{ex}Consider the following single stack, 2 fold irreducible train track map $\gamma$ on a subgraph of the 4-gonal graph of depth 2:
 \begin{center}
 \tikzset{every picture/.style={line width=0.75pt}} %set default line width to 0.75pt        

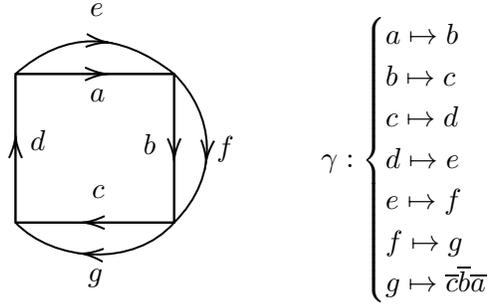
\begin{tikzpicture}[x=0.75pt,y=0.75pt,yscale=-1,xscale=1]

%Shape: Rectangle [id:dp7603928725628157] 
\draw   (34,35.5) -- (114,35.5) -- (114,110.5) -- (34,110.5) -- cycle ;
%Straight Lines [id:da8414068022729979] 
\draw    (34,35.5) -- (114,35.5) ;
\draw [shift={(80,35.5)}, rotate = 180] [color={rgb, 255:red, 0; green, 0; blue, 0 }  ][line width=0.75]    (10.93,-3.29) .. controls (6.95,-1.4) and (3.31,-0.3) .. (0,0) .. controls (3.31,0.3) and (6.95,1.4) .. (10.93,3.29)   ;
%Straight Lines [id:da5161616293719913] 
\draw    (114,35.5) -- (114,110.5) ;
\draw [shift={(114,79)}, rotate = 270] [color={rgb, 255:red, 0; green, 0; blue, 0 }  ][line width=0.75]    (10.93,-3.29) .. controls (6.95,-1.4) and (3.31,-0.3) .. (0,0) .. controls (3.31,0.3) and (6.95,1.4) .. (10.93,3.29)   ;
%Straight Lines [id:da0455298820969281] 
\draw    (114,110.5) -- (34,110.5) ;
\draw [shift={(68,110.5)}, rotate = 360] [color={rgb, 255:red, 0; green, 0; blue, 0 }  ][line width=0.75]    (10.93,-3.29) .. controls (6.95,-1.4) and (3.31,-0.3) .. (0,0) .. controls (3.31,0.3) and (6.95,1.4) .. (10.93,3.29)   ;
%Straight Lines [id:da3390346760370848] 
\draw    (34,110.5) -- (34,35.5) ;
\draw [shift={(34,67)}, rotate = 90] [color={rgb, 255:red, 0; green, 0; blue, 0 }  ][line width=0.75]    (10.93,-3.29) .. controls (6.95,-1.4) and (3.31,-0.3) .. (0,0) .. controls (3.31,0.3) and (6.95,1.4) .. (10.93,3.29)   ;
%Curve Lines [id:da5190731031471307] 
\draw    (114,35.5) .. controls (86,13.5) and (59,13.5) .. (34,35.5) ;
\draw [shift={(80.54,19.6)}, rotate = 185.37] [color={rgb, 255:red, 0; green, 0; blue, 0 }  ][line width=0.75]    (10.93,-3.29) .. controls (6.95,-1.4) and (3.31,-0.3) .. (0,0) .. controls (3.31,0.3) and (6.95,1.4) .. (10.93,3.29)   ;
%Curve Lines [id:da2705972032296371] 
\draw    (114,110.5) .. controls (138,86.5) and (134,55.5) .. (114,35.5) ;
\draw [shift={(130.07,79.89)}, rotate = 274.35] [color={rgb, 255:red, 0; green, 0; blue, 0 }  ][line width=0.75]    (10.93,-3.29) .. controls (6.95,-1.4) and (3.31,-0.3) .. (0,0) .. controls (3.31,0.3) and (6.95,1.4) .. (10.93,3.29)   ;
%Curve Lines [id:da40272072987862484] 
\draw    (34,110.5) .. controls (53,130.5) and (94,133.5) .. (114,110.5) ;
\draw [shift={(67.29,126.15)}, rotate = 4.45] [color={rgb, 255:red, 0; green, 0; blue, 0 }  ][line width=0.75]    (10.93,-3.29) .. controls (6.95,-1.4) and (3.31,-0.3) .. (0,0) .. controls (3.31,0.3) and (6.95,1.4) .. (10.93,3.29)   ;

% Text Node
\draw (186.74,4.4) node [anchor=north west][inner sep=0.75pt]  [font=\normalsize]  {$\gamma :\begin{cases}
a\mapsto b & \\
b\mapsto c & \\
c\mapsto d & \\
d\mapsto e & \\
e\mapsto f & \\
f\mapsto g & \\
g\mapsto \overline{c}\overline{b}\overline{a} & 
\end{cases}$};
% Text Node
\draw (70,41.4) node [anchor=north west][inner sep=0.75pt]    {$a$};
% Text Node
\draw (97,64.4) node [anchor=north west][inner sep=0.75pt]    {$b$};
% Text Node
\draw (71,90.4) node [anchor=north west][inner sep=0.75pt]    {$c$};
% Text Node
\draw (40,62.4) node [anchor=north west][inner sep=0.75pt]    {$d$};
% Text Node
\draw (70,-1.6) node [anchor=north west][inner sep=0.75pt]    {$e$};
% Text Node
\draw (134,65.4) node [anchor=north west][inner sep=0.75pt]    {$f$};
% Text Node
\draw (69,132.4) node [anchor=north west][inner sep=0.75pt]    {$g$};

\end{tikzpicture}
 \end{center}
This represents the irreducible outer automorphism, $\varphi: w \mapsto x \mapsto y \mapsto z \mapsto zw^{-1}$, which has stretch factor $\lambda_{\gamma}$ equal to the largest root of $x^7-x^2-x-1$, approximately $\lambda_{\gamma} \approx 1.203$. By the proof of Theorem A in \cite{wiggd1}, every irreducible $\varphi \in $ Out$(F_4)$ has an i.t.t. representative on a graph with at most $3(4)-4=8$ edges. Since
$$\lambda_{\gamma} < 3^{\frac{1}{5}} < 4^{\frac{1}{7}} < 5^{\frac{1}{8}},$$
Theorem \hyperref[lower bound]{A} implies any irreducible $\varphi \in $ Out$(F_4)$ with stretch factor less than $\lambda_{\gamma}$ must have an i.t.t. representative which is either 2 folds on a graph with 6, 7 or 8 edges or 3 folds on a graph with 8 edges. 
\end{ex}
%\bigskip
%\section{Smallest rank 4 and 5 stretch factor?}
%via some explicit computations and using the lower bound in terms of number of folds / edges.

%\section{Further Questions}

%\begin{enumerate}
%\item I think i can use this lower bound  + computer to manually prove the minimal stretch in rank 4 and maybe 5. upcoming work.
%\item how does number of stacks increase under powers?
%\item what if f is not periodic on the vertices
%\item use stack graph to get sharper lower bounds for certain cases. 
%\item use stack graph to get a better understanding of the dynamics of $f$.
%\item study minimal stretch factor possible on a fixed graph, for other graphs or other families of graphs. Or study the whole set of possible stretch factors on a fixed graph. Or family of graphs.
%\item how does stack score change as you do different types of folds. For proper full folds it changes by at most one per fold. 
%\item How are all the stack score = c for a fixed integer c "distributed" throughout outer space / moduli space. As rank increases, what happens to the "average" stack score. What is a good notion of average here? 
%\item upper bounds on stretch factor using the stack graph? 
%\end{enumerate}
\bigskip
%\newpage

\bibliographystyle{alpha}

\bibliography{PaperRefs}

\end{document}